\documentclass{article}
\usepackage[utf8]{inputenc}
\usepackage[T1]{fontenc}
\usepackage{amsmath,amssymb}
\usepackage{graphicx}
\usepackage{amsthm, enumerate}
\theoremstyle{remark}
\usepackage[colorlinks=true]{hyperref}
\usepackage{latexsym}
\usepackage{xypic}
\usepackage{pxfonts}
\usepackage{mathrsfs}
\usepackage{fullpage}
\usepackage{wasysym}
\usepackage{enumitem}
\usepackage[affil-it]{authblk}
\usepackage[hang,small,center]{caption}
\usepackage{stmaryrd}

\theoremstyle{definition}
\newtheorem{defi}{Definition}[section]
\newtheorem{const}[defi]{Construction}
\newtheorem{expl}[defi]{Example}

\newtheorem{rmk}[defi]{Remark}
\newtheorem*{conv}{Convention}
\newtheorem*{plan}{Organization of the paper}
\newtheorem*{merci}{Acknowlegments}
\newtheorem{notat}[defi]{Notation}

\theoremstyle{plain}
\newtheorem{pro}[defi]{Proposition}
\newtheorem{thm}[defi]{Theorem}
\newtheorem{cor}[defi]{Corollary}

\newtheorem{lmm}[defi]{Lemma}
\newtheorem*{lmmA}{Lemma \setword{A}{refA}}
\newtheorem*{lmmB}{Lemma \setword{B}{refB}}
\newtheorem*{Ccl}{Conclusion}
\newtheorem*{thm2}{Theorem}

\theoremstyle{remark}

\makeatletter
\newcommand{\setword}[2]{
\phantomsection
#1\def\@currentlabel{\unexpanded{#1}}\label{#2}
}
\makeatother

\title{Delooping derived mapping spaces of bimodules over an operad}
\author{Julien Ducoulombier}
\date{}

\begin{document}
\maketitle 

\abstract \noindent 
From a map of operads $\eta:O\rightarrow O'$, we introduce a cofibrant replacement of $O$ in the category of bimodules over itself such that the corresponding model of the derived mapping space of bimodules $Bimod_{O}^{h}(O\,;\,O')$ is an algebra over the one dimensional little cubes operad $\mathcal{C}_{1}$. In the present work, we also build an explicit weak equivalence of $\mathcal{C}_{1}$-algebras from the loop space $\Omega Operad^{h}(O\,;\,O')$ to $Bimod_{O}^{h}(O\,;\,O')$.

\section*{Introduction}

The little cubes operad $\mathcal{C}_{d}$ has been introduced by May, Boardman and Vogt in order to model iterated loop spaces. Together with Stasheff \cite{Stasheff63}, see also \cite{Boardman68,Boardman73,May72,May74}, they prove the recognition principle asserting that if a space $X$ is a grouplike $\mathcal{C}_{d}$-algebra, then there exists a space $Y$ such that $X\simeq \Omega^{d}Y$. In particular, such a space is endowed with a product more or less commutative up to homotopy depending on the parameter $d$. For $d=1$, the product is only associative up to homotopy while, for $d=2$, the product is commutative up to homotopy but we don't have necessarily  "homotopies between homotopies". In other words, higher is the parameter $d$, better is the commutativity up to homotopy. One of the most important examples is the space of long embeddings compactly supported modulo immersions
\begin{equation}\label{F2}
\overline{Emb}_{c}(\mathbb{R}^{d}\,;\,\mathbb{R}^{n}):=hofib\big(\, Emb_{c}(\mathbb{R}^{d}\,;\,\mathbb{R}^{n})\longrightarrow Imm_{c}(\mathbb{R}^{d}\,;\,\mathbb{R}^{n})\,\big).
\end{equation}
For $d=1$, this space, also called the space of long knots, has been intensively study by many authors and proved to be endowed with an action of the two dimensional little cubes operad $\mathcal{C}_{2}$ by Budney \cite{Budney07}. Roughly speaking, the commutative product up to homotopy is given by the concatenation of knots while the commutativity, illustrated in Figure \ref{F1}, consists in shrinking one of the knots and move it along the line through the other one. Similarly, it is natural to expect that the space of long embeddings in dimension $d$ is equipped with an action of $\mathcal{C}_{d+1}$. 
\begin{figure}[!h]
\begin{center}
\includegraphics[scale=0.25]{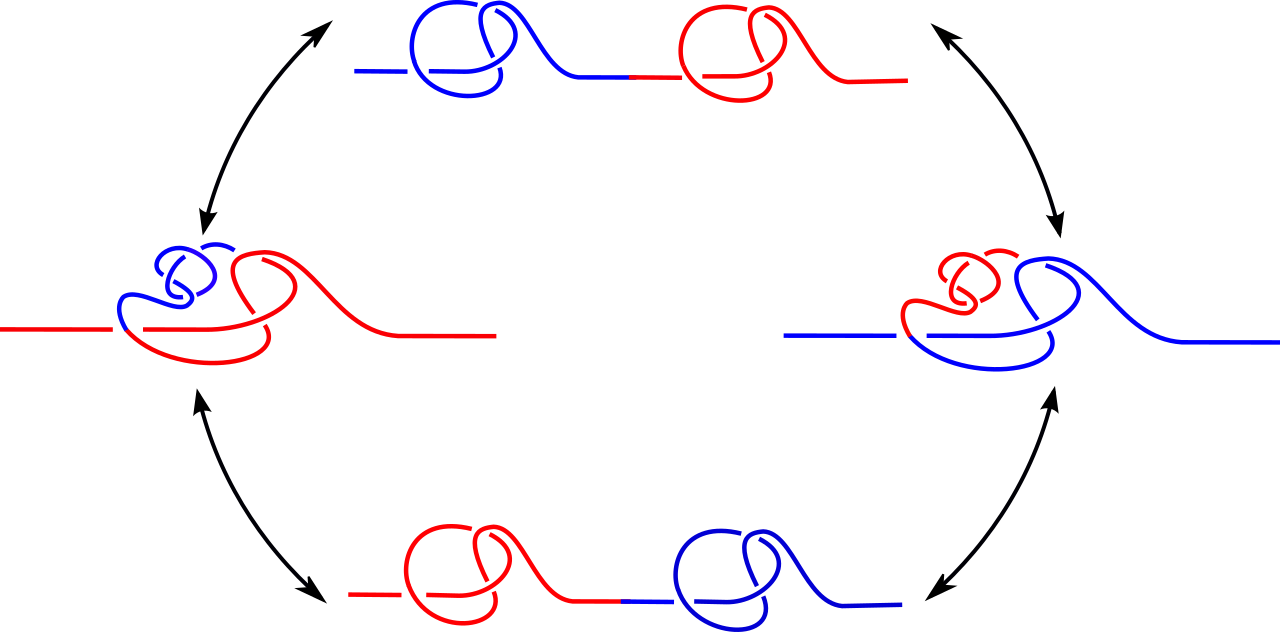}
\caption{The commutative product up to homotopy for the space of long knots.}\label{F1}\vspace{-30pt}
\end{center}
\end{figure}

\hspace{-250cm}\footnote{ETH Zurich, Ramistrasse 101, 809 \mbox{Zurich}, Switzerland}
\footnote{Key words: operads, loop spaces, space of knots, model category}
\footnote{julien.ducoulombier@math.ethz.ch}
\footnote{Homepage: http://ducoulombier-math.esy.es/}
\footnote{The author was supported by the grant ERC-2015-StG 678156 GRAPHCPX}

\newpage

However, it requires a significant amount of work to identify the iterated loop space associated to an algebra over the little cubes operad. For instance, McClure and Smith \cite{McClure04.2,McClure04} give combinatorial conditions on a cosimplicial space such that its homotopy totalization has the homotopy type of a $d$-iterated loop space but we only have a description of the loop space for $d\in\{1\,;\,2\}$ (see \cite{Dwyer12,Tourtchine10}). In the present work, we  use the category of bimodules over an operad $O$, denoted by $Bimod_{O}$, which is stronger than the notion of operad in the sense that any operad is a bimodule over itself. The main result of this paper is the following delooping theorem:

\begin{thm2}[Theorem \ref{B5}]
Let $O$ be a well pointed $\Sigma$-cofibrant operad and let $\eta:O\rightarrow O'$ be a map of operads. If the space $O'(1)$ is contractible, then one has the weak equivalence
\begin{equation}\label{F7}
Bimod^{h}_{O}(O\,;\,O')\simeq \Omega Operad^{h}(O\,;\,O').
\end{equation}
\end{thm2}

All the categories considered are model categories and the symbol $"h"$ above means that we refer to the derived mapping space. This theorem is a generalization of results obtained by Dwyer-Hess \cite{Dwyer12} and independently by Turchin \cite{Tourtchine10} in the context of non-symmetric operads and where the source object is the associative operad $\mathcal{A}s$. Let us mention that the method used to prove the above theorem is also true in the context of coloured operads. However, we focus on the uncoloured case in order to simplify the constructions and the notation.

Furthermore, we also produce a truncated version of the delooping for bimodules. Roughly speaking, $T_{k}Bimod_{O}$ and $T_{k}Operad$ are the restrictions to the operations with at most $k$ inputs (see Section \ref{G1}). The restriction functors $T_{k}Bimod_{O}\rightarrow T_{k-1}Bimod_{O}$ and $T_{k}Operad\rightarrow T_{k-1}Operad$ give rise to towers of fibrations which play an important role in understanding the manifold calculus tower associated to the space of long embeddings. We refer the reader to \cite{Arone14, Weiss99.2, Weiss99} for more details about the Goodwillie-Weiss calculus theory and its connection with the towers of fibrations. Under the conditions of Theorem \ref{B5}, we build an explicit weak equivalence between the following truncated spaces:
$$
T_{k}Bimod^{h}_{O}(T_{k}(O)\,;\,T_{k}(O'))\simeq \Omega \big(\, T_{k}Operad^{h}(T_{k}(O)\,;\,T_{k}(O'))\,\big).
$$

On the other hand, one has a little bit more information than just a weak equivalence. By using an explicit cofibrant replacement of the operad $O$ in the category $Bimod_{O}$, we are able to define a $\mathcal{C}_{1}$-algebra structure on the derived mapping space of "truncated" bimodules. We also build an explicit weak equivalence from the loop space which is also a map of $\mathcal{C}_{1}$-algebras. So, the $\mathcal{C}_{1}$-algebra structures are preserved by the weak equivalence. This kind of property is important in order to check that the structure are preserved along "zig-zag" of weak equivalences. For instance, it is not obvious that the structure introduced by Budney on the space of knots, described at the beginning, coincides with the structure of the double loop space through the identifications (\ref{F4}).\vspace{5pt}

\noindent\textbf{Applications to the Dwyer-Hess' conjecture.} To introduce the   Dwyer-Hess' conjecture, we need the notion of infinitesimal bimodule over an operad $O$, denoted by $Ibimod_{O}$, which is stronger than the notion of bimodule in the sense that any bimodule with a based point in arity $1$ is also an infinitesimal bimodule. So, the conjecture asserts that if $\eta:\mathcal{C}_{d}\rightarrow M$ is a map of bimodules over $\mathcal{C}_{d}$, then the following weak equivalences hold under the assumption $M(0)\simeq \ast$:
$$
\begin{array}{rcl}\vspace{2pt}
Ibimod^{h}_{\mathcal{C}_{d}}(\mathcal{C}_{d}\,;\,M) & \simeq  & \Omega^{d} Bimod^{h}_{\mathcal{C}_{d}}(\mathcal{C}_{d}\,;\,M),  \\ 
T_{k}Ibimod^{h}_{\mathcal{C}_{d}}(T_{k}(\mathcal{C}_{d})\,;\,T_{k}(M)) & \simeq  & \Omega^{d}\big(\, T_{k}Bimod^{h}_{\mathcal{C}_{d}}(T_{k}(\mathcal{C}_{d})\,;\,T_{k}(M))\,\big). 
\end{array} 
$$ 
This statement is proved by Boavida de Brito and Weiss \cite{Weiss15} in the special case $M=\mathcal{C}_{n}$ while the general case is proved by Turchin and the author \cite{Ducoulombier16.2} using combinatorial methods. Together with Theorem \ref{B5}, we are able to identify explicit iterated loop spaces from maps of operads. More precisely, if 
 $\eta:\mathcal{C}_{d}\rightarrow O$ is a map of operads with $O(0)\simeq O(1)\simeq \ast$, then the following weak equivalences hold:
\begin{equation}\label{F3}
\begin{array}{rcl}
Ibimod^{h}_{\mathcal{C}_{d}}(\mathcal{C}_{d}\,;\,O) & \simeq  & \Omega^{d+1} Operad^{h}(\mathcal{C}_{d}\,;\,O),  \\ 
T_{k}Ibimod^{h}_{\mathcal{C}_{d}}(T_{k}(\mathcal{C}_{d})\,;\,T_{k}(O)) & \simeq  & \Omega^{d+1}\big(\, T_{k}Operad^{h}(T_{k}(\mathcal{C}_{d})\,;\,T_{k}(O))\,\big). 
\end{array} 
\end{equation}
As an application, Arone and Turchin \cite{Arone14} develop a machinery in order to identify spaces of embeddings between two smooth manifolds with derived mapping spaces of infinitesimal bimodules. In particular, for $n-d-2>0$, the authors in \cite{Arone14} and simultaneously Turchin in \cite{Turchin13} prove that the space of long embeddings (\ref{F2}) is weakly equivalent to  $Ibimod^{h}_{\mathcal{C}_{d}}(\mathcal{C}_{d}\,;\,\mathcal{C}_{n})$. In particular, if we apply the identifications (\ref{F3}) to the map $\eta:\mathcal{C}_{d}\rightarrow \mathcal{C}_{n}$, then we get an explicit description of the iterated loop space associated to the space of long embeddings and their polynomial approximations:
\begin{equation}\label{F4}
\begin{array}{rcl}\vspace{4pt}
\overline{Emb}_{c}(\mathbb{R}^{d}\,;\,\mathbb{R}^{n}) & \simeq & \Omega^{d+1} Operad^{h}(\mathcal{C}_{d}\,;\,O), \\ 
T_{k}\overline{Emb}_{c}(\mathbb{R}^{d}\,;\,\mathbb{R}^{n}) & \simeq & \Omega^{d+1}\big(\, T_{k}Operad^{h}(T_{k}(\mathcal{C}_{d})\,;\,T_{k}(O))\,\big).
\end{array}     
\end{equation}

\vspace{5pt}

\noindent\textbf{Applications to the Swiss-Cheese operad.}
Theorem \ref{B5} is also used by the author \cite{Ducoulombier16} in order to extend the previous results to the coloured case using the Swiss-Cheese operad $\mathcal{SC}_{d}$ which is a relative version of the little cubes operad $\mathcal{C}_{d}$. In that case, a typical example of $\mathcal{SC}_{d}$-algebra is a pair of topological spaces (since the operad has two colours $S=\{o
,;\,c\}$) of the form
$$
\big( \Omega^{d}X\,\,;\,\, \Omega^{d}(X\,;\,Y):=\Omega^{d-1}(hofib(f:Y\rightarrow X))\,\big),
$$
where $f:Y\rightarrow X$ is a continuous map between pointed spaces. In particular, if $(A\,;\,B)$ is an $\mathcal{SC}_{d}$-algebra, then $A$ is a $\mathcal{C}_{d}$-algebra, $B$ is a $\mathcal{C}_{d-1}$-algebra and there is a map $\tau:A\rightarrow B$ which is more or less central up to homotopy (i.e. $\tau$ preserves the product and $\tau(a)\times b\simeq b\times \tau(a)$ with $a\in A$ and $b\in B$). In \cite{Ducoulombier16}, we give a relative version of the delooping (\ref{F7}) such that, together with Theorem \ref{B5}, we are able to identify explicit $\mathcal{SC}_{d+1}$-algebras. In particular, if $\eta_{1}:\mathcal{C}_{d}\rightarrow O$ is a map of operads and $\eta_{2}:O\rightarrow M$ is a map of bimodules over $O$, then, under technical conditions, the pair of spaces 
$$
(Ibimod^{h}_{\mathcal{C}_{d}}(\mathcal{C}_{d}\,;\,O)\,\,;\,\, Ibimod^{h}_{\mathcal{C}_{d}}(\mathcal{C}_{d}\,;\,M))
$$ 
is proved to be weakly equivalent to a typical $\mathcal{SC}_{d+1}$-algebra using the identifications
\begin{equation}\label{F8}
\begin{array}{rcl}\vspace{4pt}
Ibimod^{h}_{\mathcal{C}_{d}}(\mathcal{C}_{d}\,;\,O) & \simeq & \Omega^{d+1}Operad^{h}(\mathcal{C}_{d}\,;\,O), \\ 
Ibimod^{h}_{\mathcal{C}_{d}}(\mathcal{C}_{d}\,;\,M) & \simeq & \Omega^{d+1}\big( \,Operad^{h}(\mathcal{C}_{d}\,;\,O) \,\,;\,\, Op[\mathcal{C}_{d}\,;\,\emptyset]^{h}(\mathcal{CC}_{d}\,;\,\mathcal{L}(M)) \big), 
\end{array} 
\end{equation}
where $\mathcal{CC}_{d}$ and $\mathcal{L}(M)$ are two coloured operads described in \cite{Ducoulombier16}.

As an application, one can consider the space $(l)\text{-}Imm_{c}(\mathbb{R}^{d}\,;\,\mathbb{R}^{n})$ of $(l)$-immersions compactly supported which is the subspace of immersions $f$ such that for each subset of $l$ distinct elements  $K\subset \mathbb{R}^{d}$, the restriction $f_{|K}$ is non constant. The space of long $(l)$-immersions is defined by the following homotopy fiber:
$$
\overline{Im}m_{c}^{(l)}(\mathbb{R}^{d}\,;\,\mathbb{R}^{n}):=hofib\big(\, (l)\text{-}Imm_{c}(\mathbb{R}^{d}\,;\,\mathbb{R}^{n})\longrightarrow Imm_{c}(\mathbb{R}^{d}\,;\,\mathbb{R}^{n})\,\big).
$$
In the case $d=1$, we can easily observe that the concatenation produces a product which is only associative up to homotopy due to the condition on the cardinality of the preimage of each point. Furthermore, there is an inclusion from the space of long knots compatible with the concatenation and which is central up to homotopy as shown in the picture below. In Figure \ref{F6}, the $(3)$-immersion is represented in blue. Since the knot represented in red is injective, we can shrink the $(3)$-immersion and move it along the line through the knot.
\begin{figure}[!h]
\begin{center}
\includegraphics[scale=0.3]{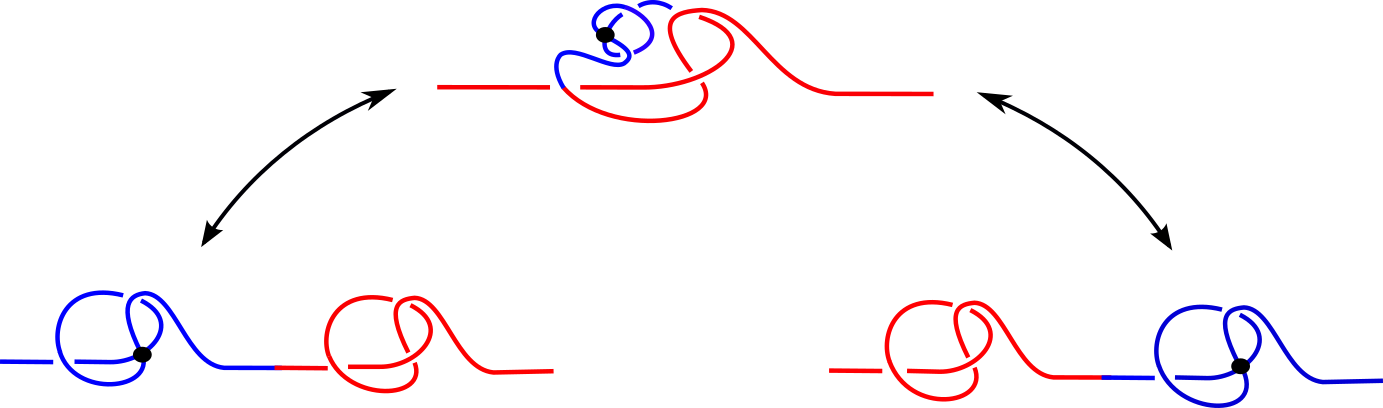}
\caption{Illustration that the inclusion is central up to homotopy.}\label{F6}\vspace{-15pt}
\end{center}
\end{figure}

\noindent It seems natural to expect that the pair of topological spaces formed by the spaces of long embeddings and the space of long $(l)$-immersions from $\mathbb{R}^{d}$ to $\mathbb{R}^{n}$ gives rise to an algebra over the Swiss-Cheese operad $\mathcal{SC}_{d+1}$. 

\noindent Similarly to the space of long embeddings, Dobrinskaya and Turchin \cite{Turchin14} prove that the space of long $(l)$-immersions can be expressed in terms of derived mapping space of infinitesimal bimodules using the non-$(l)$-overlapping little cubes bimodule $\mathcal{C}_{n}^{(l)}$ (see Example \ref{G3}). To be more precise, we only have the following weak equivalence since it is not proved that the polynomial approximation of the space of long $(l)$-immersions converges:
$$
T_{\infty}\overline{Im}m_{c}^{(l)}(\mathbb{R}^{d}\,;\,\mathbb{R}^{n})\simeq Ibimod_{\mathcal{C}_{d}}^{h}(\mathcal{C}_{d}\,;\,\mathcal{C}_{d}^{(l)}).
$$
By definition, there is an inclusion from the little cubes $\mathcal{C}_{d}$ to $\mathcal{C}_{d}^{(l)}$ compatible with the bimodule structures. Consequently, if we apply the identifications (\ref{F8}) to the maps $\eta_{1}:\mathcal{C}_{d}\rightarrow \mathcal{C}_{n}$ and $\eta_{2}:\mathcal{C}_{n}\rightarrow \mathcal{C}_{n}^{(l)}$, then the pair of topological spaces
\begin{equation*}
(\overline{Emb}_{c}(\mathbb{R}^{d},;\,\mathbb{R}^{n})\,;\,T_{\infty}\overline{Im}m_{c}^{(l)}(\mathbb{R}^{d}\,;\,\mathbb{R}^{n}))
\end{equation*}
is proved to be weakly equivalent to an explicit $\mathcal{SC}_{d+1}$-algebra.

\begin{plan}
The paper is divided into $3$ sections. The \textit{first section} introduces the notion of "truncated" operads with the example of the little cubes operad $\mathcal{C}_{d}$.  We also recall the Boardman-Vogt resolution $\mathcal{BV}(O)$ of an operad $O$ producing a functorial way to get cofibrant replacements in the category of operads. This construction is endowed with a filtration $\{\mathcal{BV}_{k}(O)[l]\}_{k\,;\,l}$ which gives rise to a tower of fibrations computing the derived mapping space $\Omega Operad^{h}(O\,;\,O')$:
$$
\xymatrix@C=15pt{
\Omega Operad(\mathcal{BV}_{1}(O)[1]\,;\,O') & \cdots \ar[l] & \Omega Operad(\mathcal{BV}_{k}(O)[l-1]\,;\,O') \ar[l] & \Omega Operad(\mathcal{BV}_{k}(O)[l]\,;\,O')\ar[l] & \cdots \ar[l]
}
$$

The \textit{second section} is devoted to the notion of bimodule over an operad $O$ with the example of the non-$(l)$-overlapping little cubes bimodule and the example of the free bimodule functor. We also introduce a cofibrant replacement $\mathcal{WB}(O)$ of the operad $O$ in the category bimodules over itself. By using the properties of this construction, we are able to define a $\mathcal{C}_{1}$-algebra structure on the mapping space $Bimod_{O}(\mathcal{WB}(O)\,;\,O')$. Furthermore, this cofibrant replacement is endowed with a filtration $\{\mathcal{WB}_{k}(O)[l]\}_{k\,;\,l}$ which gives rise to a tower of fibrations computing the derived mapping space $Bimod_{O}^{h}(O\,;\,O')$:
$$
\xymatrix@C=18pt{
Bimod_{O}(\mathcal{WB}_{1}(O)[1]\,;\,O') & \cdots \ar[l] & Bimod_{O}(\mathcal{WB}_{k}(O)[l-1]\,;\,O') \ar[l] & Bimod_{O}(\mathcal{WB}_{k}(O)[l]\,;\,O')\ar[l] & \cdots \ar[l]
}
$$

The \textit{last section} is devoted to the proof of the main theorem of the paper. First, we define the map (\ref{F7}) using the properties of the cofibrant replacement $\mathcal{WB}(O)$. The map so obtained preserves the $\mathcal{C}_{1}$-algebra structures and is compatible with the above towers of fibrations in the sense that it produces of morphism $\{\xi_{k}[l]\}$ between the two towers. So, we prove Theorem \ref{B5} by induction on the parameters $k$ and $l$. In particular, we show that the map $\{\xi_{k}[l]\}$ is a weak equivalence if an inclusion of sequences is a homotopy equivalence. Then we express the sequences in terms of functors from Reedy categories and we prove the main result using the homotopy theory of diagrams.  
\end{plan}

\begin{conv}
By a space we mean a compactly generated Hausdorff space and by abuse of notation we denote by $Top$ this category (see e.g. \cite[section 2.4]{Hovey99}). If $X$, $Y$ and $Z$ are spaces, then $Top(X;Y)$ is equipped with the compact-open topology in order to have a homeomorphism $Top(X;Top(Y;Z))\cong Top(X\times Y;Z)$. 

By using the Serre fibrations, the category $Top$ is endowed with a cofibrantly generated monoidal model  structure. In the paper the categories considered are enriched over $Top$. By convention, if $\mathcal{C}$ is a model category enriched over $Top$, then the derived mapping space $\mathcal{C}^{h}(A;B)$ is the space $\mathcal{C}(A^{c};B^{f})$ with $A^{c}$ a cofibrant replacement of $A$ and $B^{f}$ a fibrant replacement of $B$.

The category of spaces equipped with a right action of a group $G$, denoted by $G\text{-}Top$, has a model category structure coming from the adjunction $G[-]:Top\leftrightarrows G\text{-}Top:\mathcal{U}$ where $G[-]$ is the free functor sending a space $X$ to $G[X]=\coprod_{G}X$.  By convention a map in $G\text{-}Top$ is called a $G$-equivariant map whereas a cofibration in $G\text{-}Top$ is called a $G$-cofibration. In what follows, we use the following two statements which are particular cases of \cite[Lemma 2.5.3]{Berger06} and \cite[Lemma 2.5.2]{Berger06} respectively.

\begin{lmmA}\label{G6}
Let $1\rightarrow G_{1}\rightarrow G_{1}\rtimes G_{2}\rightarrow G_{2}\rightarrow 1$ be a short exact sequence of groups. Let $A\rightarrow B$ be a $G_{2}$-cofibration and $X\rightarrow Y$ be a $G_{1}\rtimes G_{2}$-equivariant $G_{1}$-cofibration. Then, the pushout product $(A\times Y)\cup_{A\times X}(B\times X)\rightarrow B\times Y$ is a $G_{1}\rtimes G_{2}$-cofibration.
\end{lmmA}

\begin{lmmB}
Let $G$ be a group. Let $A\rightarrow B$ and $X\rightarrow Y$ be two $G$-equivariant maps which are cofibration in the category of topological spaces. If one of them is a $G$-cofibration, then the pushout product $(A\times Y)\cup_{A\times X}(B\times X)\rightarrow B\times Y$ is a $G$-cofibration. Moreover the latter is acyclic if $A\rightarrow B$ or $X\rightarrow Y$ is.\vspace{-15pt}
\end{lmmB}

\end{conv}

\newpage 

\section{Topological operads and the Boardman-Vogt resolution}

In what follows, we cover the notion of "truncated" operad with the example of the little cubes operad. For more details about these objects, we refer the reader to \cite{May72}. We also recall the Boardman-Vogt resolution which produces cofibrant replacements in the category of operads. Introduced by Boardman and Vogt \cite{Boardman68,Boardman73} for topological operad, this resolution has been extended to operad in any model category equipped with a notion of interval by Berger and Moerdijk \cite{Berger06}. This construction is endowed with a filtration used to define a tower of fibrations associated to the derived mapping space $\Omega Operad^{h}(O\,;\,O')$.

\subsection{Topological operad and little cubes operad}\label{G1}

In what follows, we recall the terminology related to the notion of topological operad. By a \textit{sequence} we mean a family of topological spaces $M:=\{M(n)\}$, with $n\in \mathbb{N}$, endowed with a right action of the symmetric group: for each permutation $\sigma \in \Sigma_{n}$, there is a continuous map
\begin{equation}\label{A0}
\begin{array}{rcl}
\sigma^{\ast}:M(n) & \longrightarrow & M(n); \\ 
 x & \longmapsto & x\cdot \sigma,
\end{array} 
\end{equation}
satisfying the relation $(x\cdot\sigma)\cdot \tau=x\cdot(\sigma\tau)$ with $\tau\in \Sigma_{n}$. A map between two sequences is given by a family of continuous maps compatible with the action of the symmetric group. In the rest of the paper, we denote by $Seq$ the category of sequences. 

Given an integer $k\geq 1$, we also consider the category of $k$-truncated sequences $T_{k}Seq$. The objects are family of topological spaces $M:=\{M(n)\}$, with $0\leq n\leq k$, endowed an action of the symmetric group (\ref{A0}) for $n\leq k$. A "truncated" sequence is said to be \textit{pointed} if there is a distinguished element $\ast_{1}\in M(1)$ called \textit{unit}. We denote by $Seq_{\ast}$ and $T_{k}Seq_{\ast}$ the category of pointed sequences and pointed $k$-truncated sequences respectively. One has an obvious functor  
$$
T_{k}(-):Seq \longrightarrow T_{k}Seq.
$$

The category of "truncated" sequences is endowed with a cofibrantly generated model category structure in which a map is a weak equivalence (resp. a fibration) if each of its components is a weak equivalence (resp. a Serre fibration).  Furthermore, the objects in this model category are fibrant \cite{Hirschhorn03}. Similarly, the category of pointed "truncated" sequences inherits a model category structure with the same properties. A pointed "truncated" sequence is said to be \textit{well pointed} if the inclusion from the unit to $M(1)$ is a cofibration in the category of spaces. 

\begin{defi}
An \textit{operad} is a pointed sequence $O$ together with operations called \textit{operadic compositions}
\begin{equation}\label{A1}
\circ_{i}:O(n)\times O(m)\longrightarrow O(n+m-1),\hspace{15pt} \text{with }1\leq i\leq n, 
\end{equation}
satisfying compatibility with the action of the symmetric group, associativity, commutativity and unit axioms. A map between two operads should respect the operadic compositions.  We denote by $Operad$ the categories of operads. An algebra over the operad $O$, or \textit{$O$-algebra}, is given by a topological space $X$ endowed with operations
$$
\alpha_{n}:O(n)\times X^{\times n} \longrightarrow X,
$$ 
compatible with the operadic compositions and the action of the symmetric group.

Given an integer $k\geq 1$, we also consider the category of $k$-truncated operads $T_{k}Operad$. The objects are pointed $k$-truncated sequences endowed with operadic compositions (\ref{A1}) for $n+m-1\leq k$ and $n\leq k$. One has an obvious functor
$$
T_{k}(-):Operad \longrightarrow T_{k}Operad.
$$

The category of "truncated" operads inherits a cofibrantly generated model category structure from the category of pointed "truncated" sequences using the adjunction $\mathcal{F}:Seq_{\ast}\leftrightarrows Operad:\mathcal{U}$ where $\mathcal{U}$ is the forgetful functor while $\mathcal{F}$ is the free operad functor \cite{Berger03}. More precisely, a map of "truncated" operads $f$ is a weak equivalence (resp. a fibration) if the corresponding map $\mathcal{U}(f)$ is a weak equivalence (resp. a fibration) in the category of pointed "truncated" sequences. In particular all the "truncated" operads are fibrant. According to the convention, $f$ is said to be $\Sigma$-cofibrant if $\mathcal{U}(f)$ is cofibrant in the category of sequences. 
\end{defi}

\newpage

\begin{expl}\textbf{The overlapping little cubes operad $\mathcal{C}_{d}^{\infty}$}

\noindent A $d$-dimensional little cube is a continuous map $c:[0\,,\,1]^{d}\rightarrow [0\,,\,1]^{d}$ arising from an affine embedding preserving the direction of the axes. The operad $\mathcal{C}_{d}^{\infty}$ is the sequence $\{\mathcal{C}_{d}^{\infty}(n)\}$ whose $n$-th component is given by $n$ little cubes, that is, $n$-tuples $<c_{1},\ldots,c_{n}>$ with $c_{i}$ a $d$-dimensional little cube. The unit point in $\mathcal{C}_{d}^{\infty}(1)$ is the identity little cube $id:[0\,,\,1]^{d}\rightarrow [0\,,\,1]^{d}$ whereas $\sigma\in \Sigma_{n}$ permutes the indexation as follows:
$$
\sigma^{\ast}:\mathcal{C}_{d}^{\infty}(n)\longrightarrow \mathcal{C}_{d}^{\infty}(n)\,\,\,;\,\,\, <c_{1},\ldots,c_{n}>  \longmapsto  <c_{\sigma(1)},\ldots,c_{\sigma(n)}>.
$$ 
The operadic compositions are given by the formula
$$
\begin{array}{clcl}
\circ_{i}: & \mathcal{C}_{d}^{\infty}(n)\times \mathcal{C}_{d}^{\infty}(m) & \longrightarrow & \mathcal{C}_{d}^{\infty}(n+m-1); \\ 
 &  <c_{1},\ldots,c_{n}>\,;\, <c'_{1},\ldots,c'_{m}> & \longmapsto & <c_{1},\ldots,c_{i-1}, c_{i}\circ c'_{1},\ldots, c_{i}\circ c'_{m},c_{i+1},\ldots,c_{n}>.
\end{array} 
$$
By convention $\mathcal{C}_{d}^{\infty}(0)$ is the one point topological space and the operadic composition $\circ_{i}$ with this point consists in forgetting the $i$-th little cube.
\end{expl}

\begin{expl}\label{a8}\textbf{The little cubes operad $\mathcal{C}_{d}$}\\
\noindent The $d$-dimensional little cubes operad $\mathcal{C}_{d}$ is the 
sub-operad of $\mathcal{C}_{d}^{\infty}$ whose $n$-th component is the configuration space of $n$ little cubes with disjoint interiors. In other words, $\mathcal{C}_{d}(n)$ is the subspace of $\mathcal{C}_{d}^{\infty}(n)$  formed by configurations $<c_{1},\ldots,c_{n}>$ satisfying the relation
\begin{equation}\label{a1}
Int(Im(c_{i}))\cap Int(Im(c_{j}))=\emptyset, \hspace{15pt} \forall i\neq j.
\end{equation}
The operadic compositions and the action of the symmetric group arise from the operad $\mathcal{C}_{d}^{\infty}$. Furthermore, if  $(X\,;\,\ast)$ is a pointed topological space, then the $d$-iterated loop space $\Omega^{d}X$ is an example of $\mathcal{C}_{d}$-algebra.\vspace{-5pt}
\end{expl}

\begin{figure}[!h]
\begin{center}
\includegraphics[scale=0.27]{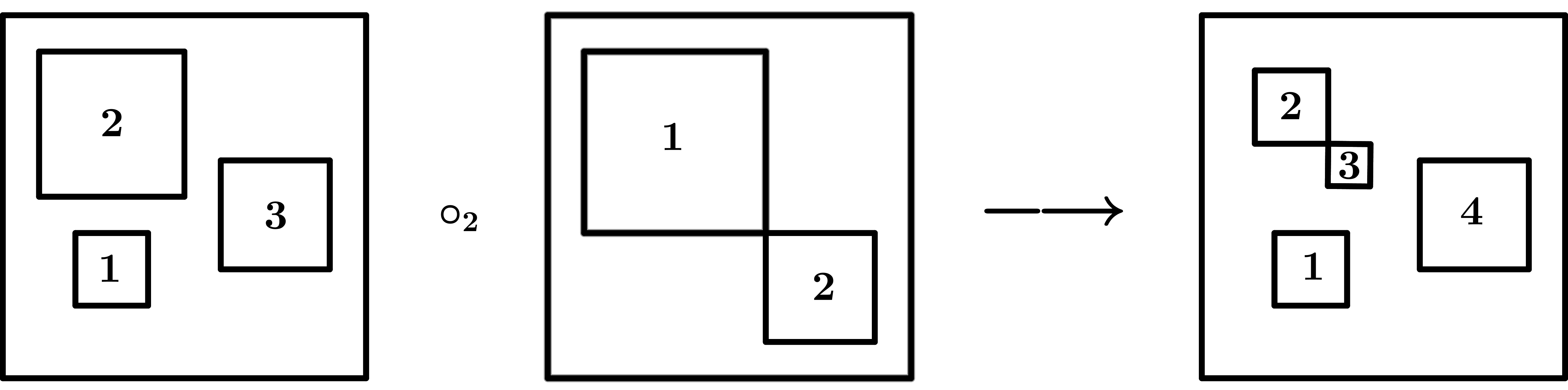}\vspace{-5pt}
\caption{The operadic composition $\circ_{2}:\mathcal{C}_{2}(3)\times \mathcal{C}_{2}(2)\rightarrow \mathcal{C}_{2}(4)$.}\vspace{-25pt}
\end{center}
\end{figure}

\subsection{The Boardman-Vogt resolution and cofibrant replacement}

As explained in the previous subsection, all the "truncated" operads are fibrant. Consequently, in order to compute derived mapping space of "truncated" operads, we only need a cofibrant replacement of the source object. The Boardman-Vogt resolution provides a functorial way to build such cofibrant replacements and has been intensively studied by Boardman-Vogt \cite{Boardman68,Boardman73} and Berger-Moerdijk \cite{Berger06}. This is a classical construction that requires the language of trees. For that reason, one has to fix some notation.

\begin{defi}\label{C9}
A \textit{planar tree} $T$ is a finite planar tree with one output edge on the bottom and inputs edges on the top. The vertex connected to the output edge, denoted by $r$, is called the \textit{root} of $T$. Such an element is endowed with an orientation from top to bottom. According to the orientation of the tree, if $e$ is an edge, then its vertex $t(e)$ towards the trunk is called \textit{the target vertex} whereas the other vertex $s(e)$ is called \textit{the source vertex}. By convention, the input and output edges are half-plan: the inputs edges do not have source vertices while the ouput edge do not have target vertex. Let $T$ be a planar tree:
\begin{itemize}[leftmargin=*]
\item[$\blacktriangleright$] The set of its vertices and the set of its edges are denoted by $V(T)$ and $E(T)$ respectively. The set of its internal edges $E^{int}(T)$ is formed by the edges connecting two vertices. Each edge is joined to the trunk by a unique path composed of edges.   
\item[$\blacktriangleright$] The input edges are called leaves and they are ordered from left to right. Let $in(T):=\{l_{1},\ldots, l_{|T|}\}$ denote the ordered set of leaves with $|T|$ the number of leaves. 
\item[$\blacktriangleright$]  The set of incoming edges of a vertex $v$ is ordered from left to right. This set is denoted by $in(v):=\{e_{1}(v),\ldots, e_{|v|}(v)\}$ with $|v|$ the number of incoming edges. The unique output edge of $v$ is denoted by $e_{0}(v)$.
\item[$\blacktriangleright$] The vertices with no incoming edge are called \textit{univalent vertices} (i.e. whose valence is $1$) while the vertices with only one input are called \textit{bivalent vertices} (i.e. whose valence is $2$).\vspace{-10pt}
\end{itemize}

\newpage

\begin{figure}[!h]
\begin{center}
\includegraphics[scale=0.25]{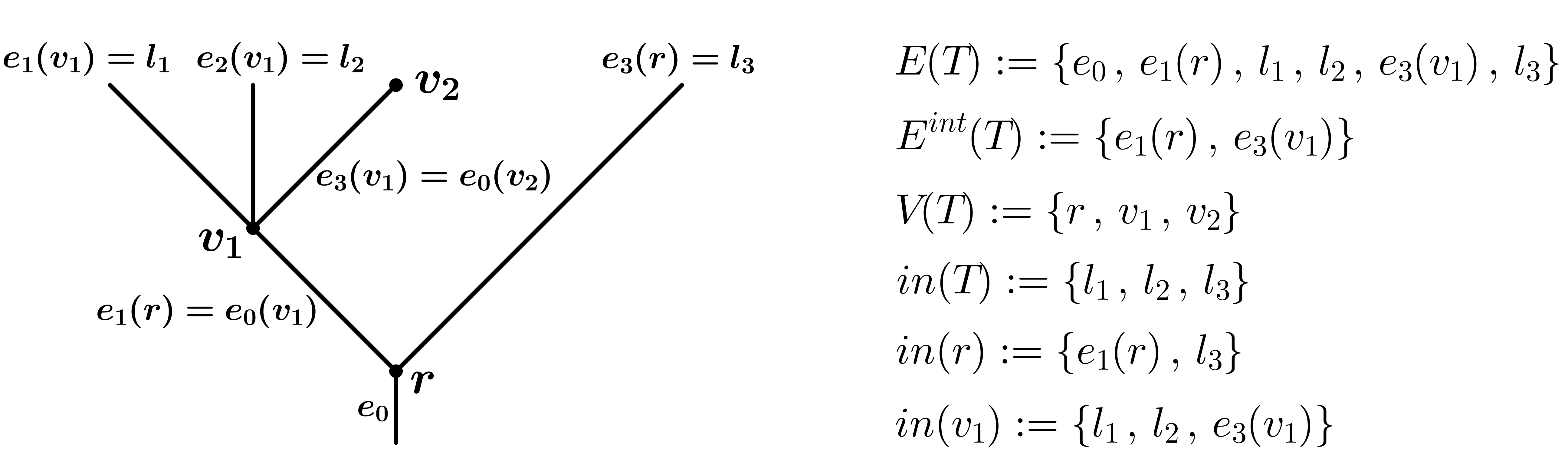}\vspace{-5pt}
\caption{Example of a planar tree.}\vspace{-15pt}
\end{center}
\end{figure}

The \textit{automorphism group} $Aut(T)$ can be described by induction on the number of vertices. If $|V(T)|=1$, then $Aut(T)$ is the group $\Sigma_{|T|}$. Otherwise, up to non-planar isomorphism, $T$ is of the form
\begin{equation}\label{B2}
T=t_{n}(T_{1}^{1},\ldots,T_{n_{1}}^{1},T_{1}^{2},\ldots,T_{n_{2}}^{2},\ldots, T_{1}^{l},\ldots,T_{n_{l}}^{l}), 
\end{equation}
where $t_{n}$ is a $n$-corolla, the trees $T_{1}^{i},\ldots,T_{n_{i}}^{i}$ are the same tree $T^{i}$ and $T^{i}$ is not isomorphic to $T^{j}$ if $i\neq j$. Since $\Sigma_{n_{i}}$ acts on the product $Aut(T^{i})^{\times n_{i}}$ by permuting the factors, the automorphism group of $T$ is the semi-direct product
\begin{equation}\label{B0}
Aut(T)\cong \big(\, Aut(T^{1})^{\times n_{1}}\times \cdots \times Aut(T^{l})^{\times n_{l}}\big)\rtimes \big(\Sigma_{n_{1}}\times \cdots \Sigma_{n_{l}}\big):= \Gamma_{T}\rtimes \Sigma_{T}.
\end{equation}

A tree is a pair $(T\,;\,\sigma)$ where $T$ is a planar tree and $\sigma:\{1,\ldots, |T|\}\rightarrow in(T)$ is a bijection labelling the leaves of $T$. Such an element will be denoted by $T$ if there is no ambiguity about the bijection $\sigma$. We denote by \textbf{tree} the set of trees.  The bijection $\sigma$ can be interpreted as an element in $\Sigma_{|T|}$. By abuse of notation, a tree $T=(T\,;\,\sigma)$ is said to be planar if $\sigma$ is the identity permutation.  
\end{defi}

\begin{const}\label{e7}
From an operad $O$, we build the operad $\mathcal{BV}(O)$. The points are equivalent classes $[T\,;\,\{t_{e}\}\,;\,\{a_{v}\}]$ where $T$ is a tree, $\{a_{v}\}_{v\in V(T)}$ is a family of points in $O$ labelling the vertices of $T$ and $\{t_{e}\}_{e\in E^{int}(T)}$ is a family of real numbers in the interval $[0\,,\,1]$ indexing the inner edges. In other words, $\mathcal{BV}(O)$ is the quotient of the coproduct
$$
\left.
\underset{T\in \,\text{\textbf{tree}}}{\coprod} \,\,\underset{v\in V(T)}{\prod}\,O(|v|) \,\,\times \,\,\underset{e\in\, E^{int}(T)}{\prod}\, [0\,,\,1]\,\,
\right/\!\sim\,\,
$$
where the equivalence relation is generated by the following axioms: 
\begin{itemize}[itemsep=-10pt, topsep=3pt, leftmargin=*]
\item[$i)$] If a vertex is labelled by the unit $\ast_{1}\in O(1)$, then one has locally the identifications

\begin{center}
\includegraphics[scale=0.27]{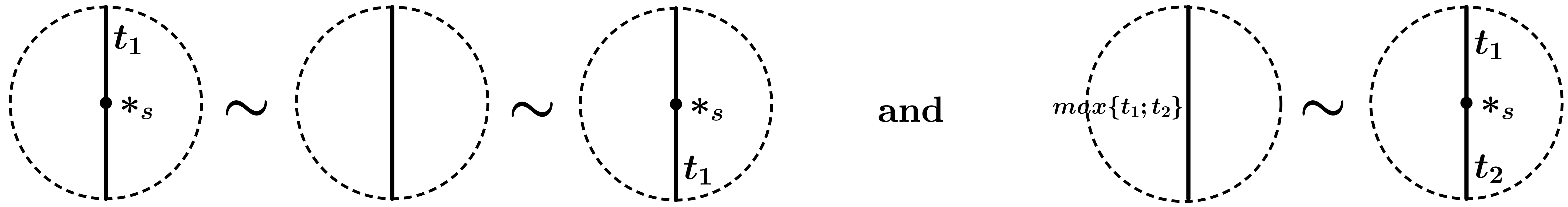}
\end{center}\vspace{11pt}

\item[$ii)$] If a vertex is labelled by $a\cdot \sigma$, with $\sigma\in \Sigma_{|v|}$, then 
\begin{center}
\includegraphics[scale=0.35]{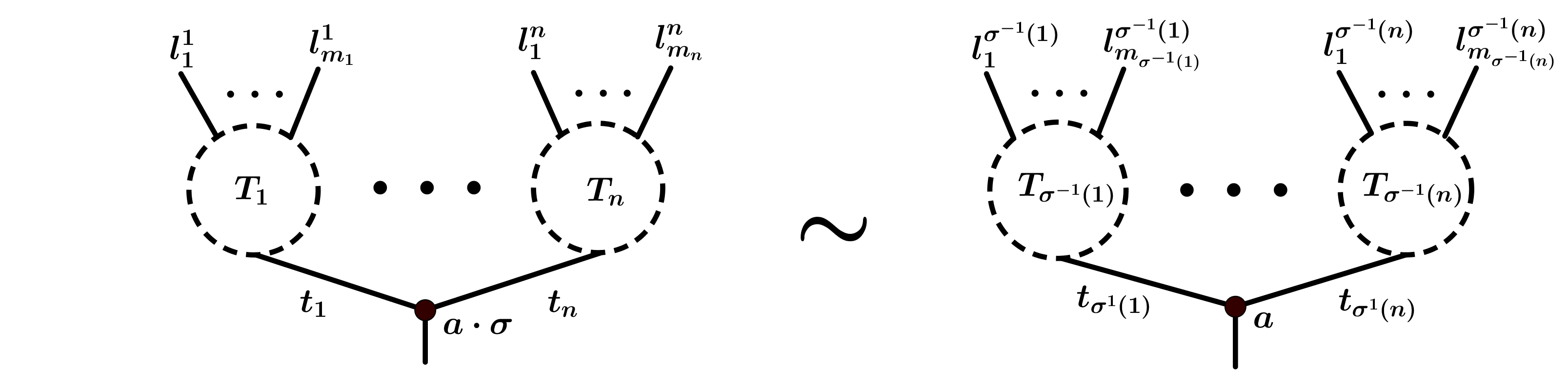}\vspace{10pt}
\end{center}

\item[$iii)$] If an inner edge is indexed by $0$, then we contract it by using the operadic structure of $O$.
\begin{figure}[!h]
\begin{center}
\includegraphics[scale=0.25]{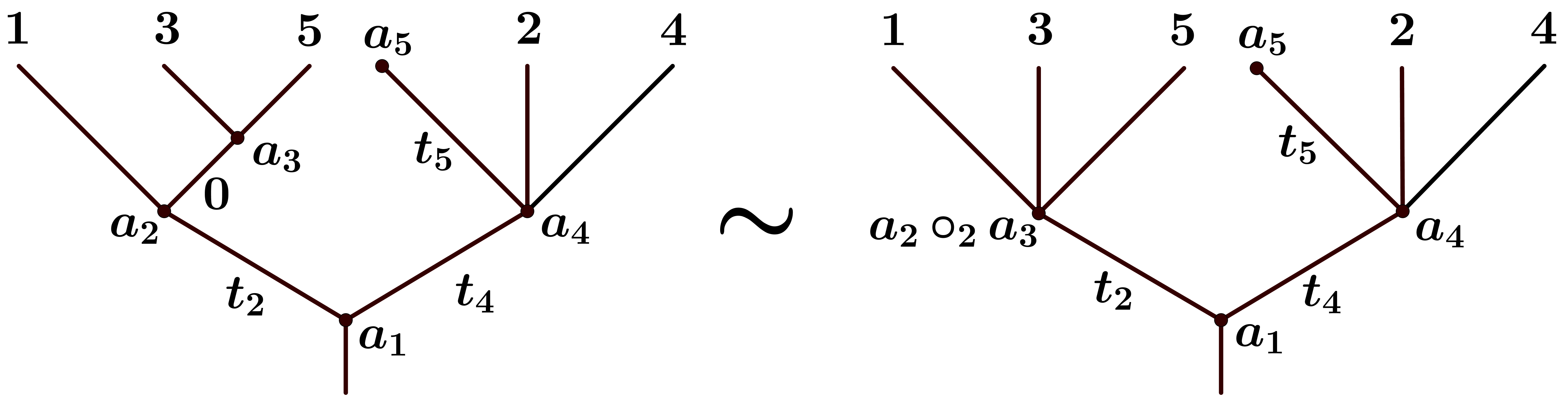}\vspace{-5pt}
\caption{Illustration of the relation $(iii)$.}\vspace{-40pt}
\end{center}
\end{figure}
\end{itemize}

\newpage

\noindent Let $[T\,;\,\{t_{e}\}\,;\,\{a_{v}\}]$ be a point in $\mathcal{BV}(O)(n)$ and $[T'\,;\,\{t'_{e}\}\,;\,\{a'_{v}\}]$ be a point in $\mathcal{BV}(O)(m)$. The operadic composition $[T\,;\,\{t_{e}\}\,;\,\{a_{v}\}]\circ_{i}[T'\,;\,\{t'_{e}\}\,;\,\{a'_{v}\}]$ consists in grafting $T'$ to the $i$-th incoming input of $T$ and indexing the new inner edge by $1$. Furthermore, there is a map of pointed sequences
\begin{equation}\label{B8}
\iota:O\longrightarrow \mathcal{BV}(O)\,\,;\,\,a\longmapsto [t_{|a|}\,;\,\emptyset\,;\,\{a\}],
\end{equation} 
sending a point $a$ to the corolla labelled by $a$. There is also the following map of operads sending the real numbers indexing the inner edges to $0$:
\begin{equation}\label{d8}
\mu:\mathcal{BV}(O)\rightarrow O\,\,;\,\, [T\,;\,\{t_{e}\}\,;\,\{a_{v}\}] \mapsto [T\,;\,\{0_{e}\}\,;\,\{a_{v}\}].
\end{equation}
\end{const}

From now on, we introduce a filtration of the resolution $\mathcal{BV}(O)$ according to the number of \textit{geometrical inputs} which is the number of leaves plus the number of univalent vertices. A point in $\mathcal{BV}(O)$ is said to be \textit{prime} if the real numbers indexing the set of inner edges are strictly smaller than $1$. Besides, a point is said to be \textit{composite} if one of its inner edges is indexed by $1$ and such a point can be decomposed into prime components. More precisely, the prime components of a point indexed by a planar tree are obtained by cutting the edges labelled by $1$ as shown in Figure \ref{G2}. Otherwise, the prime components of a point of the form $[(T\,;\,\sigma)\,;\,\{t_{e}\}\,;\,\{a_{v}\}]$, with $\sigma\neq id$, coincide with the prime components of $[(T\,;\,id)\,;\,\{t_{e}\}\,;\,\{a_{v}\}]$. 

\begin{figure}[!h]
\begin{center}
\includegraphics[scale=0.27]{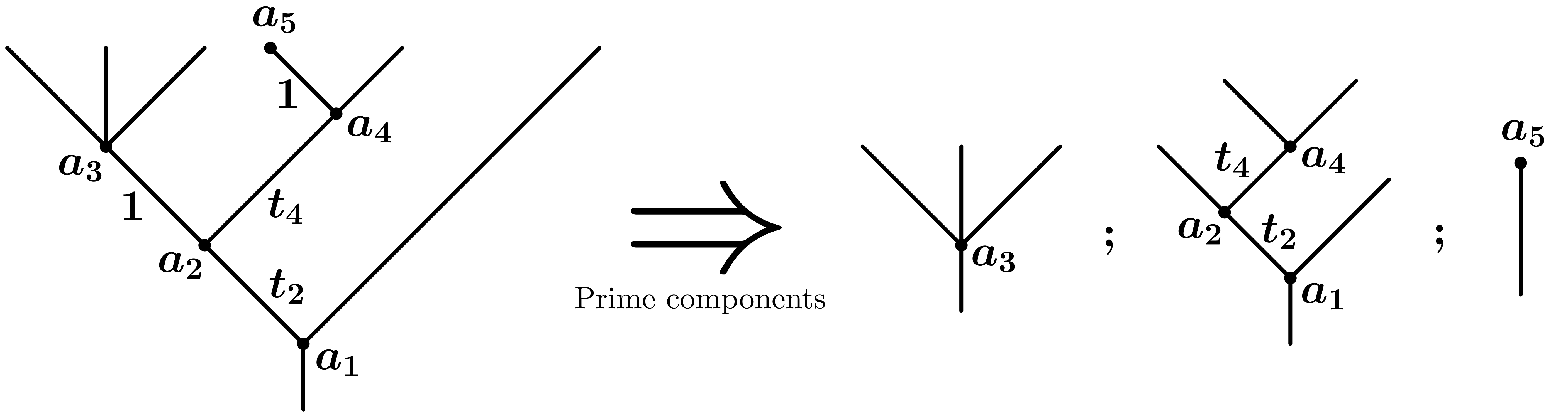}\vspace{-5pt}
\caption{Illustration of a composite point together with its prime components.}\label{G2}\vspace{-10pt}
\end{center}
\end{figure}

\noindent A prime point is in the $k$-th filtration term $\mathcal{BV}_{k}(O)$ if the number of its geometrical inputs is at most $k$. Then, a composite point is in the $k$-th filtration term if its prime components are in $\mathcal{BV}_{k}(O)$. For instance, the composite point in Figure \ref{G2} is an element in the filtration term $\mathcal{BV}_{4}(O)$. By convention, $\mathcal{BV}_{0}(O)$ is the initial object in the category of operads. For each $k\geq 0$, $\mathcal{BV}_{k}(O)$ is an operad and the family $\{\mathcal{BV}_{k}(O)\}$ produces the following filtration of $\mathcal{BV}(O)$: 
\begin{equation}\label{G4}
\xymatrix{
\mathcal{BV}_{0}(O)\ar[r] & \mathcal{BV}_{1}(O) \ar[r] &  \cdots \ar[r] & \mathcal{BV}_{k-1}(O) \ar[r] & \mathcal{BV}_{k}(O) \ar[r] & \cdots \ar[r] & \mathcal{BV}(O).
}
\end{equation}

\begin{thm}\label{I0}{\cite{Berger03,Vogt03}}
Assume that $O$ is a well pointed $\Sigma$-cofibrant operad. The objects $\mathcal{BV}(O)$ and $T_{k}(\mathcal{BV}_{k}(O))$ are cofibrant replacements of $O$ and $T_{k}(O)$ in the categories $Operad$ and $T_{k}Operad$ respectively. In particular, the map (\ref{d8}) is a weak equivalence. 
\end{thm}

From a $k$-truncated operad $O_{k}$, we consider the $k$-free operad $\mathcal{F}_{k}(O_{k})$ whose $k$ first components coincide with $O_{k}$. The functor $\mathcal{F}_{k}$ is left adjoint to the truncated functor $T_{k}$ and it can be expressed as a quotient of the free operad functor in which the equivalence relation is generated by the following axiom: if the sum of the incoming edges of any two consecutive vertices connected by an inner edge $e$ is smaller than $k+1$, then we contract $e$ using $k$-truncated operadic structure of $O_{k}$. In our case, we can easily check that $\mathcal{F}_{k}(T_{k}(\mathcal{BV}_{k}(O)))=\mathcal{BV}_{k}(O)$ since $\mathcal{BV}_{k}(O)$ is the sub-operad of $\mathcal{BV}(O)$ generated by its $k$ first components. Consequently, from this adjunction and Theorem \ref{I0}, we deduce the following identifications:
$$
T_{k}Operad^{h}(T_{k}(O)\,,\,T_{k}(O'))\cong  T_{k}Operad(T_{k}(\mathcal{BV}_{k}(O))\,,\,T_{k}(O')) \cong  Operad(\mathcal{BV}_{k}(O)\,,\,O').
$$  

\subsection{The tower of fibrations associated to the space $\Omega Operad^{h}(O\,;\,O')$}\label{F5}

In the previous subsection, we introduce a filtration of the Boardman-Vogt resolution $\mathcal{BV}(O)$ according to the number of geometrical inputs in order to get a cofibrant replacement of the truncated operad $T_{k}(O)$. Unfortunately, this is not enough to prove the main theorem of the paper and we need a refinement of the filtration (\ref{G4}). Indeed, for each inclusion $\mathcal{BV}_{k-1}(O)\rightarrow \mathcal{BV}_{k}(O)$, there is another filtration according to the number of vertices: 
\begin{equation}\label{G5}
\xymatrix@C=18pt{
\mathcal{BV}_{k-1}(O)\ar[r] & \mathcal{BV}_{k}(O)[1] \ar[r] &  \cdots \ar[r] & \mathcal{BV}_{k}(O)[l-1] \ar[r] & \mathcal{BV}_{k}(O)[l] \ar[r] & \cdots \ar[r] & \mathcal{BV}_{k}(O).
}\vspace{-20pt}
\end{equation}
\newpage

\noindent More precisely, a prime point in $\mathcal{BV}(O)$ is said to be in the filtration term $\mathcal{BV}_{k}(O)[l]$ if it has at most $k-1$ geometrical inputs or it has exactly $k$ geometrical inputs and at most $l$ vertices. Similarly, a composite point is in the filtration term $\mathcal{BV}_{k}(O)[l]$ if its prime components are in $\mathcal{BV}_{k}(O)[l]$. For instance, the composite point in Figure \ref{G2} is in $\mathcal{BV}_{4}(O)[3]$. By construction, each sequence $\mathcal{BV}_{k}(O)[l]$ is an operad.

In what follows, we recall the proof that the inclusion from $\mathcal{BV}_{k}(O)[l-1]$ to $\mathcal{BV}_{k}(O)[l]$ is a cofibration and we introduce a tower of fibrations associated to the loop space $\Omega Operad^{h}(O\,;\,O')$. For this purpose, let \textbf{tree[k\,;\,l]} be the set of trees having exactly $k$ geometrical inputs and $l$ vertices. The sequence $X_{k}[l]$ is defined by the quotient of the coproduct 
$$
\left.
\underset{T\in \,\textbf{tree[k\,;\,l]}}{\coprod} \,\,\underset{v\in V(T)}{\prod}\,O(|v|) \,\,\times \,\,\underset{e\in\, E^{int}(T)}{\prod}\, [0\,,\,1]\,\,
\right/\!\sim
$$ 
where the equivalence relation is generated by the axiom $(ii)$ of Construction \ref{e7}. The boundary sequence $\partial X_{k}[l]$ is formed by points in $X_{k}[l]$ having a bivalent vertex labelled by the unit of the operad $O$ or having an inner edge indexed by $0$ or $1$. For $(k\,;\,l)\neq (1\,;\,1)$,  $X_{k}[l]$ and $\partial X_{k}[l]$ are not pointed sequences. In order to use the free operad functor, we consider the pointed sequences $\tilde{X}_{k}[l]$ and $\partial \tilde{X}_{k}[l]$ obtained by adding a based point in arity $1$:
$$
\tilde{X}_{k}[l](n):=
\left\{
\begin{array}{ll}\vspace{4pt}
X_{k}[l](1)\sqcup \ast_{1} & \text{if } n=1, \\ 
X_{k}[l](n) & \text{otherwise},
\end{array} 
\right.
\hspace{15pt}\text{and}\hspace{15pt}
\partial\tilde{X}_{k}[l](n):=
\left\{
\begin{array}{ll}\vspace{4pt}
\partial X_{k}[l](1)\sqcup \ast_{1} & \text{if } n=1, \\ 
\partial X_{k}[l](n) & \text{otherwise}.
\end{array} 
\right.
$$ 

\noindent Then, one has the following pushout diagrams where $\mathcal{F}$ is the free operad functor from pointed sequences to operads and the left vertical maps consists in forgetting the vertices labelled by the unit, contracting the inner edges indexed by $0$ and taking the inclusion for elements having an edge indexed by $1$:
$$
\xymatrix@R=17pt{
\mathcal{F}(\mathcal{BV}_{0}(O)) \ar[r] \ar@{=}[d] & \mathcal{F}(X_{1}[1]) \ar@{=}[d] \\
\mathcal{BV}_{0}(O) \ar[r] & \mathcal{BV}_{1}(O)[1]
}
\hspace{15pt}\text{and}\hspace{15pt}
\xymatrix@R=17pt{
\mathcal{F}(\partial \tilde{X}_{k}[l]) \ar[r] \ar[d] & \mathcal{F}(\tilde{X}_{k}[l]) \ar[d] \\
\mathcal{BV}_{k}(O)[l-1] \ar[r] & \mathcal{BV}_{k}(O)[l]
}
$$

Similarly to the proof of \cite[Theorem 2.12]{Ducoulombier16} and \cite[Theorem 5.1]{Berger06}, we can check that the inclusion from $\partial X_{k}[l]$ to $X_{k}[l]$ is a $\Sigma$-cofibration. As a consequence, the horizontal maps of the above diagrams as well as the inclusion from $\mathcal{BV}_{k-1}(O)$ to $\mathcal{BV}_{k}(O)$ are cofibrations in the category of operads. Furthermore, the filtration (\ref{G5}) gives rise to a tower of fibrations computing the loop space $\Omega Operad(\mathcal{BV}(O)\,;\, O')$. Let $\partial' X_{k}[l]$ be the pushout product 
$$
\partial X_{k}[l]\times [0\,,\,1] \underset{\partial X_{k}[l]\times \{0\,,\,1\}}{\coprod} X_{k}[l] \times \{0\,,\,1\}.
$$ 
Due to Lemma \ref{refA}, the map from the pushout product  to $X_{k}[l]\times [0\,,\,1]$ is a $\Sigma$-cofibration and the vertical maps of the following pullback diagram are fibrations:
$$
\xymatrix@R=17pt{
\Omega Operad(\mathcal{BV}_{k}(O)[l]\,;\, O') \ar[r] \ar[d] & Seq(X_{k}[l]\times [0\,,\,1]\,;\, O')\ar[d] \\
\Omega Operad(\mathcal{BV}_{k}(O)[l-1]\,;\, O') \ar[r] & Seq( \partial' X_{k}[l] \,;\, O')
}
$$
Furthermore, if $g\in \Omega Operad(\mathcal{BV}_{k}(O)[l-1]\,;\, O')$, then the fiber over $g$ is homeomorphic to the mapping space of sequences from $X_{k}[l]\times [0\,,\,1]$ to $O'$ such that the restriction to the sub-sequence $\partial'X_{k}[l]$ coincides with the map induced by $g$:
\begin{equation}\label{C6}
Seq^{g}\big(\,(X_{k}[l]\times [0\,,\,1]\,,\, \partial'X_{k}[l])\,;\, O'\,\big).
\end{equation}

Finally, we give a description of the space $\Omega Operad(\mathcal{BV}_{1}(O)[1]\,;\, O')$. By using the equality $\mathcal{BV}_{1}(O)[1]=\mathcal{F}(X_{1}[1])$ together with the adjunction between operads and pointed sequences, we deduce that a point in the loop space is given by a pair of continuous maps
\begin{equation}\label{D7}
f_{0}:O(0)\times [0\,,\,1]\longrightarrow O'(0) 
\hspace{15pt}\text{and}\hspace{15pt}
f_{1}:O(1)\times [0\,,\,1]\longrightarrow O'(1),
\end{equation}
satisfying the relations 
$$f_{1}(\ast_{1}\,;\,t)=\ast_{1}',\hspace{15pt} f_{1}(x\,;\,0)=f_{1}(x\,;\,1)= \eta(x)\hspace{15pt} \text{and}\hspace{15pt} f_{0}(x\,;\,0)=f_{0}(x\,;\,1)= \eta(x)$$ where $\ast_{1}$ and $\ast_{1}'$ are the units of the operads $O$ and $O'$ respectively.

\section{Bimodule and cofibrant replacements}

This section is devoted to the category of bimodules over an operad with the example of the non-$(l)$-overlapping little cubes bimodule as well as a short presentation of the free bimodule functor $\mathcal{F}_{B}$. Then, we introduce an explicit cofibrant replacement $\mathcal{WB}(O)$ of the operad $O$ in the category $Bimod_{O}$ endowed with a filtration compatible with the tower of fibrations introduced in Section \ref{F5}. By using the properties of this cofibrant replacement, we define a $\mathcal{C}_{1}$-algebra structure on the space $Bimod_{O}(\mathcal{WB}(O)\,;\,O')$ and  its truncated versions.

\subsection{Bimodules over an operad and the free bimodule functor}

From now on, $O$ is a topological operad. A bimodule over $O$, also called $O$-\textit{bimodule}, is given by a sequence $M\in Seq$ endowed with operations 
\begin{equation}\label{A2}
\begin{array}{llr}\vspace{7pt}
\gamma_{r}: & M(n)\times O(m_{1})\times \cdots\times O(m_{n})  \longrightarrow  M(m_{1}+\cdots + m_{n}), & \text{right operations},\\ \vspace{7pt}
\gamma_{l}: & O(n)\times M(m_{1})\times \cdots\times M(m_{n})  \longrightarrow  M(m_{1}+\cdots + m_{n}),& \text{left operations},
\end{array}
\end{equation}
satisfying compatibility with the action of the symmetric group, associativity and unity axioms (see \cite{Arone14}). In particular, there is a continuous map $\gamma_{0}:O(0)\rightarrow M(0)$ in arity $0$. A map between $O$-bimodules should respect the operations. We denote by $Bimod_{O}$ the category of $O$-bimodules. Thanks to the unit in $O(1)$, the right operations $\gamma_{r}$ can equivalently be defined as a family of continuous maps
$$
\circ^{i}:M(n)\times O(m)\longrightarrow M(n+m-1),\hspace{15pt}\text{with }1\leq i\leq n.
$$

Given an integer $k\geq 0$, we also consider the category of $k$-truncated bimodules $T_{k}Bimod_{O}$. An object is a $k$-truncated sequence endowed with a bimodule structure (\ref{A2}) for $m_{1}+\cdots + m_{n}\leq k$ (and $n\leq k$ for $\gamma_{r}$). One has an obvious functor 
$$
T_{k}(-):Bimod_{O}\longrightarrow T_{k}Bimod_{O}.
$$
For the rest of the paper, we use the following notation: 
$$
\begin{array}{ll}\vspace{7pt}
x\circ^{i}y=\circ^{i}(x\,;\,y) & \text{for } x\in M(n) \text{ and } y\in O(m),  \\ 
x(y_{1},\ldots,y_{n})=\gamma_{l}(x\,;\,y_{1}\,;\ldots;\,y_{n}) & \text{for } x\in O(n) \,\,\text{ and } y_{i}\in M(m_{i}).
\end{array} 
$$

\begin{expl}\label{G3}\textbf{The non-$(l)$-overlapping little cubes bimodule $\mathcal{C}_{d}^{(l)}$}\\
The $d$-dimensional non-$(l)$-overlapping little cubes bimodule $\mathcal{C}_{d}^{(l)}$ has been introduced by Dobrinskaya and Turchin in \cite{Turchin14}. The space $\mathcal{C}_{d}^{(l)}(n)$ is the subspace of $ \mathcal{C}_{d}^{\infty}(n)$ formed by configurations of $n$ little cubes $<c_{1},\ldots,c_{n}>$ satisfying the following relation:
\begin{equation}\label{i3}
\forall\, i_{1}< \cdots < i_{l} \in \{1,\ldots, n\}, \hspace{15pt} \underset{1\leq j\leq l}{\bigcap}\, Int(Im(c_{i_{j}}))=\emptyset.
\end{equation}
In particular, $\mathcal{C}_{d}^{(2)}$ coincides with the little cubes operad $\mathcal{C}_{d}$. The action of the symmetric group and the bimodule structure over the little cubes operad $\mathcal{C}_{d}$ arise from the operadic structure of $\mathcal{C}_{d}^{\infty}$. Let us notice that the non-$(l)$-overlapping little cubes bimodule, with $l>2$, is not an operad since the operadic compositions of  $\mathcal{C}_{d}^{\infty}$ don't necessarily preserve the condition (\ref{i3}).
\begin{figure}[!h]
\begin{center}
\includegraphics[scale=0.27]{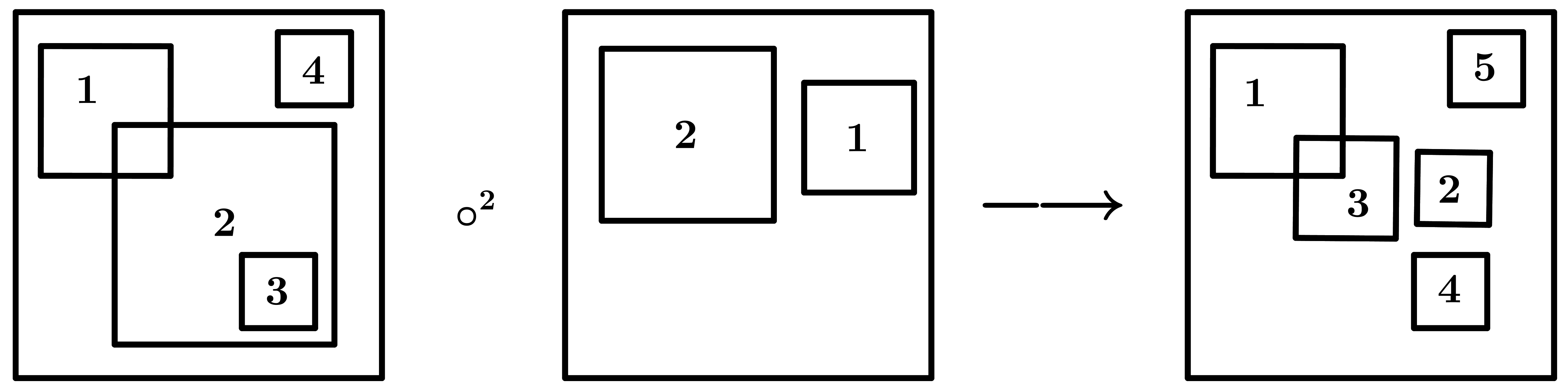}
\caption{The right operation $\circ^{2}:\mathcal{C}_{2}^{(3)}(4) \times \mathcal{C}_{2}(2)\rightarrow \mathcal{C}_{2}^{(3)}(5)$.}\vspace{-30pt}
\end{center}
\end{figure}
\end{expl}

\newpage

\begin{expl}\label{d1}
Let $\eta:O\rightarrow O'$ be a  map of operads. In that case, the map $\eta$ is also  a bimodule map over $O$ and the bimodule structure on $O'$ is defined as follows: 
$$
\begin{array}{lrll}\vspace{4pt}
\gamma_{r}: & O'(n)\times O(m_{1})\times \cdots\times O(m_{n})  & \longrightarrow &  O'(m_{1}+\cdots + m_{n}); \\\vspace{7pt}
& (x\,;\,y_{1},\ldots,y_{n})&\longmapsto & (\cdots(x\circ_{n}\eta(y_{n}))\cdots)\circ_{1}\eta(y_{1})),\\ \vspace{4pt}
\gamma_{l}: &  O(n)\times O'(m_{1})\times \cdots\times O'(m_{n})  & \longrightarrow &  O'(m_{1}+\cdots +m_{n});\\
& (x\,;\,y_{1},\ldots,y_{n})&\longmapsto & (\cdots(\eta(x)\circ_{n}y_{n})\cdots)\circ_{1}y_{1}.
\end{array} 
$$
\end{expl}

\begin{expl}\label{G8}\textbf{The free bimodule functor $\mathcal{F}_{B}$}\\
We recall quickly a presentation of the free bimodule functor introduced by the author in \cite{Ducoulombier16, Ducoulombier14}. Since we have to deal with the symmetric group action and the map $\gamma :O(0)\rightarrow M(0)$, we define an adjunction between the category of bimodules over $O$ and the under category $O_{0}\downarrow Seq$ where $O_{0}$ is the sequence given by $O_{0}(0)=O(0)$ and the empty set otherwise:
\begin{equation}\label{F9}
\mathcal{F}_{B}:O_{0}\downarrow Seq \leftrightarrows Bimod_{O}:\mathcal{U}.
\end{equation}

The functor $\mathcal{F}_{B}$ is defined using the set of reduced trees with section $rstree$ whose elements are pairs $(T\,;\,V^{p}(T))$ with $T$ a tree and $V^{p}(T)$ a subset of vertices, called the \textit{set of pearls}, usually represented by a white vertex (see Figure \ref{G7}). Such a pair has to satisfy two conditions: $i)$ each path connecting a leaf to the trunk passes through a unique pearl $ii)$ two consecutive vertices cannot belong both to the set $V(T)\setminus V^{p}(T)$. From a sequence $M\in O_{0}\downarrow Seq$, the bimodule $\mathcal{F}_{B}(M)$ consists in  labelling the pearl by points in $M$ and the other vertices by points in the operad $O$ as illustrated in Figure \ref{G7}. More precisely, $\mathcal{F}_{B}(M)$ is the quotient of the coproduct
$$
\left.\coprod\limits_{T\in rstree} \,\prod\limits_{v\in V^{p}(T)}\, M(|v|) \times \prod\limits_{v\in V(T)\setminus V^{p}(T)}\, O(|v|)\right/\sim
$$
where the equivalence relation is generated by the compatibility with the symmetric group action, the compatibility with the unit of the operad $O$ and the compatibility with the map $\gamma:O(0)\rightarrow M(0)$ (we refer the reader to \cite{Ducoulombier16}). 

If $a\in O(n)$ and $\{x_{i}\}$ is a family of points in $\mathcal{F}_{B}(M)(m_{i})$, then the left module operations is defined as follows: each tree of the family is grafted to a leaf of the $n$-corolla indexed by $a$ from left to right. The inner edges obtained are contracted if their source are not pearl by using the operadic structure of $O$. Similarly, if $x\in\mathcal{F}_{B}(M)(m)$, then the composition $x\circ^{i}a$ consists in grafting the corolla labelled by $a$ to the $i$-th incoming edge of $T$. Then, we contract the inner edge so obtained if its target is not a pearl by using the operadic structure of $O$.
\begin{figure}[!h]
\begin{center}
\includegraphics[scale=0.4]{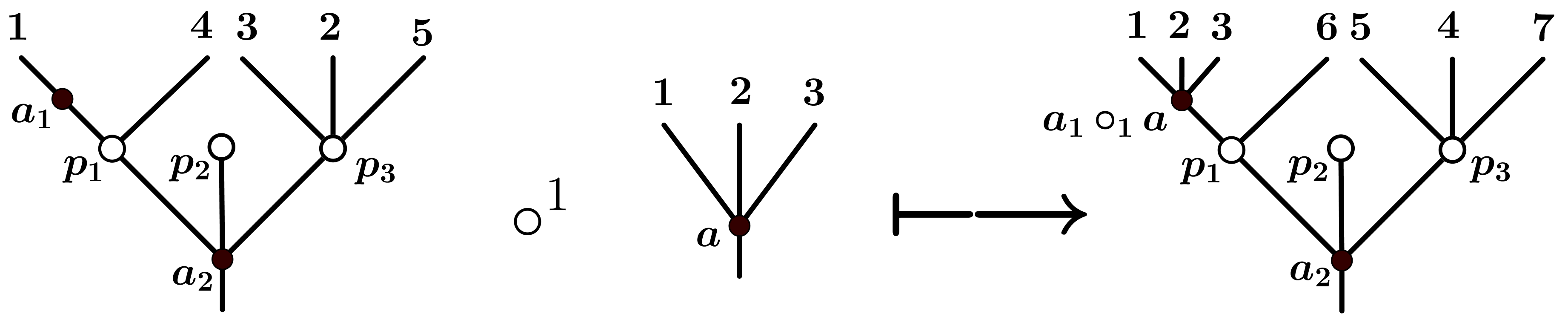}
\caption{Illustration of the right operation $\circ^{1}:\mathcal{F}_{B}(M)(5)\times O(3)\rightarrow \mathcal{F}_{B}(M)(7)$.}\label{G7}\vspace{-10pt}
\end{center}
\end{figure}
\end{expl}

The category $O_{0}\downarrow Seq$ inherits a cofibrantly generated model category structure from the category of sequences. By using the adjunction (\ref{F9}), we deduce that the category of "truncated" bimodule over $O$ is equipped with a cofibrantly generated model category structure in which a map $f$ is a weak equivalence (resp. a fibration) if the corresponding map $\mathcal{U}(f)$ is a weak equivalence (resp. a fibration) in the category of "truncated" sequences. Similarly to the operadic case, all the objects are fibrant and we only need a cofibrant replacement of $O$ in the category of $O$-bimodule in order to compute the derived mapping space $Bimod^{h}_{O}(O\,;\,O')$.

\begin{rmk}
The model category structure so defined in the category of "truncated" bimodules coincides with the usually model category structure considered by Fresse in \cite{Fresse09} that one gets from the adjunction between "truncated" sequences and "truncated" bimodules. Essentially because the classical adjunction can be factorized as follows:
$$
Seq\leftrightarrows O_{0}\downarrow Seq \leftrightarrows Bimod_{O}.
$$ 
\end{rmk}

\subsection{An alternative cofibrant replacement for bimodules coming from operads}\label{C5}

In \cite{Ducoulombier16}, we introduce a functorial way to build a cofibrant replacement in the category of bimodules over an operad $O$. Unfortunately, in the special case where the bimodule is the operad $O$ itself, this resolution forgets all the information coming from the operadic structure and we cannot expect to get an explicit $\mathcal{C}_{1}$-algebra structure on the derived mapping space of bimodules. For this reason, we consider an alternative cofibrant resolution $\mathcal{WB}(O)$ using the operadic structure of $O$ and such that the space $Bimod_{O}(\mathcal{WB}(O)\,;\,O')$ is an algebra over $\mathcal{C}_{1}$. Let us mention that this construction is inspired by  constructions in \cite{Tourtchine10} in the context of non-symmetric operad and in the particular case $O=\mathcal{A}s$.

\begin{const}\label{C0}
From an operad $O$, we build an $O$-bimodule $\mathcal{WB}(O)$. The points are equivalent classes $[T\,;\,\{t_{v}\}\,;\,\{x_{v}\}]$ where $T$ is a tree, $\{t_{v}\}$ is a family of real numbers in the interval $[0\,,\,1]$ indexing the vertices and $\{x_{v}\}$ is a family of points in $\mathcal{BV}(O)$ labelling the vertices. Furthermore, if $e$ is an inner edge of $T$, then one has $t_{s(e)}\geq t_{t(e)}$. In other words, $\mathcal{WB}(O)$ is the quotient of the sub-sequence
\begin{equation}\label{B9}
\left.
\underset{T\in \textbf{tree}}{\coprod}\,\,\,\underset{v\in V(T)}{\prod}\, \mathcal{BV}(O)(|v|)\times [0\,,\,1]\,
\right/ \sim
\end{equation}
coming from the restriction on the families of real numbers $\{t_{v}\}$. The equivalence relation is generated by the following axioms where $\iota:O\rightarrow \mathcal{BV}(O)$ and $\mu:\mathcal{BV}(O) \rightarrow O$ are the maps (\ref{B8}) and (\ref{d8}) respectively:

\begin{itemize}[itemsep=-10pt, topsep=3pt, leftmargin=*]
\item[$i)$] If a vertex is labelled by the unit $\iota(\ast_{1})$, with $\ast_{1}\in O(1)$, then  one has locally the identification

\begin{center}
\includegraphics[scale=0.27]{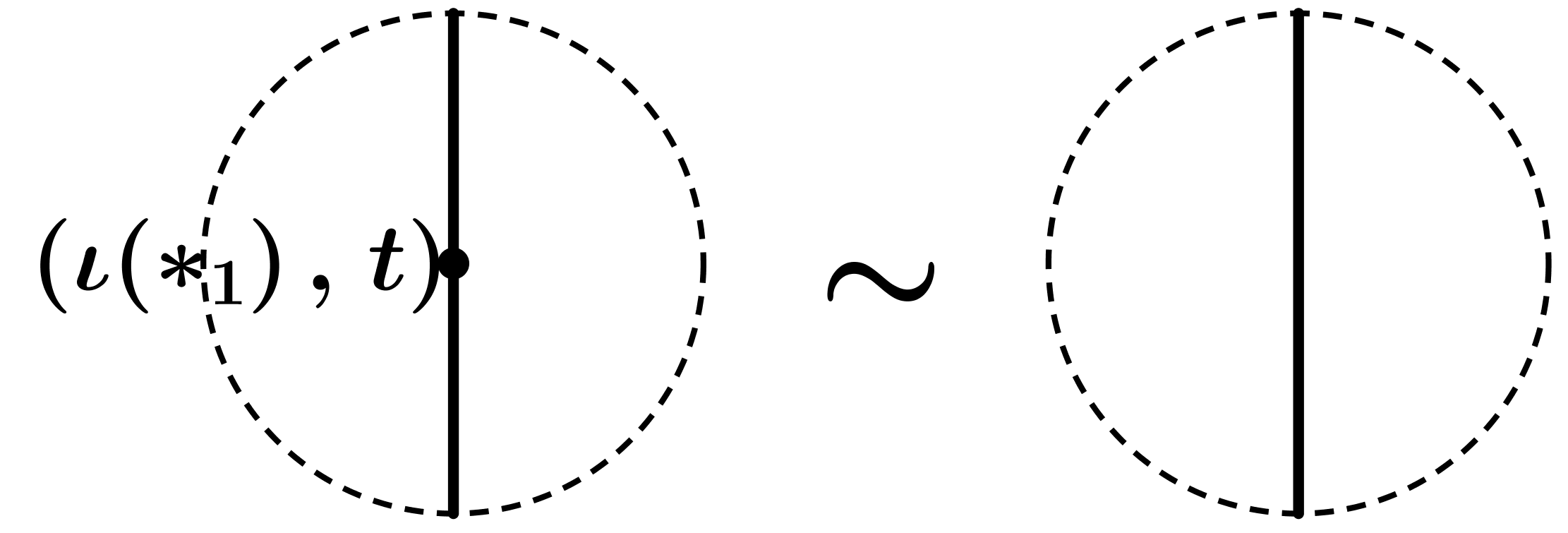}
\end{center}\vspace{7pt}

\item[$ii)$] If a vertex is labelled by $x\cdot \sigma$, with $x\in\mathcal{BV}(O)$ and $\sigma\in \Sigma_{|v|}$, then 
\begin{center}
\includegraphics[scale=0.35]{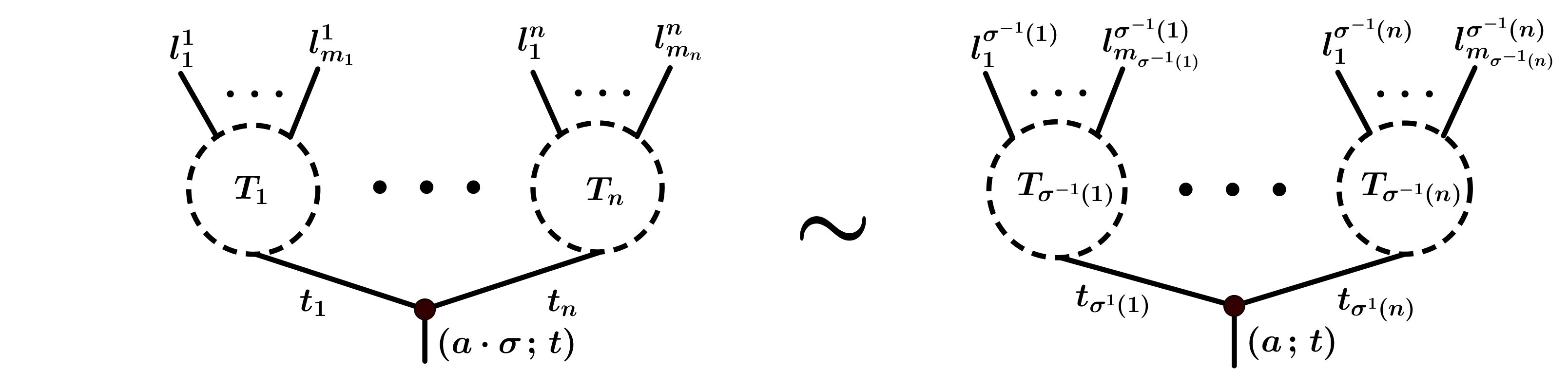}
\end{center}\vspace{7pt}

\item[$iii)$] If two consecutive vertices, connected by an edge $e$, are indexed by the same real number $t\in [0\,,\,1]$, then the two vertices are identified by contracting $e$ using the operadic structure of $\mathcal{BV}(O)$. The vertex so obtained is indexed by $t$.
\begin{figure}[!h]
\begin{center}
\includegraphics[scale=0.34]{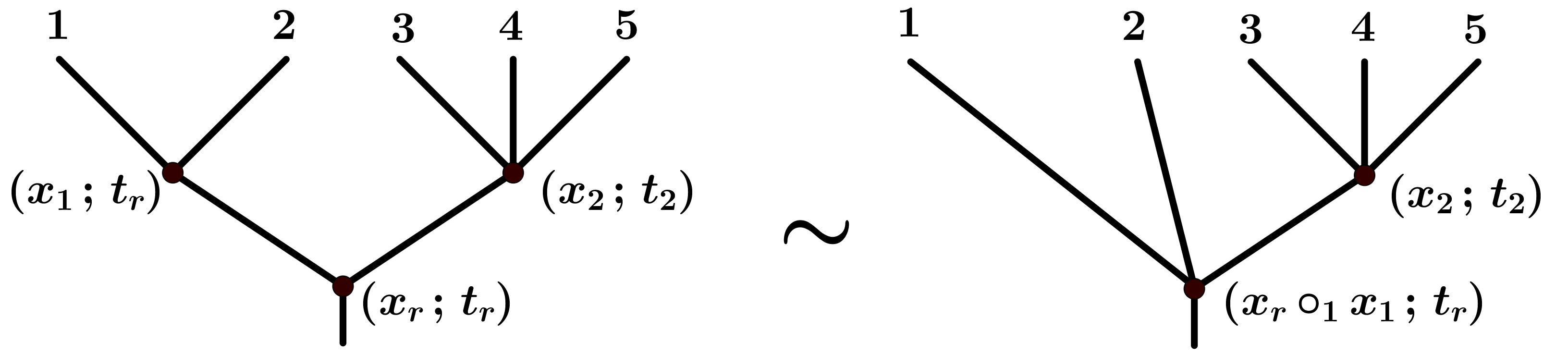}
\caption{Illustration of the relation $(iii)$.}\label{C3}
\end{center}
\end{figure}

\item[$iv)$] If a vertex is labelled by $(x\,;\,\xi)$, with $x\in \mathcal{BV}(O)$ and $\xi\in \{0\,,\,1\}$, then $(x\,;\,\xi)$ is identified with $(\iota\circ\mu(x)\,;\,\xi)$.
\begin{figure}[!h]
\begin{center}
\includegraphics[scale=0.08]{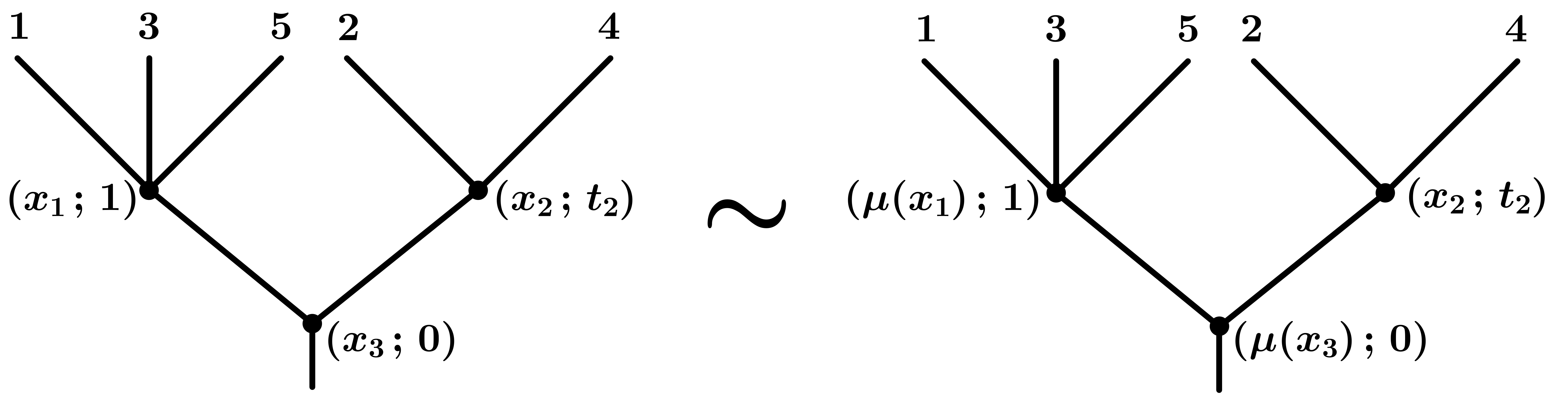}
\caption{Illustration of the relation $(iv)$.}\vspace{-25pt}
\end{center}
\end{figure}
\end{itemize}

\newpage

The map $\gamma:O(0)\rightarrow \mathcal{WB}(O)$ sends a point $a_0$ to the corolla indexed by the pair $(\iota(a_0)\,;\,0)$. Furthermore, if $[T\,;\,\{t_{v}\}\,;\,\{x_{v}\}]\in \mathcal{WB}(O)(n)$ and $a\in O(m)$, then the right operation $[T\,;\,\{t_{v}\}\,;\,\{x_{v}\}]\circ^i a$ consists in grafting the $m$-corolla indexed by the pair $(\iota(a)\,;\,1)$ to the $i$-th incoming edge of $T$. Similarly, if $[T^i\,;\,\{t^i_{v}\}\,;\,\{x^i_{v}\}]$ is a family of points in $\mathcal{WB}(O)$, then left operation $a([T^1\,;\,\{t^1_{v}\}\,;\,\{x^1_{v}\}],\ldots, [T^m\,;\,\{t^m_{v}\}\,;\,\{x^m_{v}\}])$ is defined as follows: each tree of the family is grafted from the left to right to a leaf of the $m$-corolla indexed by the pair $(\iota(a)\,;\,0)$.
\begin{figure}[!h]
\begin{center}
\includegraphics[scale=0.3]{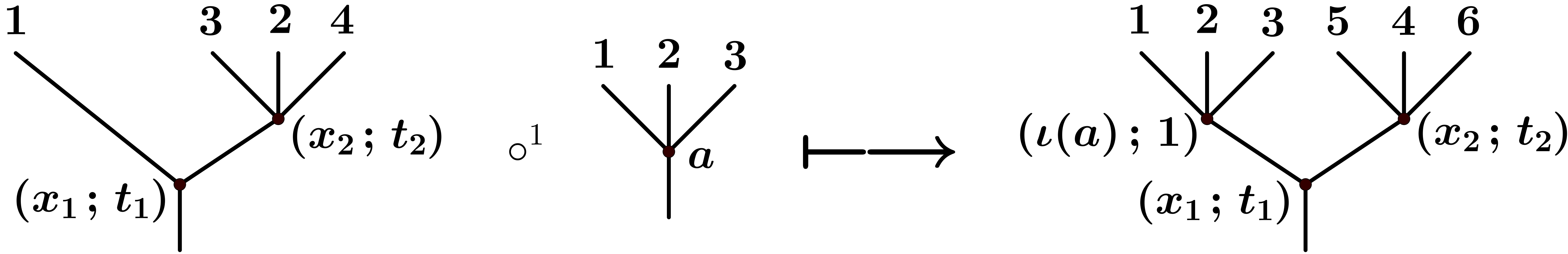}
\caption{Illustration of the right operations.}\vspace{-20pt}
\end{center}
\end{figure}
\end{const}  

From now on, we introduce a filtration of the resolution $\mathcal{WB}(O)$ according to the number of geometrical inputs of the trees labelling the vertices. For this purpose, we give an equivalent definition of the coproduct (\ref{B9}) by using the set $\Upsilon$. An element $[T\,;\,\{T_{v}\}]\in \Upsilon$ is given by a tree $T$, called the \textit{main tree}, and to each $v \in V(T)$ we associate another tree $T_{v}$ with $|v|$ leaves, called an \textit{auxiliary tree}.  So, the coproduct (\ref{B9}) is equivalent to the quotient of the sub-sequence
\begin{equation*}
\left.
\underset{[T\,;\,\{T_{v}\}]\in \Upsilon}{\coprod}\,\,\,\underset{v\in V(T)}{\prod}\, [0\,,\,1]\times \big( \, \underset{v'\in V(T_{v})}{\prod}\, O(|v'|)\times \underset{e\in E^{int}(T_{v})}{\prod}\, [0\,,\,1]\,\big)\,
\right/ \sim
\end{equation*}

\noindent coming from the restriction $t_{s(e)}\geq t_{t(e)}$ on the real numbers $\{t_{v}\}$ indexing the vertices of the main tree. The equivalence relation is generated by the axioms of Construction \ref{e7} and \ref{C0}. Let us remark that, due to the compatibility with the symmetric group axiom, we can suppose that the auxiliary trees are planar.
\begin{figure}[!h]
\begin{center}
\includegraphics[scale=0.4]{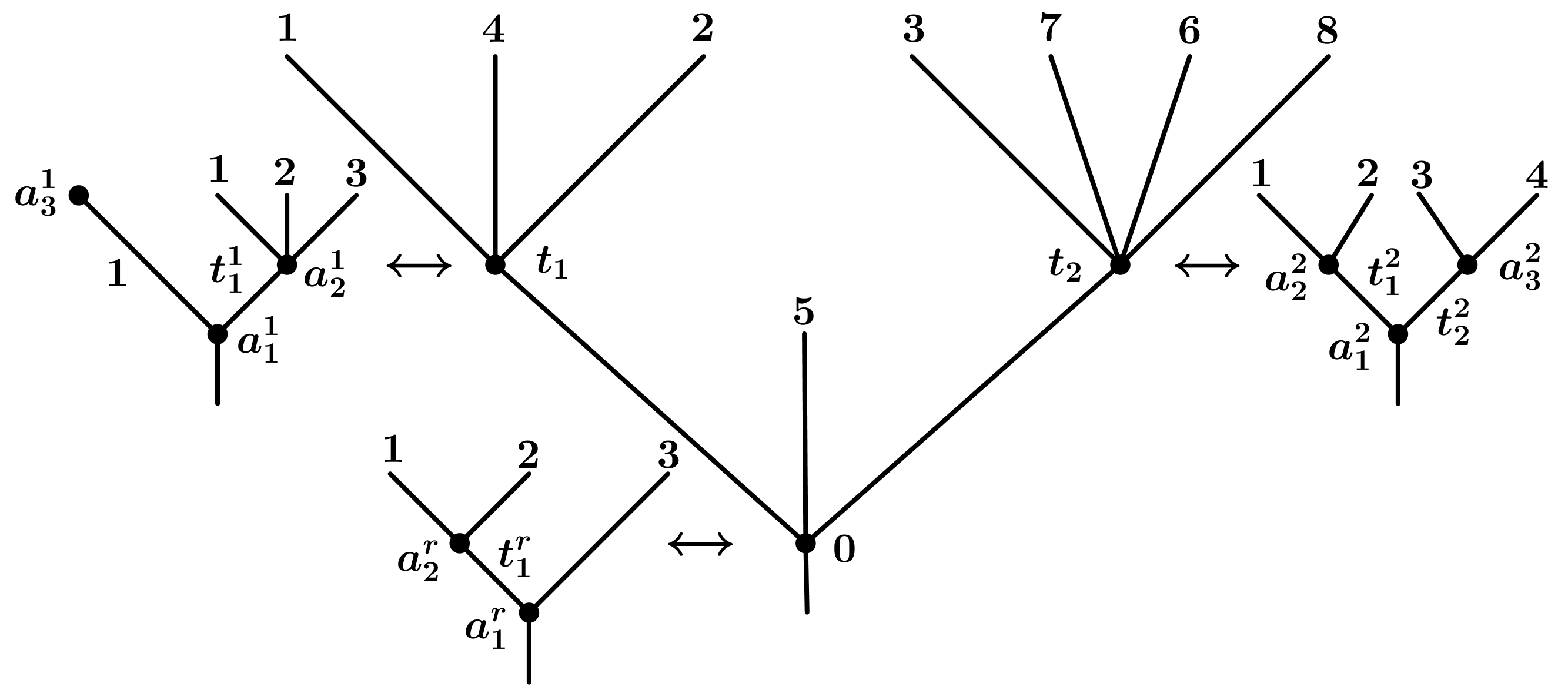}
\caption{Example of a point in $\mathcal{WB}(O)(8)$.}\label{C2}\vspace{-15pt}
\end{center}
\end{figure}

Similarly to the operadic case, we introduce a filtration of the resolution $\mathcal{WB}(O)$ according to the number of geometrical inputs which is the number of leaves of the main tree plus the number of univalent vertices of the auxiliary trees. A point in $\mathcal{WB}(O)$ is said to be prime if the real numbers indexing the vertices of the main tree are in the interval $]0\,,\,1[$. Besides, a point is said to be composite if one vertex of the main tree is indexed by $0$ or $1$ and such a point can be decomposed into prime components. More precisely, the prime components of a point, the main tree of which is planar, are obtained by removing the vertices of the main tree indexed by $0$ or $1$. Otherwise, the prime components of a point of the form $[(T\,;\,\sigma)\,;\,\{t_{v}\}\,;\,\{x_{v}\}]$, with $\sigma\neq id$, coincides with the prime components of $[(T\,;\,id)\,;\,\{t_{v}\}\,;\,\{x_{v}\}]$. For instance, the two prime components associated to the composite point in Figure \ref{C2} are the following ones: 
\begin{center}
\includegraphics[scale=0.22]{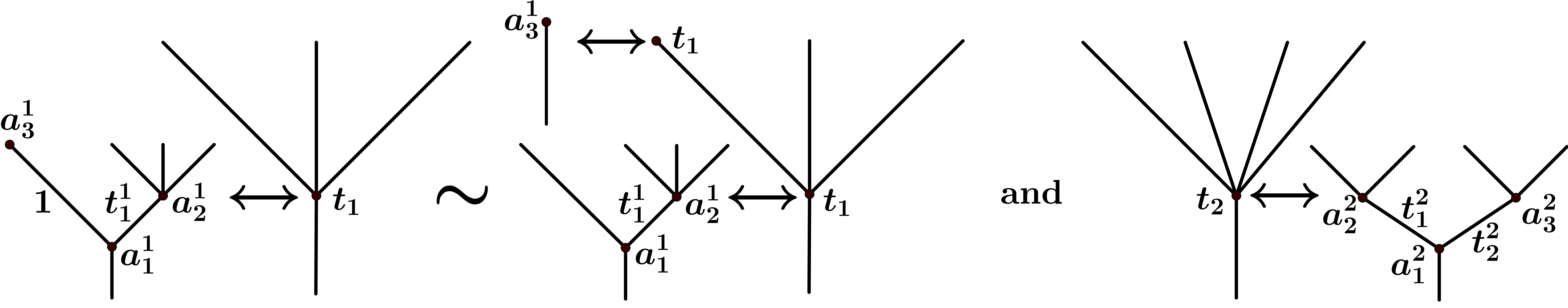}\vspace{-20pt}
\end{center}

\newpage

\noindent A prime point is in the $k$-th filtration term $\mathcal{WB}_{k}(O)$ if the number of geometrical inputs is at most $k$. Similarly, a composite point is in the $k$-filtration if its prime components are in $\mathcal{WB}_{k}(O)$. For instance, the composite point in Figure \ref{C2} is an element in the filtration term $\mathcal{WB}_{4}(O)$. By convention $\mathcal{WB}_{0}(O)=O_{0}$ (see Example \ref{G8}) and the family $\{\mathcal{WB}_{k}(O)\}$ produces the following filtration of $\mathcal{WB}(O)$:
\begin{equation}\label{C1}
\xymatrix{
\mathcal{WB}_{0}(O)\ar[r] & \mathcal{WB}_{1}(O) \ar[r] &  \cdots \ar[r] & \mathcal{WB}_{k-1}(O) \ar[r] & \mathcal{WB}_{k}(O) \ar[r] & \cdots \ar[r] & \mathcal{WB}(O).
}
\end{equation}

\begin{thm}\label{D8}
Assume that $O$ is a well pointed $\Sigma$-cofibrant operad. The objects $\mathcal{WB}(O)$ and $T_{k}(\mathcal{WB}_{k}(O))$ are cofibrant replacements of $O$ and $T_{k}(O)$ in the categories $Bimod_{O}$ and $T_{k}Bimod_{O}$ respectively.
\end{thm}

\begin{proof}
We start by showing that the map $\tilde{\mu}:\mathcal{WB}(O)\rightarrow O$, which sends the real numbers indexing the vertices of the main tree to $0$, is a weak equivalence in the category of $O$-bimodules. It is sufficient to prove that $\tilde{\mu}$ is a weak equivalence in the category of sequences. For this purpose, we consider the sequence $\mathcal{WB}'(O)$ obtained from the coproduct (\ref{B9}) by taking the equivalence relation generated by the axioms $(i)$, $(ii)$ and $(iii)$ of Construction \ref{C0} as well as the relation defined as follows: if a vertex is labelled by $(x\,;\,0)$, with $x\in \mathcal{BV}(O)$, then $(x\,;\,0)$ is identified with $(\iota\circ\mu(x)\,;\,0)$. By construction, $\mathcal{WB}'(O)$ is just the a left module over $O$ and one has the following maps of sequences:
$$
\xymatrix{
\mathcal{WB}'(O)\ar[r]^{f} & \mathcal{WB}(O)\ar[r]^{\tilde{\mu}} & O,
}
$$
where $f$ is the quotient map coming from the equivalence relation defined as follows: if a vertex is labelled by $(x\,;\,1)$, with $x\in \mathcal{BV}(O)$, then $(x\,;\,1)$ is identified with $(\iota\circ\mu(x)\,;\,1)$. The composite map $\tilde{\mu}\circ f$ is a homotopy equivalence in which the homotopy consists in bringing the real numbers indexing the vertices of the main tree to $0$. Let $x=[T\,;\,\{t_{v}\}\,;\,\{x_{v}\}]$ be a point in $\mathcal{WB}(O)$. Due to the axiom $(iii)$ of Construction \ref{C0}, we can assume that the representative point doesn't have two consecutive vertices indexed by $1$. Then, we denote by $Max(x)$ the set of vertices of the main tree indexed by $1$. Thus, the fiber of  $f$ over $x$ is the space
$$
Fib(f\,;\,x):=\left\{
\{x'_{v}\}\in \underset{v\in Max(x)}{\prod} \mathcal{BV}(O)(|v|) \,\big| \, \mu(x'_{v})=\mu(x_{v})
\right\} .
$$
The fiber is contractible and the homotopy consists in bringing the real numbers indexing the inner edges to $0$. As a consequence, the map $f:\mathcal{WB}'(O)\rightarrow \mathcal{WB}(O)$ is a weak equivalence. By using the \textbf{2to3} axiom in model category theory, we conclude that the map $\tilde{\mu}$ is also a weak equivalence in the category of sequences. Similarly, we can show that the map from $T_{k}(\mathcal{WB}_{k}(O))$ to $T_{k}(O)$ is a weak equivalence.

From now on, we prove that the filtration (\ref{C1}) is composed of cofibrations in the category of $O$-bimodules. For this purpose, we consider another filtration according to the number of vertices which is compatible with the filtration (\ref{G5}) introduced for the derived mapping space of operads. Indeed, for each inclusion $\mathcal{WB}_{k-1}(O)\rightarrow \mathcal{WB}_{k}(O)$, there is another filtration:
$$
\xymatrix@C=20pt{
\mathcal{WB}_{k-1}(O)\ar[r] & \mathcal{WB}_{k}(O)[1]\ar[r] & \cdots\ar[r] & \mathcal{WB}_{k}(O)[l-1]\ar[r] & \mathcal{WB}_{k}(O)[l]\ar[r] & \cdots \ar[r] & \mathcal{WB}_{k}(O) 
}
$$
More precisely, let $\Upsilon_{k}[l]$ be the set of elements $[T\,;\,\{T_{v}\}]\in\Upsilon$ having exactly $k$ geometrical inputs and such that the sum of the vertices of auxiliary trees $\sum_{v}|V(T_{v})|$ is equal to $l$. For instance, the element indexing the point in Figure \ref{C2} is in $\Upsilon_{9}[8]$. Then, the sequence $Y_{k}[l]$ is the quotient of the sub-sequence
\begin{equation}\label{D4}
\left.
\underset{[T\,;\,\{T_{v}\}]\in \Upsilon_{k}[l]}{\coprod}\,\,\,\underset{v\in V(T)}{\prod}\, [0\,,\,1]\times \big( \, \underset{v'\in V(T_{v})}{\prod}\, O(|v'|)\times \underset{e\in E^{int}(T_{v})}{\prod}\, [0\,,\,1]\,\big)\,
\right/ \sim
\end{equation}
coming from the restriction on the real numbers indexing the vertices of the main tree. The equivalent relation is generated by the compatibility with the symmetric group axioms of Construction \ref{e7} and \ref{C0} as well as the axiom $(iii)$ of Construction \ref{C0}. Indeed, these two relations are the only ones which don't necessarily decrease the sum of the vertices of the auxiliary trees. The sequence $\partial Y_{k}[l]$ is formed by points in $Y_{k}[l]$ satisfying one of the following conditions:
\begin{itemize}
\item[$\blacktriangleright$] there is a vertex of the main tree indexed by $0$ or $1$,
\item[$\blacktriangleright$] there is an inner edge of an auxiliary tree  indexed by $0$,
\item[$\blacktriangleright$] there is a bivalent vertex labelled by the unit of the operad $O$.
\end{itemize}
First, we define $\mathcal{WB}_{1}(O)[1]$ from $O_{0}$ (see Example \ref{G8}) by using a sequence of pushout diagrams. Indeed, let $Y_{0}^{\ast}$ be the sequence given by $Y_{0}^{\ast}(0)=O(0)$, $Y_{0}^{\ast}(1)=\ast$ and the empty set otherwise. The inclusion, $O_{0}\rightarrow Y_{0}^{\ast}$ is obviously a cofibration and we consider the following pushout diagrams:
\begin{equation}\label{D5}
 \xymatrix{
\mathcal{F}_{B}(O_{0}) \ar[r] \ar@{=}[d] & \mathcal{F}_{B}(Y_{0}^{\ast}) \ar@{=}[d] \\
O_{0} \ar[r] & O_{0}^{+}
} 
\hspace{15pt}\text{and}\hspace{15pt}
\xymatrix{
\mathcal{F}_{B}(\partial Y_{1}[1]) \ar[r] \ar[d] & \mathcal{F}_{B}(Y_{1}[1]) \ar[d] \\
O_{0}^{+} \ar[r] & \mathcal{WB}_{1}(O)[1]
} 
\end{equation}
Roughly speaking, the based point in $Y_{0}^{\ast}(1)$ corresponds with the class of points in $\mathcal{WB}(O)$ indexed by the corolla labelled by the pairs $(\iota(\ast_{1})\,;\, t)$, with $t\in [0\,,\,1]$ and $\ast_{1}$ the unit of the operad $O$. For $(k\,;\,l)\neq (1\,;\,1)$, the sequences $Y_{k}[l]$ and $\partial Y_{k}[l]$ are not objects in the under category $O_{0}\downarrow Seq$. In order to use the free bimodule functor (see Example \ref{G8}), we consider the sequences $\tilde{Y}_{k}[l]$ and $\partial \tilde{Y}_{k}[l]$ obtained as follows:
$$
\tilde{Y}_{k}[l](n):=
\left\{
\begin{array}{ll}\vspace{4pt}
Y_{k}[l](0)\sqcup O(0) & \text{if } n=0, \\ 
Y_{k}[l](n) & \text{otherwise},
\end{array} 
\right.
\hspace{15pt}\text{and}\hspace{15pt}
\partial\tilde{Y}_{k}[l](n):=
\left\{
\begin{array}{ll}\vspace{4pt}
\partial Y_{k}[l](0)\sqcup O(0) & \text{if } n=0, \\ 
\partial Y_{k}[l](n) & \text{otherwise}.
\end{array} 
\right.
$$ 
Then, we introduce the following pushout diagram: 
$$
\xymatrix{
\mathcal{F}_{B}(\partial \tilde{Y}_{k}[l]) \ar[r] \ar[d] & \mathcal{F}_{B}(\tilde{Y}_{k}[l]) \ar[d] \\
\mathcal{WB}_{k}(O)[l-1] \ar[r] & \mathcal{WB}_{k}(O)[l]
}
$$ 
Similarly to the proof \cite[Theorem 2.12]{Ducoulombier16}, we can check by induction on the number of vertices of the main tree that the inclusion from $\partial Y_{k}[l]$ to $Y_{k}[l]$ is a $\Sigma$-cofibration. Since the functor $\mathcal{F}_{B}$ and the pushout diagrams preserve the cofibrations, the map $\mathcal{WB}_{k}(O)[l-1]\rightarrow \mathcal{WB}_{k}(O)[l]$ as well as the inclusion $\mathcal{WB}_{k-1}(O)\rightarrow\mathcal{WB}_{k}(O)$ are  cofibrations in the category of $O$-bimodules.
\end{proof}

According to the notation introduced in the previous proof, there is a filtration of the resolution $\mathcal{WB}(O)$, denoted by $\{\mathcal{WB}_{k}(O)[l]\}_{k\,;\,l}$, depending on the number $k$ of geometrical inputs and the sum $l$ of the vertices of the auxiliary trees. Since the maps from $\mathcal{WB}_{k}(O)[l-1]$ to $\mathcal{WB}_{k}(O)[l]$ are cofibrations in the category of bimodules over $O$, this filtration produces a tower of fibrations computing the mapping space $Bimod_{O}(\mathcal{WB}(O)\,;\,O')$. More precisely, the vertical maps of the following diagrams are fibrations:
$$
\xymatrix{
Bimod_{O}(\mathcal{WB}_{1}(O)[1]\,;\, O') \ar[r] \ar[d] & Seq(Y_{1}[1]\,;\, O')\ar[d] \\
Bimod_{O}(O_{0}^{+}\,;\, O') \ar[r] & Seq( \partial Y_{1}[1] \,;\, O')
}
\hspace{7pt}\text{and}\hspace{7pt}
\xymatrix{
Bimod_{O}(\mathcal{WB}_{k}(O)[l]\,;\, O') \ar[r] \ar[d] & Seq(Y_{k}[l]\,;\, O')\ar[d] \\
Bimod_{O}(\mathcal{WB}_{k}(O)[l-1]\,;\, O') \ar[r] & Seq( \partial Y_{k}[l] \,;\, O')
}
$$
Furthermore, if $g\in Bimod_{O}(\mathcal{WB}_{k}(O)[l-1]\,;\, O')$, then the fiber over $g$ is homeomorphic to the mapping space of sequences from $Y_{k}[l]$ to $O'$ such that the restriction to sub-sequence $\partial Y_{k}[l]$ coincides with the map induced by $g$:
\begin{equation}\label{C7}
Seq^{g}\big(\,(Y_{k}[l]\,,\, \partial Y_{k}[l])\,;\, O'\,\big).
\end{equation}

Finally, we give a description of the space $Bimod_{O}(\mathcal{WB}_{1}(O)[1]\,;\, O')$. Let us remark that the points in $Y_{1}[1](0)$ are corollas indexed by pairs $(\iota(x)\,;\,t)$, with $x\in O(0)$ and $t\in [0\,,\,1]$. Similarly, the points in $Y_{1}[1](1)$ are corollas indexed by pairs $(\iota(x)\,;\,t)$, with $x\in O(1)$ and $t\in [0\,,\,1]$. Then, by using the pushout diagrams (\ref{D5}), we deduce that a point in the mapping space of bimodules is given by a pair of continuous maps:
\begin{equation}\label{D6}
h_{0}:O(0)\times [0\,,\,1]\longrightarrow O'(0) 
\hspace{15pt}\text{and}\hspace{15pt}
h_{1}:O(1)\times [0\,,\,1]\longrightarrow O'(1),
\end{equation}
such that, $\forall t\in [0\,,\,1]$, $h_{1}(\ast_{1}\,;\,t)$ is equal to a constant denoted by $h_{1}(\ast)$ (which is not necessarily the unit of the operad $O'$). Moreover, the pair $(h_{0}\,;\,h_{1})$ has to satisfy the following relations:
\begin{equation}\label{H0}
\left\{\begin{array}{l}\vspace{4pt}
h_{0}(x\,;\,0)=\eta(x), \\ 
h_{0}(x\,;\,1)=h_{1}(\ast)\circ^{1}\eta(x),
\end{array}\right. 
\hspace{15pt} \text{and}\hspace{15pt} 
\left\{\begin{array}{l}\vspace{4pt}
h_{1}(x\,;\,0)=\eta(x)\circ_{1}h_{1}(\ast), \\ 
h_{1}(x\,;\,1)=h_{1}(\ast)\circ^{1}\eta(x).
\end{array} \right.
\end{equation}  

Analogously to the operadic case, from a $k$-truncated bimodule $M_{k}$, we consider the $k$-free bimodule $\mathcal{F}_{B}^{k}(M_{k})$ whose $k$ first components coincide with $M_{k}$. According to the notation introduced in Example \ref{G8}, the functor $\mathcal{F}_{B}^{k}$ can be expressed as a quotient of the free bimodule functor in which the equivalence relation is generated by the following axiom: if the sum of the incoming edges of any two consecutive vertices connected by an inner edge $e$ is smaller than $k+1$, then we contract $e$ using $k$-truncated bimodule structure of $M_{k}$. In particular, $\mathcal{F}_{B}^{k}$ is left adjoint to the truncated functor $T_{k}$ and one has $\mathcal{F}_{B}^{k}(T_{k}(\mathcal{WB}_{k}(O)))=\mathcal{WB}_{k}(O)$ because $\mathcal{WB}_{k}(O)$ is defined as the sub-bimodule of $\mathcal{WB}(O)$ generated by its $k$ first components. Consequently, there are the following identifications:
$$
T_{k}Bimod_{O}^{h}(T_{k}(O)\,;\,T_{k}(O'))\cong T_{k}Bimod_{O}(T_{k}(\mathcal{WB}_{k}(O))\,;\,T_{k}(O')) \cong Bimod_{O}(\mathcal{WB}_{k}(O)\,;\,O').
$$

\subsection{The $\mathcal{C}_{1}$-algebra structure on the space $Bimod_{O}(\mathcal{WB}(O)\,;\,O')$}

In what follows, $\eta:O\rightarrow O'$ is a map of operads and we introduce a $\mathcal{C}_{1}$-algebra structure on the space $Bimod_{O}(\mathcal{WB}(O)\,;\,O')$. Let us mention that the operad $O$ is not supposed to be well pointed and $\Sigma$-cofibrant in this section. These assumptions are only needed to get a model for the derived mapping space $Bimod_{\mathcal{C}_{d}}^{h}(O\,;\,O')$. So, the purpose is to define a family of continuous maps
$$
\alpha_{n}:\mathcal{C}_{1}(n)\times Bimod_{O}(\mathcal{WB}(O)\,;\,O')^{\times n}\longrightarrow Bimod_{O}(\mathcal{WB}(O)\,;\,O')
$$ 
compatible with the operadic structure of $\mathcal{C}_{1}$. From now on, we fix a configuration $c=<c_{1},\ldots, c_{n}>\in \mathcal{C}_{1}(n)$ as well as a family of bimodule maps $f_{i}:\mathcal{WB}(O)\rightarrow O'$. Since the little cubes arise from an affine embedding, $c_{i}$ is determined by the image of $0$ and $1$. In a similar way, we define the linear embeddings $h_{i}:[0\,,\,1]\rightarrow [0\,,\,1]$, with $0\leq i \leq n$, representing the gaps between the cubes:
$$
h_{i}(0)=
\left\{
\begin{array}{cc}
0 & \text{if } i=0, \\ 
c_{i}(1) & \text{if } i\neq 0,
\end{array} 
\right.
\hspace{15pt}\text{and}\hspace{15pt}
h_{i}(1)=
\left\{
\begin{array}{cc}
1 & \text{if } i=n, \\ 
c_{i+1}(0) & \text{if } i\neq n.
\end{array} 
\right.
$$

The bimodule map $\alpha_{n}(c\,;\,f_{1},\cdots,f_{n})$ is defined by using a decomposition of the points $y=[T\,;\, \{t_{v}\}\,;\,\{x_{v}\}]\in \mathcal{WB}(O)$ according to the parameters indexing the vertices. Roughly speaking, the little cubes $<c_{1},\ldots,c_{n}>$ subdivides the tree $T$ into sub-trees as shown in Figure \ref{G0}. Then, we apply the bimodule map $f_{i}$ to the sub-trees associated to the little cube $c_{i}$ and the composite map $\eta\circ\mu:\mathcal{WB}(O)\rightarrow O\rightarrow O'$ to the sub-trees associated to gaps. Finally, we put together the pieces by using the operadic structure of $O'$. Due to the axioms of Construction \ref{C0}, we can suppose that the representative point $y$ doesn't have two consecutive vertices indexing by the same real. For the moment, we also assume that the tree $T$ is planar. \vspace{10pt}

\begin{figure}[!h]
\begin{center}
\includegraphics[scale=0.3]{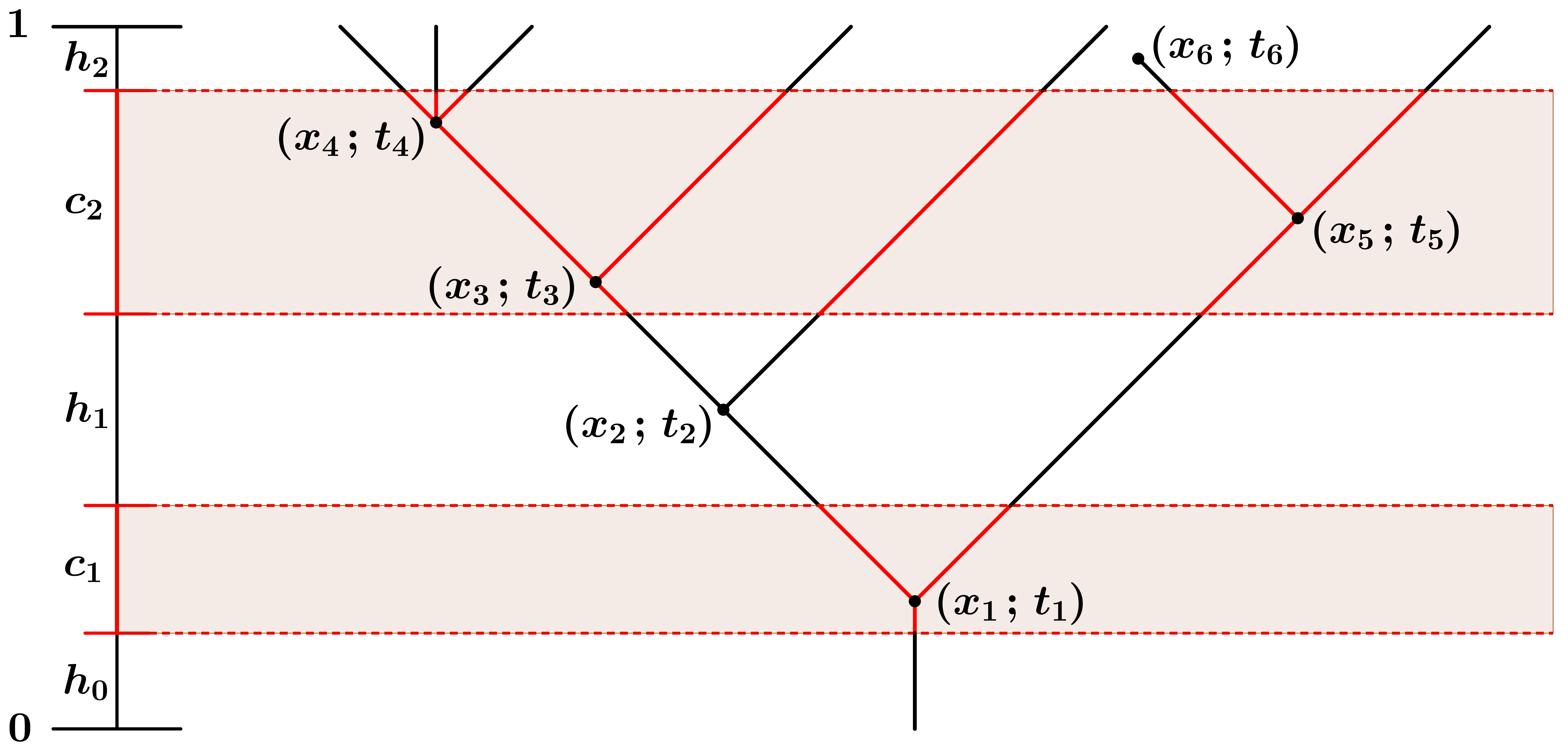}
\caption{Illustration of the subdivision of a point in $\mathcal{WB}(O)$ with the conditions\\  $c_{1}(0)<t_{1}<c_{1}(1)<t_{2}<c_{2}(0)<t_{3},t_{4},t_{5}<c_{2}(1)<t_{6}$.}\label{G0}\vspace{-25pt}
\end{center}
\end{figure}

\newpage

More precisely, a sub-point of $y=[T\,;\,\{t_{v}\}\,;\, \{x_{v}\}]$ is an element in $\mathcal{WB}(O)$ obtained from $y$ by taking a sub-tree of $T$ having the same indexation. A sub-point $w$ is said to be associated to the gap $h_{i}$ if the vertices below $w$ (seen as a sub-point of $y$) are strictly smaller than $h_{i}(0)$ whereas the vertices above $w$ have to be strictly bigger than $h_{i}(1)$. Furthermore, the parameters indexing the vertices of the main tree of $w$ are in the interval $[h_{i}(0)\,,\,h_{i}(1)]$. The set $\mathcal{T}[h_{i}\,;\,y]=\{w_{1}^{i},\ldots,w_{p_{i}}^{i}\}$ of sub-points associated the gap $h_{i}$ is ordered using the planar structure of the tree $T$. For instance, the sets $\mathcal{T}[h_{0}\,;\,y]$, $\mathcal{T}[h_{1}\,;\,y]$ and $\mathcal{T}[h_{2}\,;\,y]$ associated to the point in Figure \ref{G0} are the following ones:
\begin{figure}[!h]
\begin{center}
\includegraphics[scale=0.175]{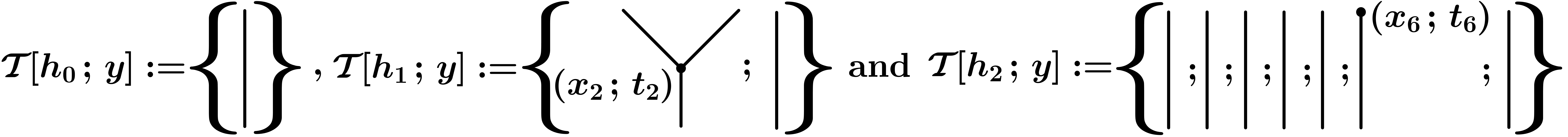}\vspace{-10pt}
\end{center}
\end{figure}

\noindent Similarly, a sub-point $z$ is said to be associated to the little cube $c_{i}$ if the vertices below $z$ (seen as a sub-point in $y$) are smaller than $c_{i}(0)$ whereas the vertices above $z$ have to be bigger than $c_{i}(1)$. Furthermore, the parameters indexing the vertices of the main tree of $z$ are in the interval $]c_{i}(0)\,,\,c_{i}(1)[$. The set $\mathcal{T}[c_{i}\,;\,y]=\{z_{1}^{i},\ldots,z_{q_{i}}^{i}\}$ of sub-points associated the little cube $c_{i}$ is ordered using the planar structure of the tree $T$. For instance, the sets $\mathcal{T}[c_{1}\,;\,y]$ and $\mathcal{T}[c_{2}\,;\,y]$ associated to the point $y$ in Figure \ref{G0} are the following ones:
\begin{figure}[!h]
\begin{center}
\includegraphics[scale=0.175]{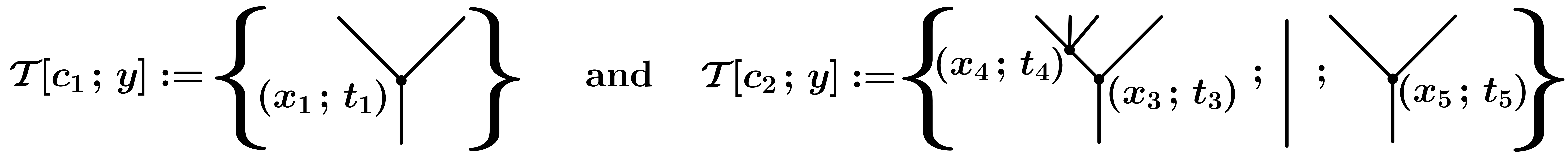}\vspace{-10pt}
\end{center}
\end{figure}

\noindent Let us remark that we really need the trivial trees (which represent the corolla indexed by the unit of the unit of the operad $O$) in the above definition since the bimodule maps $\{f_{i}\}$ don't necessarily map the trivial tree to the unit of the operad $O'$. Furthermore, we need an application rescaling the parameters of the sub-points:
\begin{equation}\label{G9}
c_{i}^{\ast}:\mathcal{T}[c_{i}\,;\,y]\longrightarrow \mathcal{WB}(O)\,\,;\,\, [T'\,;\,\{t'_{v}\}\,;\, \{x'_{v}\}]\longmapsto [T'\,;\,\{c_{i}^{-1}(t'_{v})\}\,;\, \{x'_{v}\}].
\end{equation}
The map is well defined since the parameters indexing the vertices of the elements in $\mathcal{T}[c_{i}\,;\,y]$ are in the interval $]c_{i}(0)\,,\,c_{i}(1)[$. From the operadic structure of $O'$, we build the map $\alpha_{n}(c\,;\,f_{1},\cdots,f_{n})$ by induction as follows: 
$$
\begin{array}{lcl}\vspace{4pt}
\alpha_{n}(c\,;\,f_{1},\cdots,f_{n})_{0}(y) & = & \eta\circ\mu(w_{1}^{0}), \\ 
\alpha_{n}(c\,;\,f_{1},\cdots,f_{n})_{1}(y) & = & \alpha_{n}(c\,;\,f_{1},\cdots,f_{n})_{0}(y)\big(\,f_{1}(c_{1}^{\ast}(z_{1}^{1})),\ldots,f_{1}(c_{1}^{\ast}(z_{q_{1}}^{1}))\,\big),  \\ 
 & \vdots & \\ \vspace{4pt}
\alpha_{n}(c\,;\,f_{1},\cdots,f_{n})_{2k}(y) & = & \alpha_{n}(c\,;\,f_{1},\cdots,f_{n})_{2k-1}(y)\big(\,\eta\circ\mu(w_{1}^{k}) ,\ldots,\eta\circ\mu(w_{p_{k}}^{k}) \,\big), \\ 
\alpha_{n}(c\,;\,f_{1},\cdots,f_{n})_{2k+1}(y) & = & \alpha_{n}(c\,;\,f_{1},\cdots,f_{n})_{2k}(y)\big(\,f_{k}(c_{k}^{\ast}(z_{1}^{k})),\ldots,f_{k}(c_{k}^{\ast}(z_{q_{k}}^{k}))\,\big), \\ 
 & \vdots &  \\ 
\alpha_{n}(c\,;\,f_{1},\cdots,f_{n})(y) & = & \alpha_{n}(c\,;\,f_{1},\cdots,f_{n})_{2n-1}(y)\big(\,\eta\circ\mu(w_{1}^{n}) ,\ldots,\eta\circ\mu(w_{p_{n}}^{n}) \,\big).
\end{array} 
$$
We don't need to rescale  the sub-points associated to gaps since the map $\mu:\mathcal{WB}(O)\rightarrow O$ sends all the parameters indexing the vertices to $0$. If the tree $(T\,;\,\sigma)$ indexing the element $y$ is not planar (in other words $\sigma\neq id$), then the application $\alpha_{n}$ is given by  the formula
$$
\alpha_{n}(c\,;\,f_{1},\cdots,f_{n})([(T\,;\,\sigma)\,;\,\{t_{v}\}\,;\,\{x_{v}\}])=\alpha_{n}(c\,;\,f_{1},\cdots,f_{n})([(T\,;\,id)\,;\,\{t_{v}\}\,;\,\{x_{v}\}])\cdot\sigma.
$$

The reader can check that the family of maps $\{\alpha_{n}\}$ is well defined and compatible with the operadic structure of $\mathcal{C}_{1}$. Furthermore, this construction produces also a $\mathcal{C}_{1}$-algebra structure on the mapping space of truncated bimodules $T_{k}Bimod_{O}(T_{k}(\mathcal{WB}_{k}(O))\,;\,T_{k}(O'))$ because the sub-points of an element in $T_{k}(\mathcal{WB}_{k}(O))$ are still elements in $T_{k}(\mathcal{WB}_{k}(O))$ and the rescaling maps (\ref{G9}) decrease the number of geometrical inputs. As an example, if we denote by $z_{1}^{2}$ the sub-point of the element in Figure \ref{G0} generated by the vertices indexed by $(x_{3}\,;\,t_{3})$ and $(x_{4}\,;\,t_{4})$, then the image is given by\vspace{5pt}

$$
\big(\big(\underset{c_{1}}{\underbrace{f_{1}(x_{1}\,;\,t_{1})}}(\underset{h_{1}}{\underbrace{\eta(x_{2})\,,\,\ast_{1}'}})\big)\big( \underset{c_{2}}{\underbrace{f_{2}(c_{2}^{\ast}(z_{1}^{2}))\,,\,f_{2}(\iota(\ast_{1}))\,,\, f_{2}(x_{5}\,;\,t_{5})}}\big)\big)\big( \underset{h_{2}}{\underbrace{\ast_{1}'\,,\,\ast_{1}'\,,\,\ast_{1}'\,,\,\ast_{1}'\,,\,\ast_{1}'\,,\,\eta(x_{6})\,,\,\ast_{1}'}}\big).
$$

\newpage

\section{Delooping derived mapping space of bimodules}

In the previous section, we introduce a cofibrant replacement $\mathcal{WB}(O)$ of an operad $O$ in the category of bimodules over itself. By using its properties, we have been able to define a $\mathcal{C}_{1}$-algebra structure on the space $Bimod_{O}(\mathcal{WB}(O)\,;\,O')$ from a map of operads $\eta:O\rightarrow O'$. In what follows, we give an explicit description of the loop space associated to the mapping space of bimodules. More precisely, one has the following theorem where both standard and truncated versions are considered.

\begin{thm}\label{B5}
Let $O$ be a well pointed $\Sigma$-cofibrant operad and $\eta:O\rightarrow O'$ be a map of operads. If the space $O'(1)$ is contractible, then there are explicit weak equivalences of $\mathcal{C}_{1}$-algebras:
$$
\begin{array}{rcl}\vspace{4pt}
\xi: \Omega Operad^{h}(O\,;\,O') & \longrightarrow & Bimod_{O}^{h}(O\,;\,O') , \\ 
\xi_{k}:\Omega \big( T_{k}Operad^{h}(T_{k}(O)\,;\,T_{k}(O'))\big) & \longrightarrow & T_{k}Bimod_{O}^{h}(T_{k}(O)\,;\,T_{k}(O')) .
\end{array} 
$$
\end{thm}

By using the resolutions $\mathcal{BV}(O)$ and $T_{k}(\mathcal{BV}_{k}(O))$ for "truncated" operads as well as the resolutions $\mathcal{WB}(O)$ and $T_{k}(\mathcal{WB}_{k}(O))$ for "truncated" bimodules, we can easily define the map $\xi$ and $\xi_{k}$ by induction on the number of vertices of the main tree. First of all, we recall that a point in the loop space $\Omega Operad(\mathcal{BV}(O)\,;\,O')$, based on $\eta\circ\mu:\mathcal{BV}(O)\rightarrow O\rightarrow O'$, is given by a family of $\Sigma$-invariant continuous maps
$$
g_{n}:\mathcal{BV}(O)(n)\times [0\,,\,1]\longrightarrow O'(n),\hspace{15pt}\forall n\geq 0,
$$
satisfying the following conditions:
\begin{itemize}
\item[$\blacktriangleright$] $g_{n}(\iota(\ast_{1})\,;\, t)=\ast_{1}'$\hspace{93pt} $\forall t\in [0\,,\,1]$,
\item[$\blacktriangleright$] $g_{n}(x\circ_{i}y\,;\, t)=g_{l}(x\,;\, t)\circ_{i}g_{n-l+1}(y\,;\, t)$\hspace{15pt} $\forall t\in [0\,,\,1]$, $x\in \mathcal{BV}(O)(l)$ and $y\in \mathcal{BV}(O)(n-l+1)$,
\item[$\blacktriangleright$] $g_{n}(x\,;\, t)=\eta\circ \mu(x)$\hspace{83pt} $t\in \{0\,;\,1\}$ and $x\in \mathcal{BV}(O)(n)$.\vspace{3pt}
\end{itemize}

Let $g=\{g_{n}\}$ be a point in the loop space and let $[(T\,;\,\sigma)\,;\,\{t_{v}\}\,;\,\{x_{v}\}]$ be a point in $\mathcal{WB}(O)$ where $\sigma$ is the permutation labelling the leaves of $T$. If the tree $T$ has only one vertex indexed by the pair $(x_{r}\,;\,t_{r})$, with $x_{r}\in \mathcal{BV}(O)$ and $t_{r}\in [0\,,\,1]$, then $\xi(g)$ is given by the formula
$$
\xi(g)([(T\,;\,\sigma)\,;\,\{t_{v}\}\,;\,\{x_{v}\}])=g_{|T|}(x_{r}\,;\,t_{r})\cdot \sigma.
$$
From now on, we assume that the map $\xi(g)$ is well defined  points indexed by main trees having at most $k$ vertices. Let $[(T\,;\,\sigma)\,;\,\{t_{v}\}\,;\,\{x_{v}\}]$ be a point in $\mathcal{WB}(O)$ with $|V(T)|=k+1$. In that case, the planar tree $(T\,;\,id)$ has a decomposition of the form $T_{1}\circ_{i}T_{2}$ where $T_{1}$ and $T_{2}$ are two planar trees having at most $k$ vertices. If we denote by $[(T_{1}\,;\,id)\,;\,\{t^{1}_{v}\}\,;\,\{x^{1}_{v}\}]$ and $[(T_{2}\,;\,id)\,;\,\{t^{2}_{v}\}\,;\,\{x^{2}_{v}\}]$ the points obtained from $[(T\,;\,\sigma)\,;\,\{t_{v}\}\,;\,\{x_{v}\}]$ by taking the induced indexations, then $\xi(g)$ is given by the formula
\begin{equation}\label{H1}
\xi(g)([(T\,,\,\sigma)\,;\,\{t_{v}\}\,;\,\{x_{v}\}])=\big( \xi(g)([(T_{1}\,,\,id)\,;\,\{t^{1}_{v}\}\,;\,\{x^{1}_{v}\}])\circ_{i}\xi(g)([(T_{2}\,,\,id)\,;\,\{t^{2}_{v}\}\,;\,\{x^{2}_{v}\}])\big)\cdot\sigma.
\end{equation}
For instance, the image of the point $[T\,;\,\{t_{v}\}\,;\,\{x_{v}\}]$ in Figure \ref{C3} is the following one:
$$
\begin{array}{cl}\vspace{4pt}
\xi(g)([T\,;\,\{t_{v}\}\,;\,\{x_{v}\}]) & = g_{2}(x_{r}\,;\,t_{r})\circ(g_{2}(x_{1}\,;\, t_{r})\,;\,g_{3}(x_{2}\,;\,t_{2})),  \\ 
 & =g_{3}(x_{r}\circ_{1}x_{1}\,;\,t_{r})\circ_{3}g_{3}(x_{2}\,;\, t_{2}).
\end{array} 
$$

In the same way, we define the map $\xi_{k}$ in the context of truncated operads and bimodules. By construction, the map $\xi$ and $\xi_{k}$ are morphisms of $\mathcal{C}_{1}$-algebras and they produce a morphism between the two towers of fibrations introduced in Sections \ref{F5} and \ref{C5}:
$$
\xymatrix@C=15pt{
\Omega Operad(\mathcal{BV}_{1}(O)[1]\,;\,O') \ar[d]_{\xi_{1}[1]} & \cdots \ar[l] & \Omega Operad(\mathcal{BV}_{k}(O)[l-1]\,;\,O') \ar[l] \ar[d]_{\xi_{k}[l-1]} & \Omega Operad(\mathcal{BV}_{k}(O)[l]\,;\,O') \ar[l] \ar[d]_{\xi_{k}[l]} & \cdots \ar[l]\\
Bimod_{O}(\mathcal{WB}_{1}(O)[1]\,;\,O') & \cdots \ar[l] & Bimod_{O}(\mathcal{WB}_{k}(O)[l-1]\,;\,O') \ar[l] & Bimod_{O}(\mathcal{WB}_{k}(O)[l]\,;\,O') \ar[l]& \cdots \ar[l]
}
$$
Since the horizontal maps are fibrations, the map $\xi$ is a weak equivalence if the induced maps $\xi_{k}[l]$ are weak equivalences. Similarly, the map $\xi_{k}$ is a weak equivalence if the maps $\xi_{i}[l]$, with $i\leq k$, are weak equivalences. So, the rest of this section is devoted to prove by induction that the vertical maps of the above diagram are weak equivalences. 

\subsection{Initialization: The map $\xi_{1}[1]$ is a weak equivalence}

According to the notation introduced in Section \ref{C5}, we recall that the mapping space $Bimod_{O}(\mathcal{WB}_{1}(O)[1]\,;\,O')$ is formed by pairs $(h_{0}\,;\,h_{1})$ of the form
$$
h_{0}:O(0)\times [0\,,\,1]\longrightarrow O'(0)
\hspace{15pt}\text{and}\hspace{15pt}
h_{1}:O(1)\times [0\,,\,1]\longrightarrow O'(1),
$$
satisfying the relations (\ref{H0}). In particular, the map $h_{1}(\ast_{1}\,;\,-):[0\,,\,1]\rightarrow O'(1)$ is a constant map which is not necessarily equal to the unit $\ast_{1}'$ of the operad $O'$. Nevertheless, the image of the inclusion $\xi_{1}[1]$ is formed by pairs satisfying the relation $h_{1}(\ast_{1}\,;\,t)=\ast_{1}'$ with $t\in [0\,,\,1]$. In other words, the image can be expressed as the fiber
$$
Fib\big( \, \alpha : Bimod_{O}(\mathcal{WB}_{1}(O)[1]\,;\,O') \longrightarrow Bimod_{O}(O_{0}^{+}\,;\,O')\,\big),
$$ 
over the map sending the based point $\ast\in O_{0}^{+}(1)$ to $\ast_{1}'$. As shown in the proof of Theorem \ref{D8}, $\alpha$ is a fibration and the fiber is weakly equivalent to its homotopy fiber. Furthermore, the homotopy fiber is weakly equivalent to $Bimod_{O}(\mathcal{WB}_{1}(O)[1]\,;\,O')$ since the space $Bimod_{O}(O_{0}^{+}\,;\,O')$ is contractible. Indeed, due to the left pushout diagram (\ref{D5}), one has the following identifications:
$$
Bimod_{O}(O_{0}^{+}\,;\,O')\cong O_{0}\downarrow Seq(Y_{0}^{\ast}\,;\,O') = Top(Y_{0}^{\ast}(1)\,;\, O'(1)) = O'(1)\simeq \ast.
$$
Thus proves that the map $\xi_{1}[1]$ is a weak equivalence.

\subsection{Induction: Simplification of the problem}\label{H7}

From now on, we assume that the map $\xi_{k}[l-1]$ is a weak equivalence. In this subsection, we show that $\xi_{k}[l]$ is also a weak equivalence if and only if a specific map of sequences is a homotopy equivalence. The latter is easier to understand and can be expressed using the language of diagrams. Let $g$ be a point in the loop space $\Omega Operad(\mathcal{BV}_{k}(O)[l-1]\,;\,O')$. We denote the fiber over $g$ by $F_{1}$ and the fiber over $\xi_{k}[l-1](g)$ by $F_{2}$. In other words, one has the following diagram:
$$
\xymatrix{
\Omega Operad(\mathcal{BV}_{k}(O)[l-1]\,;\,O')\ar[d]_{\xi_{k}[l-1]}^{\simeq} & \Omega Operad(\mathcal{BV}_{k}(O)[l]\,;\,O')\ar[d]_{\xi_{k}[l]}\ar[l] & F_{1} \ar[d]_{\xi^{g}}\ar[l] \\
Bimod_{O}(\mathcal{WB}_{k}(O)[l-1]\,;\,O') & Bimod_{O}(\mathcal{WB}_{k}(O)[l]\,;\,O')\ar[l] & F_{2}\ar[l]
}
$$
Since the left horizontal maps are fibrations, $\xi_{k}[l]$ is a weak equivalence if the map $\xi^{g}$ between the fibers is a weak equivalence. From the identifications (\ref{C6}) and (\ref{C7}), there is the following diagram:
\begin{equation}\label{C8}
\xymatrix@R=17pt{
F_{1}\ar[d]_{\xi^{g}} &  Seq^{g}\big(\,(X_{k}[l]\times [0\,,\,1]\,,\, \partial'X_{k}[l])\,;\, O'\,\big) \ar[d] \ar[l]_{\hspace{-50pt}\cong} \\
F_{2} & Seq^{\xi_{k}[l-1](g)}\big(\,(Y_{k}[l]\,,\, \partial Y_{k}[l])\,;\, O'\,\big) \ar[l]_{\hspace{-50pt}\cong}
}
\end{equation}

First, one has to express the fiber $F_{1}$ as a subspace of the fiber $F_{2}$. For this purpose, let us remark that there is a map from $X_{k}[l]\times [0\,,\,1]$ to $Y_{k}[l]$ sending a pair $(x\,;\,t)$ to the corolla indexed by $(x\,;\,t)$. For this reason, in order to describe $\partial'X_{k}[l]$ as a sub-sequence of $Y_{k}[l]$, we consider the set $\phi_{k}[l]$ of elements $[T\,;\,\{T_{v}\}]\in \Upsilon_{k}[l]$ (having $k$ geometrical inputs and such that $\sum_{v}|V(T_{v})|=l$ ) in which $T$ is not a corolla. Then, we introduce the following sequences:
\begin{itemize}[leftmargin=11pt]
\item[$\blacktriangleright$] $Y_{k}^{(1)}[l]$ is the quotient of the restriction of the coproduct (\ref{D4}) to the set $\phi_{k}[l]$. The equivalence relation is generated by the compatibility with the symmetric group axioms of Constructions \ref{e7} and \ref{C0} as well as the relation defined as follows: if an inner edge $e$ of the main tree $T$, with $|V(T)|>2$, satisfies $t_{t(e)}=t_{s(e)}$, then $e$ is contracted by using the operadic structure of $\mathcal{BV}(O)$. Roughly speaking, this sequence is equivalent to $\partial X_{k}[l]\times [0\,,\,1]$. For instance, these are points in the sequence $Y_{k}^{(1)}[l]$:\vspace{-5pt}

\begin{figure}[!h]
\begin{center}
\includegraphics[scale=0.30]{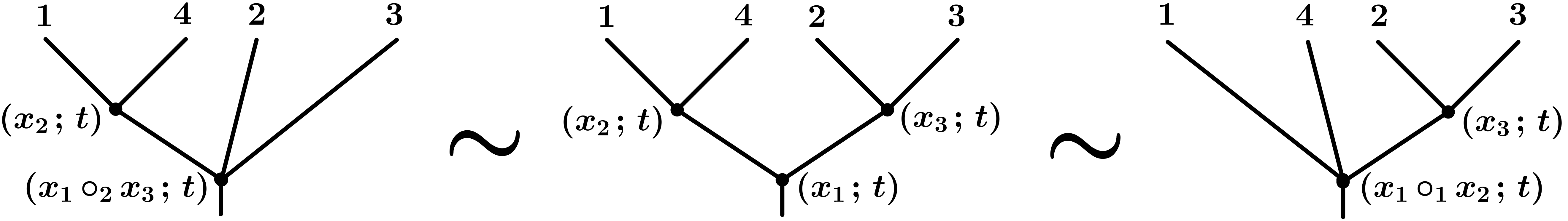}\vspace{-30pt}
\end{center}
\end{figure}

\newpage

\item[$\blacktriangleright$] $Y_{k}^{(2)}[l]$ is the quotient of the restriction of the coproduct (\ref{D4}) to the elements $[T\,;\,\{T_{v}\}]\notin \phi_{k}[l]$ in which the unique vertex of the main tree is indexed by $0$ or $1$. The equivalence relation is generated by the compatibility with the symmetric group axioms of Constructions \ref{e7} and \ref{C0}. Roughly speaking, this sequence is equivalent to $X_{k}[l]\times \{0\,,\,1\}$.

\item[$\blacktriangleright$] $Y_{k}^{(1)}[l]\cap Y_{k}^{(2)}[l]$ is formed by points in $Y_{k}^{(1)}[l]$ in which all the vertices of the main tree are indexed by the same real number $\epsilon\in \{0\,,\,1\}$. Alternatively, this sequence can be defined as the set of points in $Y_{k}^{(2)}[l]$ having at least one inner edge indexed by $1$. Roughly speaking, $Y_{k}^{(1)}[l]\cap Y_{k}^{(2)}[l]$ is equivalent to the sequence $\partial X_{k}[l]\times \{0\,,\,1\}$. For instance, this is a point in the intersection:\vspace{-5pt}
\begin{figure}[!h]
\begin{center}
\includegraphics[scale=0.21]{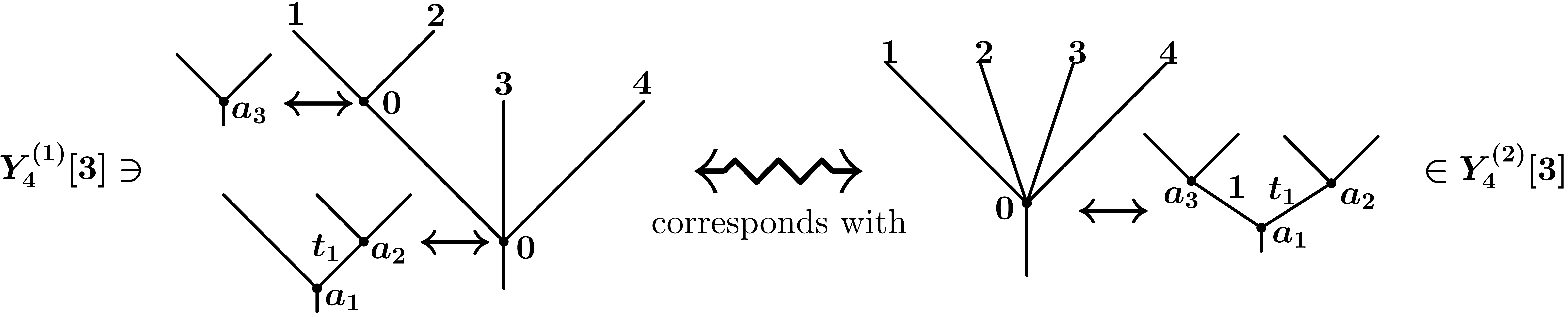}\vspace{-10pt}
\end{center}
\end{figure}
\end{itemize}
Then, we define the pushout product
$$
\partial' Y_{k}[l]:= Y_{k}^{(1)}[l] \underset{Y_{k}^{(1)}[l]\cap Y_{k}^{(2)}[l]}{\coprod} Y_{k}^{(2)}[l].
$$

\begin{lmm}
If $\tilde{g}$ is the map induced by $g$ on the sequence $\partial' Y_{k}[l]$ (see Formula (\ref{H1})), then one has the following homeomorphism:
$$
F: Seq^{\tilde{g}}\big(\,(Y_{k}[l]\,,\, \partial'Y_{k}[l])\,;\, O'\,\big)\leftrightarrows Seq^{g}\big(\,(X_{k}[l]\times [0\,,\,1]\,,\, \partial'X_{k}[l])\,;\, O'\,\big):G.
$$
\end{lmm}

\begin{proof}
Let $f$ be a point in the space $Seq^{\tilde{g}}(\,(Y_{k}[l]\,,\, \partial'Y_{k}[l])\,;\, O'\,)$. The map $F(f)$ sends a pair $(x\,;\,t)$, with $x\in X_{k}[l]$ and $t\in [0\,,\, 1]$, to the image of the corolla indexed by $(x\,;\,t)$ through the application $f$. Conversely, the map $G$ is defined in the same way as the map $\xi$. The reader can check that the maps $F$ and $G$ are well defined and produce a homeomorphism.  
\end{proof}

As a consequence, Diagram (\ref{C8}) is equivalent to Diagram (\ref{A9}) in which the spaces are easier to understand and the right vertical map is just an inclusion of topological spaces. In particular, we deduce from the following that the map $\xi^{g}$ is a weak equivalence if the inclusion from $\partial Y_{k}[l]$ to $\partial' Y_{k}[l]$ is a homotopy equivalence.
\begin{equation}\label{A9}
\xymatrix@R=15pt{
F_{1}\ar[d]_{\xi^{g}} &  Seq^{\tilde{g}}\big(\,(Y_{k}[l]\,,\, \partial'Y_{k}[l])\,;\, O'\,\big) \ar[d] \ar[l]_{\hspace{-50pt}\cong} \\
F_{2} & Seq^{\xi_{k}[l-1](g)}\big(\,(Y_{k}[l]\,,\, \partial Y_{k}[l])\,;\, O'\,\big) \ar[l]_{\hspace{-50pt}\cong}
}
\end{equation}

Both sequences $\partial Y_{k}[l]$ and $\partial Y'_{k}[l]$ have a cellular decomposition indexing by the same set $\Upsilon_{k}[l]$. We will prove, cells by cells, that the inclusion is a homotopy equivalence (or equivalently a weak equivalence between $\Sigma$-cofibrant sequences since the objects are fibrant). Fortunately, some of the cells are identical and we can easily check that we only need to take care of the cells indexed by the subset $\phi_{k}[l]$. Indeed, let us remark that $\partial Y_{k}[l]$ can also be expressed as the pushout product
$$
\partial Y_{k}[l]:= \big( Y_{k}^{(1)}[l]\cap \partial Y_{k}[l] \big) \underset{Y_{k}^{(1)}[l]\cap Y_{k}^{(2)}[l]}{\coprod} Y_{k}^{(2)}[l].
$$ 
In particular, the diagram below implies that the inclusion from $\partial Y_{k}[l]$ to $\partial' Y_{k}[l]$ is a homotopy equivalence if the inclusion from  $Y_{k}^{(1)}[l]\cap \partial Y_{k}[l]$ to $Y_{k}^{(1)}[l]$ is a homotopy equivalence:
$$
\xymatrix@R=19pt{
Y_{k}^{(1)}[l]\cap \partial Y_{k}[l]\ar[d] & Y_{k}^{(1)}[l]\cap Y_{k}^{(2)}[l]\ar@{=}[d] \ar[r]\ar[l] & Y_{k}^{(2)}[l]\ar@{=}[d]\\
Y_{k}^{(1)}[l] & Y_{k}^{(1)}[l]\cap Y_{k}^{(2)}[l]\ar[r]\ar[l] & Y_{k}^{(2)}[l]
}
$$

\begin{Ccl}
The map $\xi_{k}[l]$ is a weak equivalence if the inclusion $Y_{k}^{(1)}[l]\cap \partial Y_{k}[l]\rightarrow Y_{k}^{(1)}[l]$ is a homotopy equivalence.\vspace{-20pt}
\end{Ccl}

\newpage

However, we cannot prove by induction on the number of vertices of the main tree that the inclusion is a homotopy equivalence. To solve this issue, we express the two sequences in terms of colimit of diagrams indexed by the same Reedy category. Then, we use technical tools coming from the homotopy theory of diagrams to prove the result. For this reason, the next subsection gives a short introduction to the homotopy theory of diagrams.

\subsection{Induction: The homotopy theory of diagrams}

The homotopy theory of diagrams consists in introducing a model  structure on the category of functors from a small category $\mathcal{D}$ to a model category $\mathcal{C}$. This structure is called the Reedy model category structure and we refer the reader to \cite{Chacholski02,Hirschhorn03,Hirschhorn15} for more details about the notions discussed in the following. In particular, we recall the notion of Reedy model category and we introduce the Reedy category $\mathcal{D}_{k}[l]$ used in the next subsection in order to prove the Theorem \ref{B5}.

\begin{defi}
A \textit{Reedy category} is a small category $\mathcal{D}$ endowed with two subcategories $\mathcal{D}_{-}$ (the \textit{inverse category}) and $\mathcal{D}_{+}$ (the \textit{direct category}), both of which contain all the objects of $\mathcal{D}$, in which every object can be assigned a non-negative integer, called the degree, such that:
\begin{itemize}
\item[$\blacktriangleright$] every non-identity map of $\mathcal{D}_{+}$ raises degree,
\item[$\blacktriangleright$] every non-identity map of $\mathcal{D}_{-}$ lowers degree,
\item[$\blacktriangleright$] every map in $\mathcal{D}$ factors uniquely as a map in $\mathcal{D}_{-}$ followed by a map in $\mathcal{D}_{+}$.
\end{itemize}
\noindent Many small categories of diagram shapes are Reedy categories such as
$$
\left(
\xymatrix@C=20pt{
\cdot \ar[r] & \cdot \ar[r] & \cdot \ar[r] & \cdots
}
\right)\hspace{15pt}\text{or}\hspace{15pt}
\left(
\xymatrix@C=20pt{
\cdot  & \ar[l] \cdot \ar[r] & \cdot
}
\right)
$$
\end{defi}

\begin{expl}
\textbf{The Reedy category associated to a directed graph}

\noindent Let $\mathcal{G}$ be a directed graph. The objects of the Reedy category $\mathcal{D}$ associated to the directed graph $\mathcal{G}$ is composed of the set of vertices and the set of ordered pairs of vertices of $\mathcal{G}$. For any object in $\mathcal{D}$ corresponding to an ordered pair $(g_{1}\,;\,g_{2})$ in $\mathcal{G}$, there are two morphisms in the category $\mathcal{D}$
$$
d_{0}:(g_{1}\,;\,g_{2})\longrightarrow g_{1} \hspace{15pt}\text{and}\hspace{15pt} d_{1}:(g_{1}\,;\,g_{2})\longrightarrow g_{2}.
$$   
The objects corresponding to a vertices in the graph $\mathcal{G}$ are assigned to the non-negative integer $1$ while the objects corresponding to ordered pairs of vertices are assigned to $0$. In that case, the direct category is $\mathcal{D}$ itself whereas the inverse category contains only the identity maps. \vspace{-5pt} 
\begin{figure}[!h]
\begin{center}
\includegraphics[scale=0.3]{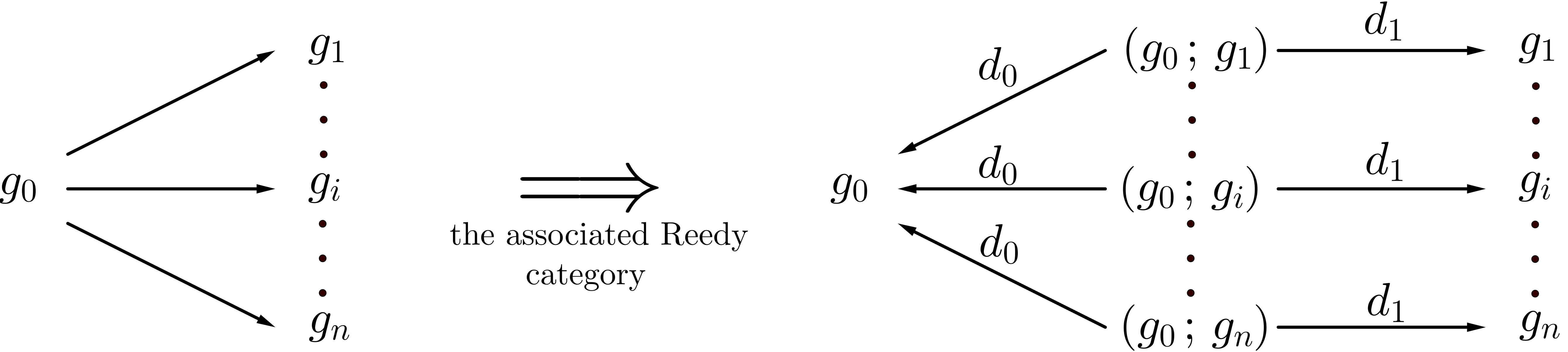}
\caption{Illustration of the directed graph $\mathcal{G}_{n}$ together with the corresponding Reedy category $\mathcal{D}_{n}$.}\label{A5}\vspace{-15pt}
\end{center}
\end{figure}
\end{expl}

\begin{defi}\label{A3}
\textbf{The directed graph $\mathcal{G}_{k}^{i}[l]$ and the corresponding Reedy category $\mathcal{D}_{k}^{i}[l]$}

\noindent Let $\phi_{k}^{p}[l]$ be the set of elements $[T\,;\,\{T_{v}\}]\in \Upsilon$ (see Section \ref{C5}) having $k$ geometrical inputs and such that the trees $T$ and $T_{v}$ are planar, the main tree $T$ is not a corolla and the sum of the vertices of the auxiliary trees $\sum_{v}|V(T_{v})|$ is equal to $l$. 

Analogously to Definition \ref{C9}, one can talk about non-planar isomorphism for the elements in the set $\phi_{k}^{p}[l]$ and more particularly about the automorphisms group $Aut([T\,;\,\{T_{v}\}])$ associated to an element $[T\,;\,\{T_{v}\}]$. More precisely, a non-planar isomorphism from $[T_{1}\,;\,\{T_{v}^{1}\}]$ to $[T_{2}\,;\,\{T_{v}^{2}\}]$  is given by a family of non-planar isomorphisms of trees
$$
f:T_{1}\longrightarrow T_{2}
\hspace{15pt}\text{and}\hspace{15pt}
\{g_{v}:T_{v}^{1}\longrightarrow T_{f(v)}^{2}\}_{v\in V(T_{1})}
$$
such that, for each $v\in V(T_{1})$, the permutation induced by $f$ on the incoming edges of $v$ coincides with the permutation  induced by $g_{v}$ on the leaves of $T_{v}^{1}$. In particular, the non-planar isomorphism of trees $f$ is entirely determined by the family $\{g_{v}\}$. So, the class of an element $[T\,;\,\{T_{v}\}]$ up to a non-planar ismorphism,

\newpage

\noindent  denoted by $\llbracket T\,;\,\{T_{v}\}\rrbracket$, is isomorphic to the product of the classes of the auxiliary trees. For instance, in Figure \ref{H2} the elements $[T_{1}\,;\,\{T_{v}^{1}\}]$ and $[T_{2}\,;\,\{T_{v}^{2}\}]$ are in the same class up to a non-planar isomorphism while the element $[T_{3}\,;\,\{T_{v}^{3}\}]$ represents a different class.
\begin{center} 
\begin{figure}[!h]
\begin{center}
\includegraphics[scale=0.17]{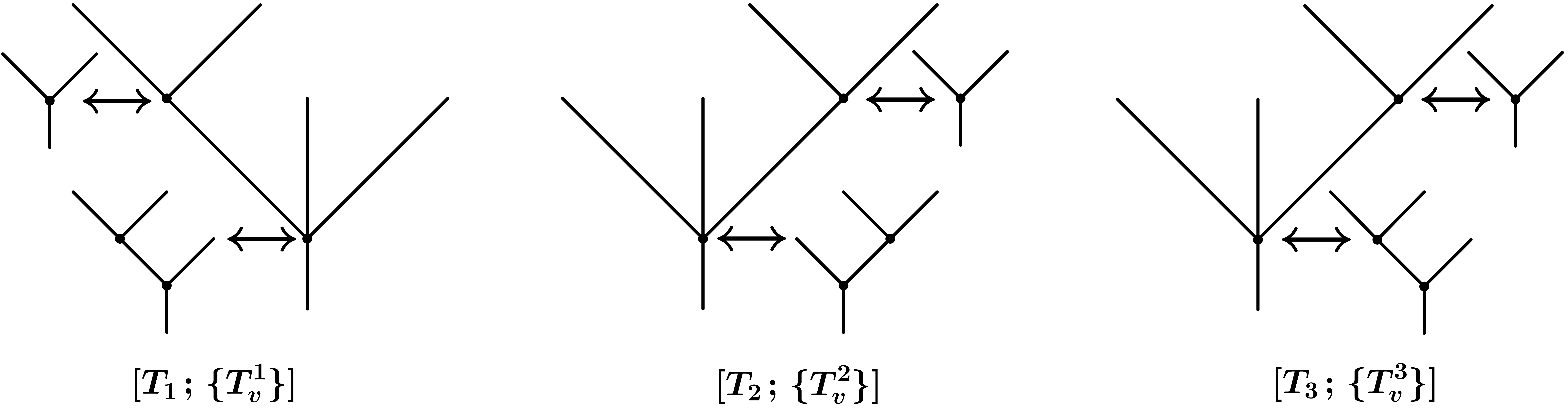}
\caption{Illustration of elements in $\phi^{p}_{4}[3]$.}\vspace{-10pt}\label{H2}
\end{center}
\end{figure}
\end{center}

From now on, we introduce the directed graph $\mathcal{G}_{k}[l]$. Its set of vertices is composed of the set of equivalence classes of the elements in $\phi_{k}^{p}[l]$ up to a non-planar isomorphism. There is an ordered pair $(\mathcal{T}_{1}\,;\,\mathcal{T}_{2})$, with $\mathcal{T}_{1}=\llbracket T_{1}\,;\,\{T^{1}_{v}\}\rrbracket$ and $\mathcal{T}_{2}=\llbracket T_{2}\,;\,\{T^{2}_{v}\}\rrbracket$, if there exists a representative element $[T\,;\,\{T_{v}\}]\in \mathcal{T}_{1}$ such that $T_{2}$ is obtained from $T$ by contracting a unique inner edge $e$. Moreover, if $v'\in V(T_{2})$ is the vertex coming from the contraction of $e=e_{i}(t(e))$, then $T_{v'}^{2}=T_{t(e)}\circ_{i}T_{s(e)}$ and $T_{v}^{2}=T_{v}$ for $v\neq v'$. We denote by $\mathcal{D}_{k}[l]$ the Reedy category associated to the directed graph $\mathcal{G}_{k}[l]$.  By construction, $\mathcal{G}_{k}[l]$ is composed of a finite number of connected components. Each component has an initial element $\mathcal{T}_{0}=\llbracket T_{0}\,;\,\{T^{0}_{v}\}\rrbracket$ satisfying the condition $l-|V(T_{0})|=0$. This condition is equivalent to say that the auxiliary trees $T^{0}_{v}$ are corollas.\vspace{-5pt} 
\begin{figure}[!h]
\begin{center}
\includegraphics[scale=0.145]{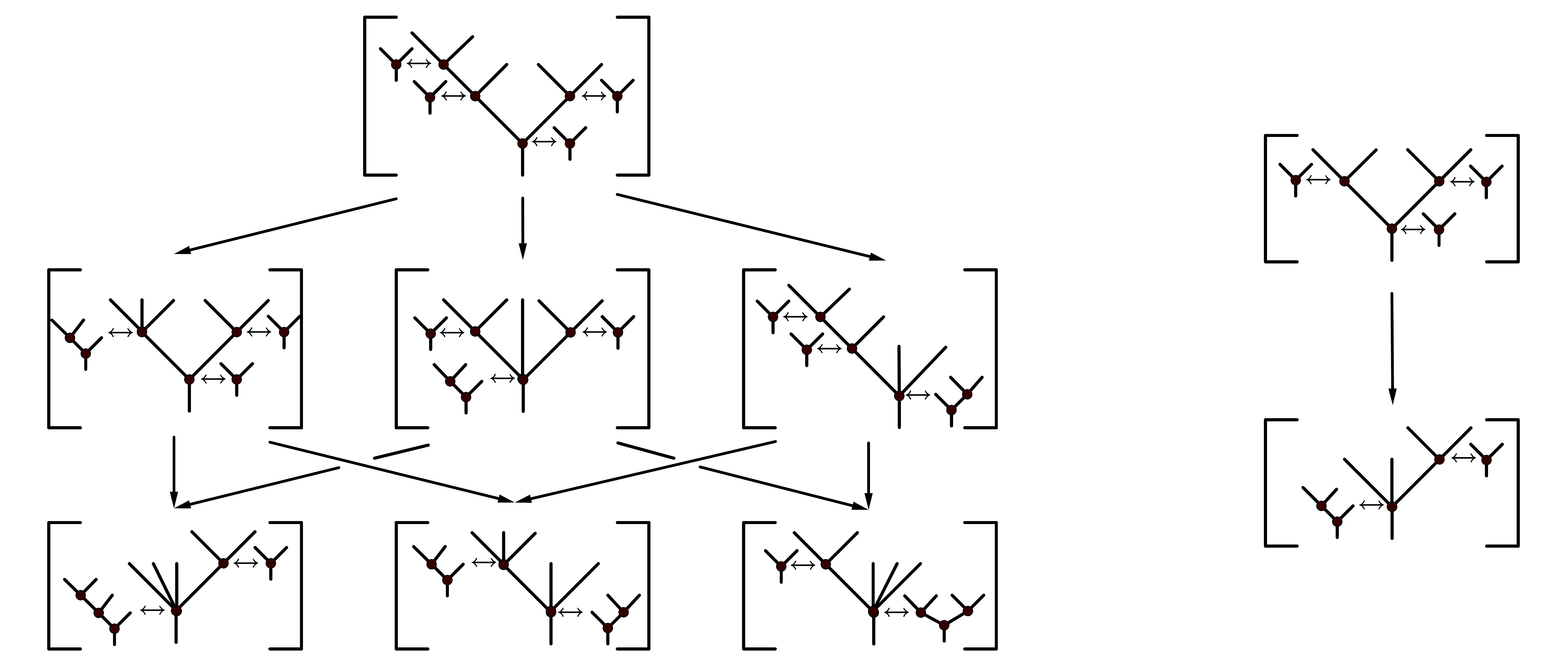}
\caption{Illustration of connected components in $\mathcal{G}_{5}[4]$ and $\mathcal{G}_{4}[3]$, respectively.}\vspace{-5pt}\label{A4}\vspace{-10pt}
\end{center}
\end{figure}

Let us consider $\mathcal{G}_{k}^{i}[l]$, with $0\leq i\leq l-2$, to be the directed graph composed of the vertices $\mathcal{T}=\llbracket T\,;\,\{T_{v}\}\rrbracket$ in $\mathcal{G}_{k}[l]$ satisfying $l-|V(T)|\in \{0\,;\, i\}$. There is an ordered pair $(\mathcal{T}_{1}\,;\,\mathcal{T}_{i})$ in $\mathcal{G}_{k}^{i}[l]$ if there exist equivalences classes $\mathcal{T}_{u}=\llbracket T_{u}\,;\,\{T^{u}_{v}\}\rrbracket$, with $2\leq u\leq i-1$, such that $(\mathcal{T}_{u}\,;\,\mathcal{T}_{u+1})$ are ordered pairs in $\mathcal{G}_{k}[l]$. We denote by $\mathcal{D}_{k}^{i}[l]$ the Reedy category associated to $\mathcal{G}_{k}^{i}[l]$. Furthermore, if $\mathcal{T}$ is an initial element in $\mathcal{G}_{k}^{i}[l]$, then we denote by $\mathcal{G}_{k}^{i}[l]_{\mathcal{T}}$ the full connected sub-graph which contains the element $\mathcal{T}$ and we denote by $\mathcal{D}_{k}^{i}[l]_{\mathcal{T}}$ the Reedy category associated to $\mathcal{G}_{k}^{i}[l]_{\mathcal{T}}$. For instance, the graph $\mathcal{G}_{5}^{2}[4]_{\mathcal{T}}$ associated to the left connected component in Figure \ref{A4} is the following one:\vspace{-4pt}
\begin{figure}[!h]
\begin{center}
\includegraphics[scale=0.13]{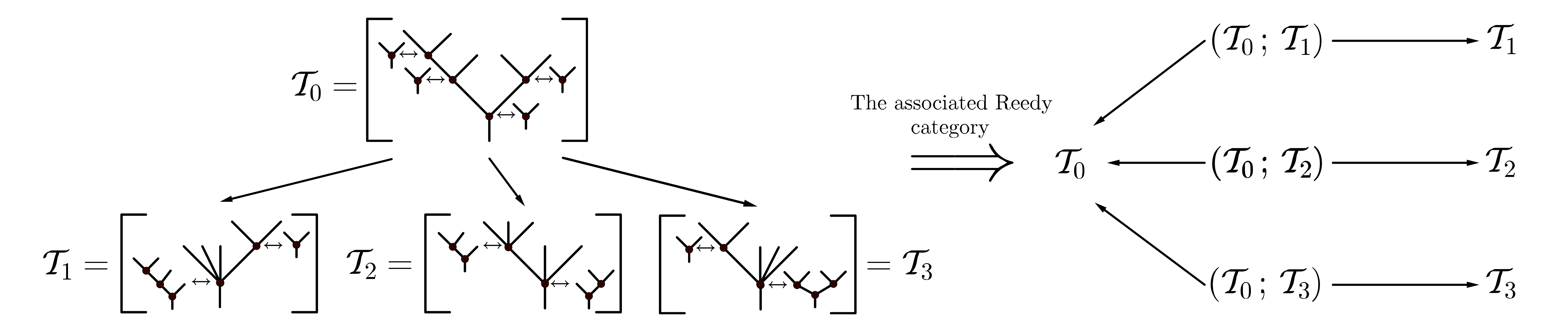}\vspace{-40pt}
\end{center}
\end{figure}
\end{defi}

\newpage

\begin{notat}\label{H6}
Let $F_{1},\,\,F_{2}:\mathcal{D}\rightarrow \mathcal{C}$ be two functors and $t:F_{1}\Rightarrow F_{2}$ be a natural transformation. For $d\in \mathcal{D}$, we denote by $[d]_{\emptyset}$ the category of objects of $\mathcal{D}$ over $d$ containing all of the objects except the identity map of $d$. The latching object of $F_{1}$ and the relative latching object of $t$ at $d$ are respectively:
$$
L_{d}(F_{1}):=\hspace{-7pt} \underset{\hspace{25pt}[d]_{\emptyset}}{colim}F_{1}
\hspace{15pt} \text{and} \hspace{15pt}
L_{d}(f):= F_{1}(d)\underset{L_{d}(F_{1})}{\coprod} L_{d}(F_{2}).
$$  
\end{notat}

\begin{thm}{\cite[Theorem 15.3.4]{Hirschhorn03}}\label{D9}
Let $\mathcal{C}$ be a model category and $\mathcal{D}$ be a reedy category such that $\mathcal{D}_{-}$ contains just the identity maps. There is a model structure on the category $Func(\mathcal{D}\,;\,\mathcal{C})$ of $\mathcal{D}$-diagram in $\mathcal{C}$, called the Reedy model category structure, in which a natural transformation $t:F_{1}\Rightarrow F_{2}$ is:
\begin{itemize}
\item[$\blacktriangleright$] a Reedy weak equivalence (resp. a fibration) if for every object $d\in \mathcal{D}$ the map $t_{d}:F_{1}(d)\rightarrow F_{2}(d)$ is a weak equivalence (resp. a fibration) in $\mathcal{C}$,
\item[$\blacktriangleright$] a Reedy cofibration if for every object $d\in \mathcal{D}$ the map $L_{d}(f)\rightarrow F_{2}(d)$ is a cofibration in $\mathcal{C}$.
\end{itemize}
\noindent In particular, $F_{1}$ is Reedy cofibrant if for every object $d\in \mathcal{D}$ the map $L_{d}(F_{1})\rightarrow F_{1}(d)$ is a cofibration in $\mathcal{C}$. Furthermore, a Reedy weak equivalence $t:F_{1}\Rightarrow F_{2}$ between Reedy cofibrant functors induces a weak equivalence between cofibrant objects in the category $\mathcal{C}$:
$$
\underset{\hspace{20pt}\mathcal{D}}{colim}F_{1}\longrightarrow\hspace{-4pt} \underset{\hspace{20pt}\mathcal{D}}{colim}F_{2}.
$$
\end{thm}

Unfortunately, in the next subsection we introduce functors $\mathbb{F}_{i}^{-}$ and $\mathbb{F}_{i}$ from the Reedy category $\mathcal{D}_{k}^{i}[l]$ to sequences which are not necessarily Reedy cofibrant. However, it doesn't mean that the induced colimits are not cofibrant or that any Reedy weak equivalence between them doesn't induce a weak equivalence between the colimits. For instance, the pushout diagram $\ast\leftarrow S^{1}\rightarrow D^{2}$ in the category of spaces is not Reedy cofibrant since the map $S^{1}\rightarrow \ast$ is not a cofibration. Nevertheless, the colimit $S^{2}$ is cofibrant and the Reedy weak equivalence to the pushout diagram $D^{2}\leftarrow S^{1} \rightarrow D^{2}$ induces a weak equivalence between the colimits. More precisely, one has the following statements: 

\begin{pro}{\cite[Proposition 2.6]{Chacholski02}}\label{D1}
In the case $\mathcal{D}:=(\ast_{1}\longleftarrow \ast_{2}\longrightarrow \ast_{3})$, a natural transformation $t:F_{1}\Rightarrow F_{2}$:
$$
\xymatrix{
\underset{\hspace{20pt}\mathcal{D}}{colim}F_{1}\ar[d]^{f} \ar@{=}[r] & \underset{\hspace{20pt}\mathcal{D}}{colim}\big(\hspace{-31pt} & F_{1}(\ast_{1})\ar[d]_{t_{1}} &\ar[l]\ar[r]\ar[d]_{t_{2}} F_{1}(\ast_{2}) & \ar[d]_{t_{3}} F_{1}(\ast_{3})\big)\\
\underset{\hspace{20pt}\mathcal{D}}{colim}F_{2} \ar@{=}[r] &  \underset{\hspace{20pt}\mathcal{D}}{colim}\big(\hspace{-31pt} & F_{2}(\ast_{1}) &\ar[l]\ar[r] F_{2}(\ast_{2}) & F_{2}(\ast_{3})\big)
}
$$
induces a cofibration $f$ if the maps $t_{3}$ and $L_{1}(t)\rightarrow F_{2}(\ast_{1})$ are cofibrations. Furthermore, $f$ is a weak equivalence between cofibrant objects if the natural transformation $t$ is a Reedy weak equivalence, the object $F_{i}(\ast_{3})$ is cofibrant and the map $F_{i}(\ast_{2})\rightarrow F_{i}(\ast_{1})$ is a cofibration, with $i\in \{1\,;\,2\}$.
\end{pro}

\begin{cor}\label{A8}
Let $t:F_{1}\Rightarrow F_{2}$ be a Reedy weak equivalence between functors from the Reedy category $\mathcal{D}_{n}$, illustrated in Figure \ref{A5}, to a model category $\mathcal{C}$. For $i\in \{1\,;\,2\}$, if $F_{i}(g_{0})$ is cofibrant and the morphisms $F_{i}(d_{1})$ are cofibrations, then the natural transformation $t$ induces a weak equivalence between cofibrant objects in the category $\mathcal{C}$:
$$
\underset{\hspace{20pt}\mathcal{D}_{n}}{colim}F_{1}\longrightarrow\hspace{-4pt} \underset{\hspace{20pt}\mathcal{D}_{n}}{colim}F_{2}.
$$
\end{cor}

\begin{proof}
For $i\in \{1\,;\,2\}$, the colimit of the functor $F_{i}$ is homeomorphic to the colimit of the diagram
$$
\xymatrix@R=2pt{
 & F_{i}(g_{0})\underset{F_{i}(g_{0}\,;\,g_{1})}{\coprod} F_{i}(g_{1})\ar@{.}[dd] & & & \\
 & & & F_{i}(g_{0}\,;\,g_{j}) \ar[r]^{F_{i}(d_{1})} \ar[dd]_{F_{i}(d_{0})} & F_{i}(g_{j})\ar[dd]\\
 F_{i}(g_{0})\ar[ruu]^{f_{1}^{i}} \ar[r]^{\hspace{-15pt}f_{j}^{i}} \ar[rdd]_{f_{n}^{i}} & F_{i}(g_{0})\underset{F_{i}(g_{0}\,;\,g_{j})}{\coprod} F_{i}(g_{j})\ar@{.}[dd] & \text{with the pushout diagrams} & & \\
 & & & F_{i}(g_{0}) \ar[r]^{\hspace{-20pt}f_{j}^{i}} & F_{i}(g_{0})\underset{F_{i}(g_{0}\,;\,g_{j})}{\coprod} F_{i}(g_{j})\\
 & F_{i}(g_{0})\underset{F_{i}(g_{0}\,;\,g_{n})}{\coprod} F_{i}(g_{n}) & & &
} 
$$ 
Since the pushout diagrams preserve the cofibrations and the morphisms $F(d_{1})$ are supposed to be cofibrations, the morphisms $f_{i}$ are also cofibrations. So, the left diagram is Reedy cofibrant. As a consequence of Theorem \ref{D9}, the natural transformation $t$ induces a weak equivalence between cofibrant objects.\vspace{-20pt}    
\end{proof}

\newpage

\subsection{Induction: The morphism $\xi_{k}[l]$ is a weak equivalence}

The aim of this subsection is to prove that the inclusion from $Y^{(1)}_{k}[l]\cap \partial Y_{k}[l]$ to $Y_{k}^{(1)}[l]$, introduced in Section \ref{H7}, is a homotopy equivalence. Since all the objects considered are fibrant, it is sufficient to prove that the inclusion is a weak equivalence between cofibrant sequences. First, we describe a cellular decomposition indexed by the directed graph $\mathcal{G}_{k}[l]$ (see Definition \ref{A3}) of the two sequences taking into account the symmetric group axiom of Construction \ref{C0}. Then, using the Reedy categories $\mathcal{D}_{k}^{i}[l]$, we give an alternative construction by induction of the two sequences and we prove in each step of the construction that the inclusion is a homotopy equivalence. More precisely, for $0\leq i \leq l-2$ we introduce the following functors: 
$$
\mathbb{F}_{i}^{-};\mathcal{D}_{k}^{i}[l]\longrightarrow Seq \hspace{15pt}\text{and}\hspace{15pt} \mathbb{F}_{i}^{-};\mathcal{D}_{k}^{i}[l]\longrightarrow Seq,
$$
as well as natural transformations $t_{i}:\mathbb{F}_{i}^{-}\Rightarrow \mathbb{F}_{i}$ such that one has the following properties:
\begin{itemize}
\item[$\blacktriangleright$] the colimits of the functors $\mathbb{F}_{l-2}^{-}$ and $\mathbb{F}_{l-2}$ are the sequences $Y^{(1)}_{k}[l]\cap \partial Y_{k}[l]$ and $Y_{k}^{(1)}[l]$ respectively,
\item[$\blacktriangleright$]  the functors $\mathbb{F}_{i}^{-}$ and $\mathbb{F}_{i}$ are obtained from the colimit of the functors $\mathbb{F}_{i-1}^{-}$ and $\mathbb{F}_{i-1}$ respectively,
\item[$\blacktriangleright$] the functors $\mathbb{F}_{i}^{-}$ and $\mathbb{F}_{i}$ satisfy the conditions of Corollary \ref{A8}, 
\item[$\blacktriangleright$] the natural transformations $t_{i}:\mathbb{F}_{i}^{-}\Rightarrow \mathbb{F}_{i}$ are Reedy weak equivalences in $Func(\mathcal{D}_{k}^{i}[l]\,;\,Seq)$.
\end{itemize}

\subsubsection{Cellular decompositions of the sequences $Y^{(1)}_{k}[l]\cap \partial Y_{k}[l]$ and $Y_{k}^{(1)}[l]$}

First, one has to fix some notation. Let $[T\,;\,\{T_{v}\}]$ be an element in $\phi_{k}^{p}[l]$. Two permutations $\sigma,\,\,\sigma'\in \Sigma_{|v'|}$, with $v'\in V(T_{v})$, are said to be equivalent according to $[T\,;\,\{T_{v}\}]$ if the non-planar isomorphisms induced by $\sigma$ and $\sigma'$ are equal up to an automorphism. In other words, there exists $\tau\in \Sigma_{|v'|}$ such that the non-planar isomorphism induced by $\tau$ is an automorphism in $Aut([T\,;\,\{T_{v}\}])$ and $\sigma=\sigma'\circ\tau$. This provides an equivalence relation $\sim$ on the symmetric group $\Sigma_{|v'|}$ and we fix a set of representative elements $Stab([T\,;\,\{T_{v}\}]\,;\,v')$ in which the identity permutation is the representative element of the automorphism:
$$
Stab([T\,;\,\{T_{v}\}]\,;\,v')\in \left\{ \left.
\{\sigma_{i}\}\in \Sigma_{|v'|}^{\times |\Sigma_{|v'|}/\sim|}\,\right|\, \sigma_{1}=id \text{ and } [\sigma_{i}]\neq [\sigma_{j}] \text{ if } i\neq j
\right\}.
$$

Let $\mathcal{T}=\llbracket T\,;\,\{T_{v}\}\rrbracket$ be a vertex in the directed graph $\mathcal{G}_{k}[l]$ and let $[T\,;\,\{T_{v}\}]$ be a representative element. In order to introduce the space labelling the vertices of $[T\,;\,\{T_{v}\}]$ taking into account the symmetric group axiom of Construction \ref{C0}, let us remark that the operad $O$ is $\Sigma$-cofibrant and each space $O(n)$ is of the form $\Sigma_{n}[V_{n}]$ where $\Sigma_{n}[-]:Top\rightarrow \Sigma_{n}\text{-}Top$ is the free functor and $V_{n}$ is a cofibrant space. Then we consider the following spaces:
$$
M(\mathcal{T})= \underset{v\in V(T)}{\prod}\,\,\underset{v'\in V(T_{v})}{\prod} V_{|v'|}\times Stab([T\,;\,\{T_{v}\}]\,;\,v')
\hspace{15pt}\text{and}\hspace{15pt} 
H(\mathcal{T})\subset\underset{v\in V(T)}{\prod} [0\,,\,1]\,\, \times \underset{e\in E^{int}(T_{v})}{\prod} [0\,,\,1],
$$ 
where $H(\mathcal{T})$ is the subspace formed by points satisfying the restriction of Construction \ref{C0} on the real numbers $\{t_{v}\}$ indexing the vertices of the main tree $T$. It means that if $e$ is an inner edge of $T$, then one has $t_{s(e)}\geq t_{t(e)}$. Finally, the following sequences give rise to a cellular decomposition of the sequence $Y_{k}^{(1)}[l]$:
$$
\mathbb{F}(\mathcal{T})=M(\mathcal{T}) \times H(\mathcal{T}) \times \Sigma_{|T|}.
$$

To introduce the cellular decomposition of the sequence $Y^{(1)}_{k}[l]\cap \partial Y_{k}[l]$, we consider the subspace $M^{-}(\mathcal{T})$ formed by points in $M(\mathcal{T})$ having at least one bivalent vertex labelled by the unit $\ast_{1}$ of the operad $O$. Similarly, we introduced the subspace $H^{-}(\mathcal{T})$ formed by families of real numbers in $H(\mathcal{T})$ having inner edges of auxiliary trees indexed by $0$ or having vertices of the main tree indexed by $0$ or $1$. Then, we consider the following sequence:
$$
\mathbb{F}^{-}(\mathcal{T})=(M\times H)^{-}(\mathcal{T})\times \Sigma_{|T|} 
\hspace{15pt}\text{with}\hspace{15pt}
(M\times H)^{-}(\mathcal{T})= M^{-}(\mathcal{T}) \times H(\mathcal{T}) \underset{M^{-}(\mathcal{T}) \times H^{-}(\mathcal{T})}{\coprod} M(\mathcal{T}) \times H^{-}(\mathcal{T}).
$$

Finally, the inclusions $M^{-}(\mathcal{T})\rightarrow  M(\mathcal{T})$ and $H^{-}(\mathcal{T})\rightarrow  H(\mathcal{T})$ induce a morphism from the sequence $
\mathbb{F}^{-}(\mathcal{T})$ to $\mathbb{F}(\mathcal{T})$. As proved in the next proposition, these two cellular decompositions are weakly equivalent in the sense that, for each $\mathcal{T}=\llbracket T\,;\,\{T_{v}\}\rrbracket$, the morphism $\mathbb{F}^{-}(\mathcal{T})\rightarrow\mathbb{F}(\mathcal{T})$ is a weak equivalence. However, it is not sufficient to conclude that the sequences $Y^{(1)}_{k}[l]\cap \partial Y_{k}[l]$ and $Y_{k}^{(1)}[l]$ are weakly equivalent since these cellular decompositions don't take into account the axiom $(iii)$ of Construction \ref{C0}. 

\newpage

\begin{pro}\label{B1}
Let $\mathcal{T}=\llbracket T\,;\,\{T_{v}\}\rrbracket$ be an element in the directed graph $\mathcal{G}_{k}[l]$. The morphism of sequences from $\mathbb{F}^{-}(\mathcal{T})$ to $\mathbb{F}(\mathcal{T})$ is a weak equivalence.
\end{pro}

\begin{proof}
First, we have to show that the inclusion from $(M\times H)^{-}(\mathcal{T})$ to $M(\mathcal{T})\times H(\mathcal{T})$ is a weak equivalence in the category of topological spaces. Since the operad $O$ is supposed to be well pointed, the inclusion $M^{-}(\mathcal{T})\rightarrow M(\mathcal{T})$ is a cofibration. Similarly, the inclusion $H^{-}(\mathcal{T})\rightarrow H(\mathcal{T})$ is a cofibration as an inclusion of $CW$-complexes. As a consequence of Lemma \ref{refA}, the pushout product map
$$
(M\times H)^{-}(\mathcal{T})= M^{-}(\mathcal{T})\times H(\mathcal{T})\underset{M^{-}(\mathcal{T})\times H^{-}(\mathcal{T})}{\coprod} M(\mathcal{T})\times H^{-}(\mathcal{T}) \longrightarrow M(\mathcal{T})\times H(\mathcal{T})
$$
is an acyclic cofibration if the inclusion $H^{-}(\mathcal{T})\rightarrow H(\mathcal{T})$ is a weak equivalence. For this purpose, it is sufficient to prove that the $CW$-complexes $H^{-}(\mathcal{T})$ and $H(\mathcal{T})$ are contractible. In order to define the homotopies, we consider the following set:
$$
Max(\mathcal{T}):=\left\{ \left.
v\in V(T)\,
\right|
\, \nexists e\in E^{int}(T),\,\,t(e)=v
\right\}.
$$
Since $[T\,;\,\{T_{v}\}]\in \phi_{k}^{p}[l]$, the main tree $T$ is not a corolla. As a consequence, the root $r$ is not an element in $Max(\mathcal{T})$ and $Max(\mathcal{T})$ is not empty. First, let $\mathcal{H}_{1}$ be the homotopy bringing the real numbers indexing the vertices in $Max(\mathcal{T})$ to $1$, the real number indexing the root to $0$ and the other real numbers to $1/2$. In other words, $\mathcal{H}_{1}:H(\mathcal{T})\times [0\,,\,1]\rightarrow H(\mathcal{T})$ sends $(\{t_{v}\}\,;\,\{t_{v}^{e}\})\times u$ to the point $(\{H_{\mathcal{T}}(t_{v}\,;\,u)\}\,;\,\{t_{v}^{e}\})$ where
$$
H_{\mathcal{T}}(t_{v}\,;\,u):=
\left\{
\begin{array}{ll}\vspace{3pt}
(1-t_{v})u+t_{v} & \text{if } v\in Max(\mathcal{T}), \\ \vspace{3pt}
(1-u)t_{v} & \text{if } v=r, \\ 
(1/2-t_{v})u+t_{v} & \text{otherwise}.
\end{array} 
\right.
$$
Thereafter, we use the homotopy $\mathcal{H}_{2}$ bringing the real numbers indexing the inner edges of auxiliary trees to $1$. Finally, the homotopy, showing that the space $H(\mathcal{T})$ is contractible, is given by:
\begin{equation}\label{B3}
\begin{array}{rcl}\vspace{3pt}
\mathcal{H}:H(\mathcal{T})\times [0\,,\,1] & \longrightarrow & H(\mathcal{T}); \\ 
(x\,;\,u) & \longmapsto & \left\{
\begin{array}{ll}\vspace{3pt}
\mathcal{H}_{1}(x\,;\,2u) & \text{if } u\leq 1/2, \\ 
\mathcal{H}_{2}(\mathcal{H}_{1}(x\,;\,1)\,;\,2u-1) & \text{if } u\geq 1/2.
\end{array} 
\right.
\end{array} 
\end{equation}
\begin{figure}[!h]
\begin{center}
\includegraphics[scale=0.3]{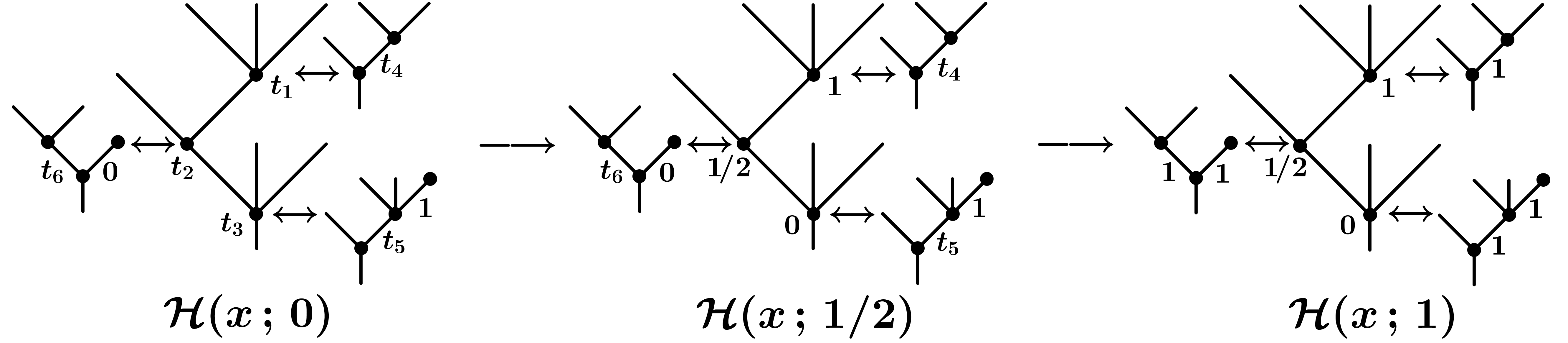}
\caption{Illustration of the homotopy $\mathcal{H}$.}\vspace{-10pt}
\end{center}
\end{figure}

\noindent The restriction of the homotopy $\mathcal{H}$ to the subspace $H^{-}(\mathcal{T})$ is well defined and it proves that $H^{-}(\mathcal{T})$ is also contractible. Thus shows that the inclusion $i:(M\times H)^{-}(\mathcal{T})\rightarrow M(\mathcal{T})\times H(\mathcal{T})$ is a acyclic cofibration and the morphism from $\mathbb{F}^{-}(\mathcal{T})$ to $\mathbb{F}(\mathcal{T})$, which coincides with $\Sigma_{|T|}[i]$, is an acyclic $\Sigma$-cofibration. In particular, it is a weak equivalence.  
\end{proof}

\subsubsection{Construction of the functors $\mathbb{F}^{-}_{0},\,\mathbb{F}_{0}:\mathcal{D}_{k}^{0}[l]\rightarrow Seq$}

By definition, the directed graph $\mathcal{G}_{k}^{0}[l]$ is only composed of the initial elements in $\mathcal{G}_{k}[l]$ and it doesn't have ordered pair of vertices. Consequently, for any object $\mathcal{T}=\llbracket T\,;\,\{T_{v}\}\rrbracket$ in the Reedy category $\mathcal{D}_{k}^{0}[l]$, the sequences $\mathbb{F}^{-}_{0}(\mathcal{T})$ and $\mathbb{F}_{0}(\mathcal{T})$, concentrated in arity $|T|$, are defined as follows:
$$
\mathbb{F}^{-}_{0}(\mathcal{T})= \mathbb{F}^{-}(\mathcal{T}) 
\hspace{15pt} \text{and} \hspace{15pt}
\mathbb{F}_{0}(\mathcal{T})=\mathbb{F}(\mathcal{T}).
$$
The natural transformation $t_{0}:\mathbb{F}^{-}_{0}\Rightarrow \mathbb{F}_{0}$ is induced by the inclusion from $\mathbb{F}^{-}(\mathcal{T})$ to $\mathbb{F}(\mathcal{T})$. Since there is no ordered pair of vertices in $\mathcal{G}_{k}^{0}[l]$, the colimits of the functors $\mathbb{F}^{-}_{0}$ and $\mathbb{F}_{0}$ are just coproduct along the objects in the category $\mathcal{D}_{k}^{0}[l]$. In particular, one has the following statement:

\newpage

\begin{pro}
The natural transformation $t_{0}:\mathbb{F}_{0}^{-}\Rightarrow \mathbb{F}_{0}$ induces a weak equivalence between cofibrant sequences:
$$
\underset{\hspace{20pt}\mathcal{D}_{k}^{0}[l]}{colim}\mathbb{F}_{0}^{-}\longrightarrow\hspace{-4pt} \underset{\hspace{20pt}\mathcal{D}_{k}^{0}[l]}{colim}\mathbb{F}_{0}.
$$
\end{pro}

\begin{proof}
Since there is no morphism other than the identity maps in the Reedy category $\mathcal{D}_{k}^{0}[l]$, we only have to check that the morphism $t_{\mathcal{T}}: \mathbb{F}^{-}_{0}(\mathcal{T})\rightarrow \mathbb{F}_{0}(\mathcal{T})$, with $\mathcal{T}$ an object in $\mathcal{D}_{k}^{0}[l]$, is a weak equivalence between cofibrant sequences. Nevertheless, we also know from Proposition \ref{B1} that the morphism $t_{\mathcal{T}}$ is a weak equivalence. So, it is sufficient to show that the two sequences are $\Sigma$-cofibrant.

By definition the spaces $V_{n}$ are cofibrant. So, $M(\mathcal{T})$ is cofibrant as a finite product of cofibrant spaces. Furthermore, $H(\mathcal{T})$ is cofibrant as a $CW$-complex. Consequently, $M(\mathcal{T})\times H(\mathcal{T})$ is a cofibrant space and the sequence $\mathbb{F}_{0}(\mathcal{T})=\Sigma_{|T|}[M(\mathcal{T})\times H(\mathcal{T})]$ is $\Sigma$-cofibrant. 

Similarly, the map from  $M^{-}(\mathcal{T})$ to  $M(\mathcal{T})$ is a cofibration between cofibrant spaces since the operad $O$ is assumed to be well pointed. Moreover, the map $H^{-}(\mathcal{T})\rightarrow H(\mathcal{T})$ is a cofibration between cofibrant spaces as an inclusion of $CW$-complexes. In particular, the maps which compose the following diagram are cofibrations:
$$
\xymatrix@R=15pt{
\emptyset \ar[d] & \emptyset \ar[l] \ar[r] \ar[d] & \emptyset \ar[d] \\
M^{-}(\mathcal{T}) \times H(\mathcal{T}) & M^{-}(\mathcal{T}) \times H^{-}(\mathcal{T}) \ar[r] \ar[l] & M(\mathcal{T}) \times H^{-}(\mathcal{T})
}
$$ 
Proposition \ref{D1}, applied to the above diagram, implies that $(M\times H)^{-}(\mathcal{T})$ is a cofibrant space and the sequence $\mathbb{F}_{0}^{-}(\mathcal{T})=\Sigma_{|T|}[(M\times H)^{-}(\mathcal{T})]$ is $\Sigma$-cofibrant.
\end{proof}

\subsubsection{Construction of the functors $\mathbb{F}^{-}_{i},\,\mathbb{F}_{i}:\mathcal{D}_{k}^{0}[l]\rightarrow Seq$}

The colimit of the functor $\mathbb{F}_{0}$, defined in the previous subsection, is the sub-sequence formed by points in $Y_{k}^{(1)}[l]$ indexed by elements $[T\,;\,\{T_{v}\}]\in \Upsilon_{k}[l]$ in which the auxiliary trees are corollas (i.e. $l-|V(T)|=0$). Roughly speaking, the functor $\mathbb{F}_{1}$ is built such that its colimit is obtained from the colimit of $\mathbb{F}_{0}$ by gluing cells (indexed by element $[T\,;\,\{T_{v}\}]\in \Upsilon_{k}[l]$ satisfying $l-|V(T)|=1$) taking into account the axiom $(iii)$ of Construction \ref{C0}. 

More generally, let us assume that the functors $\mathbb{F}^{-}_{i-1}$, $\mathbb{F}_{i-1}:\mathcal{D}_{k}^{i-1}[l]\rightarrow Seq$ as well as the natural transformation $t_{i-1}:\mathbb{F}_{i-1}^{-}\Rightarrow \mathbb{F}_{i-1}$, with $0<i\leq l-2$, are defined such that the colimits of the functors are $\Sigma$-cofibrant and the morphism induced by $t_{i-1}$ between the colimits is a weak equivalence. If $\mathcal{T}=\llbracket T\,;\,\{T_{v}\}\rrbracket$ is an initial element in the Reedy category $\mathcal{D}_{k}^{i}[l]$, then $\mathcal{T}$ is also an initial element in $\mathcal{D}_{k}^{i-1}[l]$ and we denote by $\mathbb{F}^{-}_{i-1\,|\, \mathcal{T}}$ and $\mathbb{F}_{i-1\,|\, \mathcal{T}}$ the restriction of the functors $\mathbb{F}^{-}_{i-1}$ and $\mathbb{F}_{i-1}$ to the category $\mathcal{D}_{k}^{i-1}[l]_{\mathcal{T}}$ described in Definition \ref{A3}. Finally, the sequences $\mathbb{F}^{-}_{i}(\mathcal{T})$ and $\mathbb{F}_{i}(\mathcal{T})$, concentrated in arity $|T|$, are defined as follows:
$$
\mathbb{F}^{-}_{i}(\mathcal{T})=\left\{
\begin{array}{cl}\vspace{4pt}
\mathbb{F}^{-}(\mathcal{T}) & \text{if } l-|V(T)|\neq 0, \\ 
colim_{\mathcal{D}_{k}^{i-1}[l]_{\mathcal{T}}} \mathbb{F}^{-}_{i-1\,|\, \mathcal{T}} & \text{if } l-|V(T)|= 0,
\end{array} 
\right. \hspace{5pt}\text{and}\hspace{10pt} 
\mathbb{F}_{i}(\mathcal{T})=\left\{
\begin{array}{cl}\vspace{4pt}
\mathbb{F}(\mathcal{T}) & \text{if } l-|V(T)|\neq 0, \\ 
colim_{\mathcal{D}_{k}^{i-1}[l]_{\mathcal{T}}} \mathbb{F}_{i-1\,|\, \mathcal{T}} & \text{if } l-|V(T)|= 0.
\end{array} 
\right.
$$

Let $(\mathcal{T}\,;\,\mathcal{T}')$, with $\mathcal{T}=\llbracket T\,;\,\{T_{v}\}\rrbracket$ and $\mathcal{T}'=\llbracket T'\,;\,\{T'_{v}\}\rrbracket$, be an ordered pair of vertices in $\mathcal{G}^{i}_{k}[l]$. In particular, $\mathcal{T}$ is necessarily an initial element. In some sense, we want to define the sequence $\mathbb{F}_{i}(\mathcal{T}\,;\,\mathcal{T}')$ as the common points between the colimit of $\mathbb{F}_{i-1\,|\, \mathcal{T}}$ and $\mathbb{F}_{i}(\mathcal{T}')$ coming from the axiom $(iii)$ of Construction \ref{C0}. For this purpose we introduce the subspaces  $H^{-}(\mathcal{T}\,;\,\mathcal{T}')$ and $H(\mathcal{T}\,;\,\mathcal{T}')$ formed by points in $H^{-}(\mathcal{T}')$ and $H(\mathcal{T}')$, respectively, having at least one inner edge of an auxiliary tree indexed by $1$. Then, we consider the product space $(M\times H)(\mathcal{T}\,;\,\mathcal{T}')= M(\mathcal{T}')\times H(\mathcal{T}\,;\,\mathcal{T}')$ as well as pushout product
$$
(M\times H)^{-}(\mathcal{T}\,;\,\mathcal{T}')=M^{-}(\mathcal{T}') \times H(\mathcal{T}\,;\,\mathcal{T}') \underset{M^{-}(\mathcal{T}') \times H^{-}(\mathcal{T}\,;\,\mathcal{T}')}{\coprod} M(\mathcal{T}') \times H^{-}(\mathcal{T}\,;\,\mathcal{T}').
$$
Finally, the sequences $\mathbb{F}_{i}^{-}(\mathcal{T}\,;\,\mathcal{T}')$ and $\mathbb{F}_{i}(\mathcal{T}\,;\,\mathcal{T}')$ are defined as follows:
$$
\mathbb{F}_{i}^{-}(\mathcal{T}\,;\,\mathcal{T}')=(M\times H)^{-}(\mathcal{T}\,;\,\mathcal{T}')\times \Sigma_{|T|} 
\hspace{15pt}\text{and}\hspace{15pt} \mathbb{F}_{i}(\mathcal{T}\,;\,\mathcal{T}')=M(\mathcal{T}')\times H(\mathcal{T}\,;\,\mathcal{T}')\times \Sigma_{|T|}.
$$

The morphisms $\mathbb{F}_{i}^{-}(d_{1})$ and $\mathbb{F}_{i}(d_{1})$ are induced by the inclusions $H(\mathcal{T}\,;\,\mathcal{T}')\rightarrow H(\mathcal{T}')$ and $H^{-}(\mathcal{T}\,;\,\mathcal{T}')\rightarrow H^{-}(\mathcal{T}')$, respectively. The morphisms $\mathbb{F}_{i}^{-}(d_{0})$ and $\mathbb{F}_{i}(d_{0})$ consist in cutting an inner edge of an auxiliary tree indexed by $1$. As shown in Figure \ref{A7}, some of the permutation in $Stab(\mathcal{T}'\,;\,v')$ are not necessarily elements in $Stab(\mathcal{T}\,;\,v')$. In that case, we fix it by applying the appropriate non-planar isomorphism. The morphisms $\mathbb{F}_{i}^{-}(d_{0})$ and $\mathbb{F}_{i}(d_{0})$ doesn't depend on the choice of the inner edge indexed by $1$ since $\mathbb{F}_{i}^{-}(\mathcal{T})$ and $\mathbb{F}_{i}(\mathcal{T})$ are obtained as colimits.\vspace{-15pt}

\newpage

\begin{figure}[!h]
\begin{center}
\includegraphics[scale=0.14]{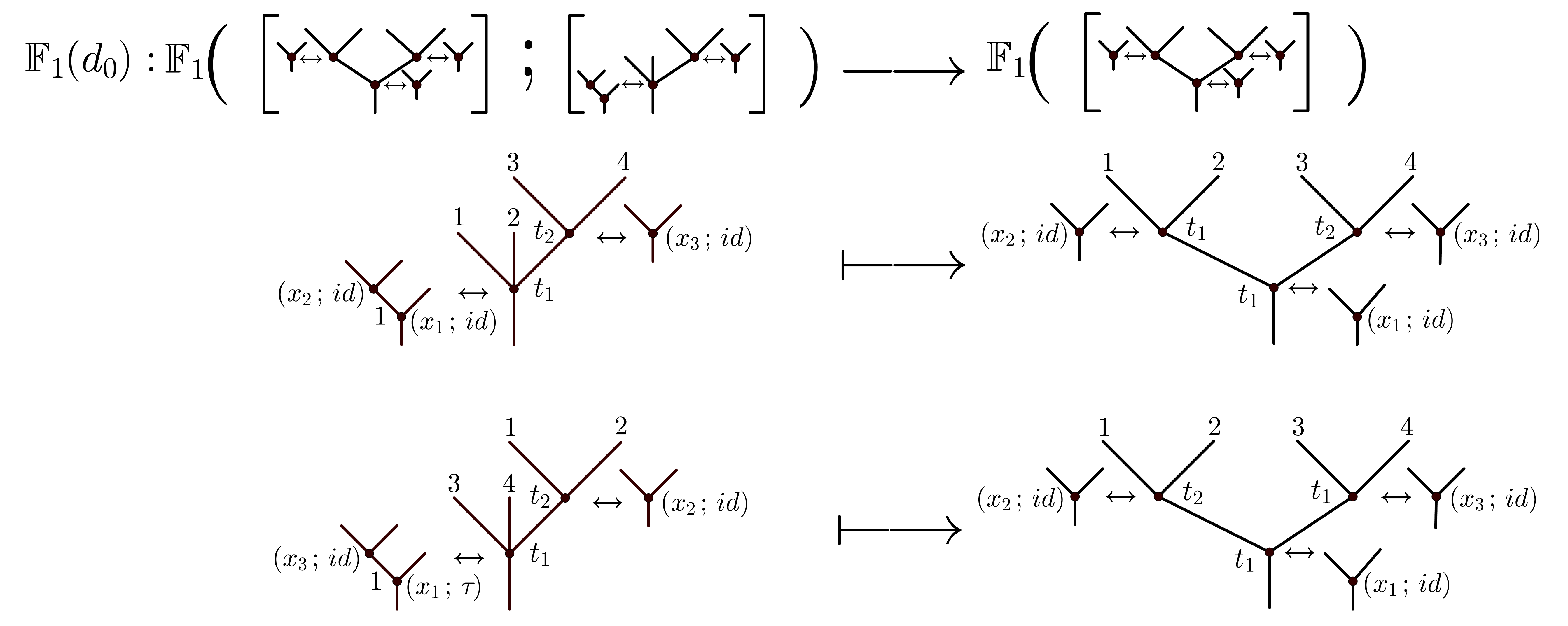}
\caption{Illustration of the morphism $\mathbb{F}_{1}(d_{0})$ with $\tau=(12)\in \Sigma_{2}$.}\label{A7}
\end{center}
\end{figure}

The inclusion $H^{-}(\mathcal{T}\,;\,\mathcal{T}')\rightarrow H(\mathcal{T}\,;\,\mathcal{T}')$ induces a natural transformation $t_{i}:\mathbb{F}_{i}^{-}\Rightarrow \mathbb{F}_{i}$. Finally, let us notice that, as shown in the above picture, the morphisms $\mathbb{F}_{i}(d_{0})$ are not necessarily injective (just take $t_{1}=t_{2}$ in Figure \ref{A7}) and they don't seem to be cofibrations. Nevertheless, one has the following statement:

\begin{pro}\label{B6}
The natural transformation $t_{i}:\mathbb{F}_{i}^{-}\Rightarrow \mathbb{F}_{i}$ induces a weak equivalence between cofibrant sequences:
$$
\underset{\hspace{20pt}\mathcal{D}_{k}^{i}[l]}{colim}\mathbb{F}_{i}^{-}\longrightarrow\hspace{-4pt} \underset{\hspace{20pt}\mathcal{D}_{k}^{i}[l]}{colim}\mathbb{F}_{i}.
$$
\end{pro}

\begin{proof}
As mentioned in Definition \ref{A3}, the Reedy category $\mathcal{D}_{k}^{i}[l]$ is composed of connected components $\{\mathcal{D}_{k}^{i}[l]_{\mathcal{T}}\}$ indexed by the initial elements of the directed graph $\mathcal{G}_{k}^{i}[l]$. In order to prove the proposition, we only have to check that the restriction of the natural transformation $t_{i\,|\,\mathcal{T}}:\mathbb{F}_{i\,|\,\mathcal{T}}^{-}\Rightarrow \mathbb{F}_{i\,|\,\mathcal{T}}$ induces a weak equivalence between cofibrant sequences:
$$
\underset{\hspace{20pt}\mathcal{D}_{k}^{i}[l]_{\mathcal{T}}}{colim}\mathbb{F}_{i\,|\, \mathcal{T}}^{-}\longrightarrow\hspace{-4pt} \underset{\hspace{20pt}\mathcal{D}_{k}^{i}[l]_{\mathcal{T}}}{colim}\mathbb{F}_{i\,|\, \mathcal{T}}.
$$

Furthermore, each sub-category $\{\mathcal{D}_{k}^{i}[l]_{\mathcal{T}}$ is of the form $\mathcal{D}_{n}$ with $n$ an integer (see Figure \ref{A5}). In order to apply Corollary \ref{A8} to the natural transformation $t_{i\,|\,\mathcal{T}}$, let us remark that, by assumption, the inclusion from $\mathbb{F}_{i\,|\,\mathcal{T}}^{-}(\mathcal{T})$ to $\mathbb{F}_{i\,|\,\mathcal{T}}(\mathcal{T})$ is a weak equivalence between cofibrant sequences since this map is obtained from the natural transformation $t_{i-1\,|\,\mathcal{T}}$. For any pair of vertices $(\mathcal{T}\,;\,\mathcal{T}')$ in the directed graph $\mathcal{G}_{k}[l]$, Proposition \ref{B1} implies that the inclusion from $\mathbb{F}_{i\,|\,\mathcal{T}}^{-}(\mathcal{T}')$ to $\mathbb{F}_{i\,|\,\mathcal{T}}(\mathcal{T}')$ is a weak equivalence. Moreover, the reader can easily check that the homotopy (\ref{B3}) restricts to the subspaces $H^{-}(\mathcal{T}\,;\,\mathcal{T}')$ and $H(\mathcal{T}\,;\,\mathcal{T}')$. Thus proves that the inclusion from $\mathbb{F}_{i\,|\,\mathcal{T}}^{-}(\mathcal{T}\,;\,\mathcal{T}')$ and $\mathbb{F}_{i\,|\,\mathcal{T}}(\mathcal{T}\,;\,\mathcal{T}')$ is also a weak equivalence. Consequently, the natural transformation $t_{i\,|\,\mathcal{T}}$ is a weak equivalence and we only have to show that the morphisms $\mathbb{F}_{i\,|\,\mathcal{T}}^{-}(d_{1})$ and $\mathbb{F}_{i\,|\,\mathcal{T}}(d_{1})$ are $\Sigma$-cofibrations.

Since the inclusion from $H(\mathcal{T}\,;\,\mathcal{T'})$ to $H(\mathcal{T}')$ is a cofibration as an inclusion of $CW$-complexes, the map $i:M(\mathcal{T'})\times H(\mathcal{T}\,;\,\mathcal{T'})\rightarrow M(\mathcal{T'})\times H(\mathcal{T'})$ is also a cofibration. Consequently, $\mathbb{F}_{i\,|\,\mathcal{T}}(d_{1})=\Sigma_{|T|}[i]$ is a $\Sigma$-cofibration. Similarly, $H^{-}(\mathcal{T}\,;\,\mathcal{T'})\rightarrow H^{-}(\mathcal{T}')$ is a cofibration and the morphisms which compose the following diagram are cofibrations:
$$
\xymatrix{
M^{-}(\mathcal{T}')\times H(\mathcal{T}\,;\,\mathcal{T'}) \ar[d] & M^{-}(\mathcal{T}')\times H^{-}(\mathcal{T}\,;\,\mathcal{T'}) \ar[d]\ar[l] \ar[r] & M(\mathcal{T}')\times H^{-}(\mathcal{T}\,;\,\mathcal{T'}) \ar[d]\\
M^{-}(\mathcal{T}')\times H(\mathcal{T'})  & M^{-}(\mathcal{T}')\times H^{-}(\mathcal{T'}) \ar[l] \ar[r] & M(\mathcal{T}')\times H^{-}(\mathcal{T'})
}
$$ 
Proposition \ref{D1}, applied to the above diagram, implies that the inclusion $i^{-}:(M\times H)^{-}(\mathcal{T}\,;\,\mathcal{T}')\rightarrow (M\times H)^{-}(\mathcal{T}')$  is a cofibration. Consequently, $\mathbb{F}_{i\,|\,\mathcal{T}}^{-}(d_{1})=\Sigma_{|T|}[i^{-}]$ is a $\Sigma$-cofibration. Thus proves the proposition.
\end{proof}  

\newpage

\begin{pro}\label{B4}
There are the following identifications:
$$
Y^{(1)}_{k}[l]\cap \partial Y_{k}[l]=colim_{\mathcal{D}_{k}^{l-2}[l]} \mathbb{F}^{-}_{l-2}
\hspace{15pt}\text{and}\hspace{15pt}
Y_{k}^{(1)}[l]=colim_{\mathcal{D}_{k}^{l-2}[l]} \mathbb{F}_{l-2}.
$$
\end{pro}

\begin{proof}
The choice of a representative element for each class $\llbracket T\,;\,\{T_{v}\}\rrbracket$ and the restriction on the space of labels $M(\llbracket T\,;\,\{T_{v}\}\rrbracket)$ for the construction of the colimit $\mathbb{F}_{i}^{-}$ and $\mathbb{F}_{i}$ are equivalent to the symmetric group axiom $(ii)$ of Construction \ref{C0}.  Similarly, the equivalence relation induced by the colimits is equivalent to the axiom $(iii)$ of Construction \ref{C0}. Finally, the reader can check that the inclusions from the colimits to the sequences induce homeomorphisms. 
\end{proof}

\begin{proof}[Proof of Theorem \ref{B5}] 
As seen in Section \ref{H7}, the map $\xi_{k}[l]$ is a weak equivalence if the map of sequences from $Y^{(1)}_{k}[l]\cap \partial Y_{k}[l]$ to $Y_{k}^{(1)}[l]$ is a homotopy equivalence. Since all the objects in the categories considered are fibrant, it is sufficient to prove that the map of sequences is a weak equivalence between cofibrant sequences. In Proposition \ref{B4}, we express these two sequences in terms of colimits using the Reedy category $\mathcal{D}_{k}^{l-2}[l]$:
\begin{equation}\label{H9}
Y^{(1)}_{k}[l]\cap \partial Y_{k}[l]=colim_{\mathcal{D}_{k}^{l-2}[l]} \mathbb{F}^{-}_{l-2}\longrightarrow colim_{\mathcal{D}_{k}^{l-2}[l]} \mathbb{F}_{l-2}=Y_{k}^{(1)}[l].
\end{equation}
More precisely, the above map arises from the natural transformation $t_{l-2}:\mathbb{F}_{l-2}^{-}\Rightarrow \mathbb{F}_{l-2}$. Nevertheless, we prove in Proposition \ref{B6} that the map of sequences induced by the natural transformation $t_{l-2}$ is a weak equivalence between cofibrant sequences. Thus proves that $\xi_{k}[l]$ is a weak equivalence.
\end{proof}

\begin{merci}
I would like to thank Muriel Livernet for her help in preparing this paper. I also wish to express my gratitude to Victor Turchin and Thomas Willwacher  for many helpful comments. I am also grateful to the Max Planck Institute and the ETH Zurich for their hospitality and their financial support.
\end{merci}

\bibliographystyle{amsplain}
\bibliography{bibliography}

\providecommand{\bysame}{\leavevmode\hbox to3em{\hrulefill}\thinspace}
\providecommand{\MR}{\relax\ifhmode\unskip\space\fi MR }
\providecommand{\MRhref}[2]{%
  \href{http://www.ams.org/mathscinet-getitem?mr=#1}{#2}
}
\providecommand{\href}[2]{#2}
\begin{thebibliography}{10}

\bibitem{Arone14}
G.~Arone and V.~Turchin, \emph{On the rational homology of high dimensional
  analogues of spaces of long knots}, Geometry and Topology \textbf{18} (2014),
  1261–1322.

\bibitem{Berger03}
C.~Berger and I.~Moerdijk, \emph{Axiomatic homotopy theory for operads},
  Comment. Math. Helv. \textbf{78} (2003), 805--831.

\bibitem{Berger06}
\bysame, \emph{The {B}oardman-{V}ogt resolution of operads in monoidal model
  categories}, Topology \textbf{45} (2006), no.~5, 807--849.

\bibitem{Boardman68}
J.~M. Boardman and R.~M. Vogt, \emph{Homotopy-everything {$H$}-spaces}, Bull.
  Amer. Math. Soc. \textbf{74} (1968), 1117--1122.

\bibitem{Boardman73}
J.M. {Boardman} and R.M. {Vogt}, \emph{Homotopy invariant algebraic structures
  on topological spaces}, Lecture Notes in Mathematics, Vol. 347,
  Springer-Verlag, Berlin, 1973.

\bibitem{Weiss15}
P.~Boavida~de Brito and M.~Weiss, \emph{Spaces of smooth embeddings and
  configuration categories}, arXiv:1502.01640 (2015).

\bibitem{Budney07}
R.~Budney, \emph{Little cubes and long knots}, Topology \textbf{46} (2007),
  no.~1, 1--27. \MR{2288724 (2008c:55015)}

\bibitem{Chacholski02}
W.~Chach\'olski and J.~Scherer, \emph{Homotopy theory of diagrams}, American
  Mathematical Society \textbf{155} (2002).

\bibitem{Turchin14}
N.~{Dobrinskaya} and V.~{Turchin}, \emph{{Homology of non-k-overlapping
  discs}}, Homology, Homotopy, and Applications \textbf{17} (2015).

\bibitem{Ducoulombier16}
J.~Ducoulombier, \emph{From map between coloured operads to swiss-cheese
  algebras}, arXiv:1603.07162.

\bibitem{Ducoulombier14}
\bysame, \emph{Swiss-cheese action on the totalization of action-operads},
  Algebraic and Geometric topology (2016).

\bibitem{Ducoulombier16.2}
J.~Ducoulombier and V.~Turchin, \emph{Delooping of manifold calculus functors
  towers on a closed disc}, To appear.

\bibitem{Dwyer12}
W.~{Dwyer} and K.~{Hess}, \emph{{Long knots and maps between operads}},
  Geometry and topology \textbf{16} (2012), 919--955.

\bibitem{Fresse09}
B.~Fresse, \emph{Modules over operads and functors}, Lecture Notes in
  Mathematics, Vol. 1967, Springer-Verlag, Berlin, 2009.

\bibitem{Weiss99.2}
T.~G. Goodwillie and M.~Weiss, \emph{Embeddings from the point of view of
  immersion theory. {II}}, Geom. Topol. \textbf{3} (1999), 103--118.
  \MR{1694808 (2000c:57055b)}

\bibitem{Hirschhorn03}
P.~S. Hirschhorn, \emph{Model categories and their localizations}, Mathematical
  Surveys and Monographs, vol.~99, American Mathematical Society, Providence,
  RI, 2003.

\bibitem{Hirschhorn15}
P.~S. {Hirschhorn} and I.~{Volic}, \emph{{Functors between Reedy model
  categories of diagrams}}, ArXiv e-prints (2015).

\bibitem{Hovey99}
M.~Hovey, \emph{Model categories}, Mathematical Surveys and Monographs,
  vol.~63, American Mathematical Society, Providence, RI, 1999.

\bibitem{May72}
J.~P. May, \emph{The geometry of iterated loop spaces}, Lecture Notes in
  Mathematics, Vol. 271, Springer-Verlag, Berlin-New York, 1972.

\bibitem{May74}
\bysame, \emph{{$E_{\infty }$} spaces, group completions, and permutative
  categories}, New developments in topology ({P}roc. {S}ympos. {A}lgebraic
  {T}opology, {O}xford, 1972), Cambridge Univ. Press, London, 1974, pp.~61--93.
  London Math. Soc. Lecture Note Ser., No. 11. \MR{0339152 (49 \#3915)}

\bibitem{McClure04.2}
James~E. McClure and Jeffrey~H. Smith, \emph{Cosimplicial objects and little
  {$n$}-cubes. {I}}, Amer. J. Math. \textbf{126} (2004), no.~5, 1109--1153.
  \MR{2089084 (2005g:55011)}

\bibitem{McClure04}
\bysame, \emph{Operads and cosimplicial objects: an introduction}, Axiomatic,
  enriched and motivic homotopy theory, NATO Sci. Ser. II Math. Phys. Chem.,
  vol. 131, Kluwer Acad. Publ., Dordrecht, 2004, pp.~133--171.

\bibitem{Stasheff63}
J.D. Stasheff, \emph{Homotopy associativity of h-spaces}, I. Trans. Amer. Math.
  Soc. \textbf{108} (1963), 275--292.

\bibitem{Turchin13}
Victor Turchin, \emph{Context-free manifold functor calculus and the
  fulton-macpherson operad}, Algebraic and Geometric Topology \textbf{13}
  (2013), no.~3, 1243--1271.

\bibitem{Tourtchine10}
\bysame, \emph{Delooping totalization of a multiplicative operad}, J. Homotopy
  Relat. Struct. \textbf{9} (2014), no.~2, 349--418. \MR{3258687}

\bibitem{Vogt03}
R.~M. Vogt, \emph{Cofibrant operads and universal {$E_\infty$}-operads},
  Topology Appl. \textbf{133} (2003), no.~1, 69--87.

\bibitem{Weiss99}
M.~Weiss, \emph{Embeddings from the point of view of immersion theory. {I}},
  Geom. Topol. \textbf{3} (1999), 67--101. \MR{1694812 (2000c:57055a)}

\end{thebibliography}

\end{document}